\documentclass[11pt]{amsart}  

\usepackage{amsmath}
\usepackage{amsthm}
\usepackage{amsfonts}
\usepackage{amssymb}
\usepackage{eucal}
\usepackage{yfonts}
\usepackage[all]{xy}
\usepackage{amsxtra}
\usepackage{appendix} 
\usepackage{tensor} 
\usepackage{url}
\usepackage{upgreek}
\usepackage{etoolbox}
\usepackage{mathtools}

\apptocmd{\sloppy}{\hbadness 10000\relax}{}{}

\usepackage[pdftex, breaklinks, linktocpage=true, bookmarksopen=true,bookmarksopenlevel=0,bookmarksnumbered=true]{hyperref}

\RequirePackage{color}
\definecolor{bwgreen}{rgb}{0.183,1,0.5}
\definecolor{bwmagenta}{rgb}{0.7,0.0,0.1}
\definecolor{bwblue}{rgb}{0.317,0.161,1}

\hypersetup{
pdfauthor = {\textcopyright\ Bryden Cais},
pdfcreator = {\LaTeX\ with package \flqq hyperref\frqq},
colorlinks = {true},
linkcolor={bwmagenta},	 
anchorcolor={},	 
citecolor={bwblue},	
filecolor={},	
menucolor={},	
runcolor={},	
urlcolor={bwgreen}, 
}

\CompileMatrices
\entrymodifiers={+!!<0pt,\fontdimen22\textfont2>}

\DeclareFontFamily{OT1}{rsfs}{}
\DeclareFontShape{OT1}{rsfs}{n}{it}{<-> rsfs10}{}
\DeclareMathAlphabet{\mathscr}{OT1}{rsfs}{n}{it}

\DeclareFontFamily{OT1}{pzc}{}
\DeclareFontShape{OT1}{pzc}{n}{it}{<->s*[2.2]pzc}{}
\DeclareMathAlphabet{\mathpzc}{OT1}{pzc}{b}{sl}

\makeatletter

\newcommand{\Rmnum}[1]{\expandafter\@slowromancap\romannumeral #1@}
\makeatother

\DeclareMathOperator{\id}{id}

\DeclareMathOperator{\Frac}{Frac}
\DeclareMathOperator{\ord}{ord} 
\DeclareMathOperator{\nil}{nil}

\newcommand*{\pr}{\rho}
\newcommand*{\ps}{\sigma}

\DeclareMathOperator{\Hom}{Hom}
\DeclareMathOperator{\End}{End}
\DeclareMathOperator{\Ext}{Ext}
\DeclareMathOperator{\Gal}{Gal}
\DeclareMathOperator{\GL}{GL}

\DeclareMathOperator{\Aut}{Aut}
\DeclareMathOperator{\Spec}{Spec}

\DeclareMathOperator{\et}{\acute{e}t}

\DeclareMathOperator{\dR}{dR}

\DeclareMathOperator{\tr}{tr}
\DeclareMathOperator{\Tr}{Tr}

\DeclareMathOperator{\Pic}{Pic}
\DeclareMathOperator{\Alb}{Alb}
\DeclareMathOperator{\sm}{sm}

\DeclareMathOperator{\Lie}{Lie}

\DeclareMathOperator{\red}{red}

\DeclareMathOperator{\Rep}{Rep}

\DeclareMathOperator{\res}{res}

\DeclareMathOperator{\Ig}{Ig}

\DeclareMathOperator{\im}{im}

\DeclareMathOperator{\Irr}{Irr}
\DeclareMathOperator{\bal}{bal.}
\DeclareMathOperator{\Ell}{Ell}

\DeclareMathOperator{\univ}{univ}

\DeclareMathOperator{\proj}{proj}

\DeclareMathOperator{\Sen}{Sen}

\newcommand*{\R}{\ensuremath{\mathbf{R}}}   
\newcommand*{\Z}{\ensuremath{\mathbf{Z}}}               
\newcommand*{\Q}{\ensuremath{\mathbf{Q}}}                           
             
\newcommand*{\Qbar}{\overline{\Q}}

\renewcommand*{\P}{\ensuremath{\mathbf{P}}}           		
\newcommand*{\Aff}{\ensuremath{\mathbf{A}}}

\newcommand*{\s}{\mathfrak{S}}

\newcommand*{\scrA}{\mathscr{A}}

\newcommand*{\C}{\mathbf{C}}
     
\newcommand*{\F}{\mathbf{F}}
\newcommand*{\scrF}{\mathscr{F}}

\newcommand*{\scrG}{\mathscr{G}}  
\newcommand*{\scrH}{\mathscr{H}}                           
\newcommand*{\h}{\mathscr{H}}                               
\newcommand*{\I}{\mathscr{I}}                               

\renewcommand*{\L}{\mathscr{L}}
\newcommand*{\scrM}{\mathscr{M}}

\renewcommand*{\O}{\mathscr{O}}                    
 
\newcommand*{\X}{\mathcal{X}}     
\newcommand*{\Y}{\mathcal{Y}}

\newcommand*{\scrP}{\mathscr{P}}  
\newcommand*{\scrQ}{\mathscr{Q}}                             

\newcommand*{\scrHom}{\mathscr{H}\mathit{om}}

\newcommand*{\D}{\ensuremath{\mathbf{D}}}
\newcommand*{\M}{\ensuremath{\mathbf{M}}}
              
\renewcommand*{\H}{\ensuremath{\mathfrak{H}}}

\newcommand*{\Dual}[1]{{{#1}^t}}
\newcommand*{\VDual}[1]{{{#1}^{\vee}}}

\renewcommand*{\int}{\ensuremath{\mathrm{int}}}

\renewcommand*{\SS}{\ensuremath{\underline{\mathrm{ss}}}}
\renewcommand*{\u}[1]{\underline{#1}}
\renewcommand*{\o}[1]{\overline{#1}}
\renewcommand*{\c}[1]{{#1}^{\mathrm{c}}}
\newcommand*{\wh}[1]{\widehat{#1}}
\newcommand*{\wt}[1]{\widetilde{#1}}
\newcommand*{\nor}[1]{{#1}^{\mathrm{n}}}

\newcommand*{\tens}{\mathop{\otimes}\limits}
\newcommand*{\can}{\text{-}\mathrm{can}}

\theoremstyle{plain}
  \newtheorem{theorem}{Theorem}
  \newtheorem{proposition}[theorem]{Proposition}
  \newtheorem{lemma}[theorem]{Lemma}
  \newtheorem{corollary}[theorem]{Corollary}

\theoremstyle{definition}
  \newtheorem{definition}[theorem]{Definition}

\theoremstyle{remark}
  
  \newtheorem{remark}[theorem]{Remark}

\include{header}

\numberwithin{theorem}{subsection}  
\numberwithin{equation}{subsection}

\begin{document}
\title{The Geometry of Hida Families \Rmnum{1}: $\Lambda$-adic de Rham cohomology}

\author{Bryden Cais}
\address{University of Arizona, Tucson}
\curraddr{Department of Mathematics, 617 N. Santa Rita Ave., Tucson AZ. 85721}
\email{cais@math.arizona.edu}

\thanks{
	During the writing of this paper, the author was partially supported by an NSA Young Investigator grant
	(H98230-12-1-0238) and an NSF RTG (DMS-0838218).
	}

\dedicatory{To Masami Ohta}

\date{June 6, 2016}

\begin{abstract}
	We construct the $\Lambda$-adic de Rham analogue of Hida's
	ordinary $\Lambda$-adic \'etale cohomology and of Ohta's $\Lambda$-adic Hodge cohomology,
	and by exploiting the geometry 
	of integral models of modular curves over the cyclotomic extension
	of $\Q_p$, we give a purely geometric proof of the expected finiteness,
	control, and $\Lambda$-adic duality theorems.  Following Ohta, we then prove that
	our $\Lambda$-adic module of differentials is canonically isomorphic to  
	the space of ordinary $\Lambda$-adic cuspforms.
	In the sequel \cite{CaisHida2} to this paper, 
	we construct the crystalline counterpart to Hida's ordinary $\Lambda$-adic 
	\'etale cohomology, and 
	employ integral $p$-adic Hodge theory
	to prove $\Lambda$-adic
	comparison isomorphisms between all of these cohomologies.
	As applications of our work in this paper and \cite{CaisHida2}, we 
	will be able to provide a ``cohomological" construction of the 
	family of $(\varphi,\Gamma)$-modules attached to Hida's ordinary $\Lambda$-adic \'etale
	cohomology by \cite{Dee}, as well as a new and purely geometric 
	proof of Hida's finitenes and control theorems.  We are also able to prove
	refinements of the main theorems in \cite{MW-Hida} and \cite{OhtaEichler}.
\end{abstract}

\maketitle

\section{Introduction}\label{intro}

\subsection{Motivation}

In his landmark papers \cite{HidaGalois} and \cite{HidaIwasawa}, Hida proved that the $p$-adic Galois representations 
attached to ordinary cuspidal Hecke eigenforms by Deligne (\cite{DeligneFormes}, \cite{CarayolReps})
interpolate $p$-adic analytically in the weight variable to a family of $p$-adic representations
whose specializations to integer weights $k\ge 2$ recover
the ``classical" Galois representations attached to weight $k$ cuspidal eigenforms.
Hida's work paved the way for a revolution---
from the pioneering work of Mazur
on Galois deformations to Coleman's construction of $p$-adic families of finite slope overconvergent 
modular forms---and began a trajectory of thought whose fruits include some of the most spectacular
achievements in modern number theory.

Hida's proof is constructive and has at its heart the \'etale cohomology of the tower
of modular curves $\{X_1(Np^r)\}_{r}$ over $\Q$.  More precisely,  
Hida considers the projective limit $H^1_{\et}:=\varprojlim_r H^1_{\et}(X_1(Np^r)_{\Qbar},\Z_p)$
(taken with respect to the trace mappings), which is naturally a module for the 
``big" $p$-adic Hecke algebra $\H^*:=\varprojlim_r \H_r^*$, which is itself an algebra
over the completed group ring $\Lambda:=\Z_p[\![1+p\Z_p]\!]\simeq \Z_p[\![T]\!]$
via the diamond operators.  
Using the idempotent $e^*\in \H^*$ attached to the (adjoint) Atkin operator $U_p^*$ 
to project to the part of $H^1_{\et}$ where $U_p^*$ acts invertibly,
Hida proves in \cite[Theorem 3.1]{HidaGalois} 
(via the comparison isomorphism between \'etale and topological cohomology and explicit 
calculations in group cohomology) that 
$e^* H^1_{\et}$ is finite and free as a module over $\Lambda$, and that the resulting Galois representation
\begin{equation*}
	\xymatrix{
		{\rho: \scrG_{\Q}} \ar[r] & {\Aut_{\Lambda}(e^*H^1_{\et}) \simeq \GL_m(\Z_p[\![T]\!])} 
		}
\end{equation*}
$p$-adically interpolates the representations attached to ordinary cuspidal eigenforms.

By analyzing the geometry of the tower of modular curves, Mazur and Wiles \cite{MW-Hida}
were able to relate the inertial invariants of the local (at $p$) representation $\rho_p$ to the 
\'etale cohomology of the Igusa tower studied in \cite{MW-Analogies}, and in so doing 
proved\footnote{Mazur and Wiles treat only the case of tame level $N=1$.} 
that the ordinary filtration of the Galois representations attached
to ordinary cuspidal eigenforms can be $p$-adically interpolated:
both the inertial invariants and covariants 
are free of the same finite rank over $\Lambda$ and specialize to the corresponding subquotients
in integral weights $k\ge 2$.  
As an application, they provided examples of cuspforms $f$ and primes $p$
for which the specialization of the associated Hida family of Galois representations 
to weight $k=1$ is not Hodge--Tate,
and so does not arise from a weight one cuspform via the construction of Deligne-Serre
\cite{DeligneSerre}.
Shortly thereafter, Tilouine \cite{Tilouine} clarified the geometric underpinnings
of \cite{HidaGalois} and \cite{MW-Hida}, and removed most of the restrictions on the $p$-component of the
nebentypus of $f$.  Central to both \cite{MW-Hida} and \cite{Tilouine} is a careful study of
the tower of $p$-divisible groups attached to the ``good quotient" modular abelian varieties
introduced in \cite{MW-Iwasawa}.

With the advent of integral $p$-adic Hodge theory, 
and in view of the prominent role
it has played in furthering the trajectory initiated by Hida's work, 
it is natural to ask 
if one can construct Hodge--Tate, de~Rham and crystalline analogues of $e^*H^1_{\et}$,
and if so, to what extent the integral comparison isomorphsms of $p$-adic Hodge theory
can be made to work in $\Lambda$-adic families.
In \cite{OhtaEichler},
Ohta has addressed this question in the case of Hodge cohomology. 
Using the invariant differentials on the tower of $p$-divisible groups studied in \cite{MW-Hida} and \cite{Tilouine},
Ohta constructs a $\Lambda \wh{\otimes}_{\Z_p} \Z_p[\mu_{p^{\infty}}]$-module
from which, via an integral version of the Hodge--Tate comparison isomorphism \cite{Tate}
for ordinary $p$-divisible groups, he is able to recover the semisimplification
of the ``semilinear representation"
$\rho_{p}\wh{\otimes} \O_{\C_p}$, where 
$\C_p$ is, as usual, the $p$-adic completion of an algebraic closure of $\Q_p$.
Using Hida's results, Ohta proves that his Hodge cohomology analogue of
$e^*H^1_{\et}$ is free of finite rank over $\Lambda\wh{\otimes}_{\Z_p} \Z_p[\mu_{p^{\infty}}]$
and specializes to finite level exactly as one expects.  As applications
of his theory, Ohta provides a construction of two-variable $p$-adic $L$-functions
attached to families of ordinary cuspforms differing from that of Kitagawa \cite{Kitagawa},
and, in a subsequent paper \cite{Ohta2}, 
provides a new and streamlined proof of the theorem of Mazur--Wiles \cite{MW-Iwasawa}
(Iwasawa's Main Conjecture for $\Q$; see also \cite{WilesTotallyReal}).
We remark that Ohta's $\Lambda$-adic Hodge-Tate isomorphism is a crucial ingredient
in the forthcoming partial proof of Sharifi's conjectures \cite{SharifiConj}, \cite{SharifiEisenstein} 
due to Fukaya and Kato \cite{FukayaKato}.

\subsection{Results}\label{resultsintro}

In this paper, we construct the de Rham
counterpart to Hida's ordinary $\Lambda$-adic \'etale cohomology and Ohta's 
$\Lambda$-adic Hodge cohomology, and we prove the expected control 
and finiteness theorems via a purely geometric argument
involving a careful study of the geometry of
certain Katz--Mazur integral models of modular curves and a classical
result of Nakajima \cite{Nakajima}.

In the sequel \cite{CaisHida2} to this paper, we will 
use crystalline Dieudonn\'e theory to provide a cohomoligcal construction of the $\Lambda$-adic family 
of $(\varphi,\Gamma)$-modules attached to $e^*H^1_{\et}$
by Dee \cite{Dee}, and will establish
a suitable $\Lambda$-adic 
version of every integral comparison isomorphism one could hope for.
In particular, we will be able to recover the entire family
of $p$-adic Galois representations $\rho_{p}$ (and not just its
semisimplification) from our $\Lambda$-adic crystalline cohomology.
As an application of our theory and the results of this paper, we will 
in \cite{CaisHida2} give a new 
and purely geometric proof of Hida's freeness and control theorems
for $e^*H^1_{\et}$.  
As our methods are geometric,
we expect that they can be adapted to the setting of other Shimura curves
once one has a sufficiently sophisticated theory of integral models.

In order to survey our main results more precisely, we introduce
some notation.  Throughout, we fix a prime $p>2$ and a positive integer
$N>3$.
Fix an algebraic closure $\Qbar_p$
of $\Q_p$ 
as well as a $p$-power compatible sequence 
$\{\varepsilon^{(r)}\}_{r\ge 0}$ of primitive $p^r$-th roots of unity in $\Qbar_p$.
We set $K_r:=\Q_p(\mu_{p^r})$ and $K_r':=K_r(\mu_N)$, and we
write $R_r$ and $R_r'$ for the rings of integers
in $K_r$ and $K_r'$, respectively.  
Denote by $\scrG_{\Q_p}:=\Gal(\Qbar_p/\Q_p)$ the absolute Galois group
and by $\scrH$
the kernel of the $p$-adic cyclotomic character
$\chi: \scrG_{\Q_p}\rightarrow \Z_p^{\times}$.
We write $\Gamma:=\scrG_{\Q_p}/\scrH \simeq \Gal(K_{\infty}/K_0)$ for the quotient and,
using that $K_0'/\Q_p$ is unramified, we canonically identify $\Gamma$
with $\Gal(K_{\infty}'/K_0')$.
We put $\Delta:=\Z_p^{\times}$, and for $r\ge 1$ set $\Delta_r:=1+p^r\Z_p$.
For any ring $A$ define
$\Lambda_{A}:=\varprojlim_r A[\Delta_1/\Delta_r]$ to be the completed group ring
on $\Delta_1$ over $A$; if $\varphi$ is an endomorphism of $A$, we again write $\varphi$
for the induced endomorphism of $\Lambda_A$ that acts as the identity on $\Delta_1$.
We will denote by $\langle u\rangle^*$ 
the adjoint diamond operator 
attached to $u\in \Delta$, 
and give $\H_r^*$ the structure of a $\Lambda$-module via the map 
$\langle \cdot\rangle^*:\Delta_1\hookrightarrow \H^*$.
Finally, we denote by $X_r:=X_1(Np^r)$ the 
canonical model over $\Q$ with rational cusp at $i\infty$
of the modular curve arising as the quotient of the extended upper-halfplane
by the congruence subgroup $\Upgamma_1(Np^r)$, and we write
$J_r:=J_1(Np^r)$ for its Jacobian.

The goal of this paper is to construct a de Rham analogue of Hida's $e^*H^1_{\et}$.
A na\"ive idea would be to mimic Hida's construction, using the
(relative) de Rham cohomology of $\Z_p$-integral models of the modular curves $X_r$
in place of $p$-adic \'etale cohomology.  However, this approach fails due
to the fact that $X_r$ has bad reduction
at $p$, so the relative de Rham
cohomology of integral models does not provide good
$\Z_p$-lattices in the de Rham cohomology of $X_r$ over $\Q_p$.
To address this problem, we use the canonical integral structures in de Rham cohomology
studied in \cite{CaisDualizing} and the canonical integral model $\X_r$ of $X_r$
over $R_r$ associated to the moduli problem 
$([\bal\ \Upgamma_1(p^r)]^{\varepsilon^{(r)}\can};\ [\mu_N])$
\cite{KM} to construct well-behaved integral ``de Rham cohomology" for the tower of modular
curves.  For each $r$, we obtain a short exact sequence of free $R_r$-modules
with semilinear $\Gamma$-action and commuting $\H_r^*$-action
\begin{equation}
	\xymatrix{
		0\ar[r] & {H^0(\X_r, \omega_{\X_r/R_r})} \ar[r] & {H^1(\X_r/R_r)} \ar[r] & {H^1(\X_r,\O_{\X_r})}
		\ar[r] & 0 
	}\label{finiteleveldRseq}
\end{equation}
which is functorial in finite $K_r$-morphisms of the generic fiber $X_r$,
and whose scalar extension to $K_r$ recovers the Hodge filtration of $H^1_{\dR}(X_r/K_r)$.
Extending scalars to $R_{\infty}$ and taking projective limits, 
we obtain a short exact sequence of $\Lambda_{R_{\infty}}$-modules with semilinear $\Gamma$-action
and commuting linear $\H^*$-action 
\begin{equation}
	\xymatrix{
		0\ar[r] & {H^0(\omega)} \ar[r] & {H^1_{\dR}} \ar[r] & {H^1(\O)} 
	}.\label{dRseq}
\end{equation}
Our main result (see Theorem \ref{main}) is that the ordinary part of (\ref{dRseq}) 
is the correct de Rham analogue of Hida's ordinary $\Lambda$-adic \'etale cohomology:

\begin{theorem}\label{MTA}
	There is a canonical short exact sequence of finite free $\Lambda_{R_{\infty}}$-modules
	with semilinear $\Gamma$-action and commuting linear $\H^*$-action
	\begin{equation}
		\xymatrix{
		0\ar[r] & {e^*H^0(\omega)} \ar[r] & {e^*H^1_{\dR}} \ar[r] & {e^*H^1(\O)} \ar[r] & 0
	}.\label{orddRseq}
	\end{equation}
	As a $\Lambda_{R_{\infty}}$-module, $e^*H^1_{\dR}$ is free of rank $2d$, while each of the flanking terms 
	in $(\ref{orddRseq})$ is free of rank $d$, for $d=\sum_{k=3}^{p+1}\dim_{\F_p} S_k(\Upgamma_1(N);\F_p)^{\ord}$.
	Applying $\otimes_{\Lambda_{R_{\infty}}} R_{\infty}[\Delta_1/\Delta_r]$ to $(\ref{orddRseq})$
	recovers the ordinary part of the scalar extension of $(\ref{finiteleveldRseq})$ to $R_{\infty}$.
\end{theorem}

We then show that the $\Lambda_{R_{\infty}}$-adic Hodge filtration (\ref{orddRseq}) is very nearly ``auto dual".
To state our duality result more succintly,
for any ring homomorphism $A\rightarrow B$,
we will write $(\cdot)_B:=(\cdot)\otimes_A B$ and $(\cdot)_B^{\vee}:=\Hom_B((\cdot)\otimes_A B , B)$ for these
functors from $A$-modules to $B$-modules.  If $G$ is any group of automorphisms of $A$ and $M$
is an $A$-module with a semilinear action of $G$, for any ``crossed" 
homomorphism\footnote{That is, $\psi(\sigma\tau) = \psi(\sigma)\cdot\sigma\psi(\tau)$ for all $\sigma,\tau\in\Gamma$,} 
$\psi:G\rightarrow A^{\times}$ we will write $M(\psi)$ for the 
$A$-module $M$ with ``twisted" semilinear $G$-action given by $g\cdot m:=\psi(g)g m$.
Our duality theorem is (see Proposition \ref{dRDuality}):

\begin{theorem}\label{dRdualityThm}
	The natural cup-product auto-duality of 
	$(\ref{finiteleveldRseq})$ over $R_r':=R_r[\mu_N]$ induces a canonical
	$\Lambda_{R_{\infty}'}$-linear and $\H^*$-equivariant isomorphism 
	of exact sequences
	\begin{equation*}
		\xymatrix{
				0\ar[r] & {e^*H^0(\omega)(\langle\chi\rangle\langle a\rangle_N)_{\Lambda_{R_{\infty}'}}}
				\ar[r]\ar[d]^-{\simeq} & 
				{e^*H^1_{\dR}(\langle\chi\rangle\langle a\rangle_N)_{\Lambda_{R_{\infty}'}}} 
				\ar[r]\ar[d]^-{\simeq} & 
				{e^*H^1(\O)(\langle\chi\rangle\langle a\rangle_N)_{\Lambda_{R_{\infty}'}}} 
				\ar[r]\ar[d]^-{\simeq} & 0\\
	0\ar[r] & {(e^*H^1(\O))^{\vee}_{\Lambda_{R_{\infty}'}}} \ar[r] & 
	{({e^*H^1_{\dR}})^{\vee}_{\Lambda_{R_{\infty}'}}} \ar[r] & 
	{(e^*H^0(\omega))^{\vee}_{\Lambda_{R_{\infty}'}}} \ar[r] & 0
		}
	\end{equation*}
	that is compatible with the natural action of $\Gamma \times \Gal(K_0'/K_0)\simeq \Gal(K_{\infty}'/K_0)$ 
	on the bottom row and the twist of the natural action on the top row by the $\H^*$-valued character 
	$\langle \chi\rangle \langle a\rangle_N$,
	where 
	$a(\gamma) \in (\Z/N\Z)^{\times}$ is
	determined for $\gamma\in \Gal(K_0'/K_0)$ by 
	$\zeta^{a(\gamma)}=\gamma\zeta$ for every $N$-th root of unity $\zeta$.
\end{theorem}

We moreover prove that, as one would expect, the $\Lambda_{R_{\infty}}$-module $e^*H^0(\omega)$ 
is canonically isomorphic to the  module $eS(N,\Lambda_{R_{\infty}})$ of ordinary $\Lambda_{R_{\infty}}$-adic
cusp forms of tame level $N$; see Corollary \ref{LambdaFormsRelation}.

\subsection{Overview of the article}\label{Overview}

We now describe the contents and structure of the article in more detail.
Appendix \ref{GD} is devoted to reviewing and extending the theory of \cite{CaisDualizing},
which provides the ``good integral structures" in the de~Rham cohomology of
curves that plays a central role in our constructions; in particular, the existence and 
properties of the short exact sequence (\ref{finiteleveldRseq}) are discussed
in the generality that we shall need them.
As this discussion
will be applied to certain integral (Katz--Mazur) models of modular curves (like $\X_r$), 
we review the construction and relevant features
of these models, as well as correspondences on them,
in Appendix \ref{tower}.  The reader who is content to accept this foundational material
on faith can safely skip the discussion in Appendices \ref{GD}--\ref{tower}, and refer to it on a ``need to know" basis.  As we explain below, a key point (which is
the reason that we must work over $R_r$) is that the regular and proper relative curve $\X_r$
over $R_r$ has {\em reduced} special fiber.

Writing $H(\X_r/R_r)$ for this short exact sequence (\ref{finiteleveldRseq}),
the natural degeneracy mappings on modular curves $X_{r+1}\rightarrow X_r$
induce trace mappings 
$$\rho_{*}:e_{r+1}^*H(\X_{r+1}/R_{r+1})\rightarrow e_r^*H(\X_r/R_r)\otimes_{R_r}{R_{r+1}}$$
on ordinary parts, and by definition the sequence (\ref{orddRseq}) is obtained from these mappings by
passing to projective limits.  
In order to prove theorems \ref{MTA} and \ref{dRdualityThm},
we will reduce to working at finite level and in
characteristic $p$ via 
a general commutative algebra formalism for dealing with 
such projective limits of ``towers" of cohomology
that we use throughout this paper and its sequel \cite{CaisHida2}.  
This formalism, which is the the subject of \S\ref{TowerFormalism},  
reduces the proof of Theorem \ref{MTA} to the following
two assertions:
\begin{enumerate}
	\item The terms in the short exact sequence $e_r^*H(\X_r/R_r)\otimes_{R_r} \F_p$
	are free $\F_p[\Delta_1/\Delta_r]$-modules of ranks $d$, $2d$, and $d$, respectively.
	\label{keypfA}
	
	\item For all $r$, the induced maps 
	$\xymatrix@1{{\rho_*\otimes 1 : e_{r+1}^*H(\X_{r+1}/R_{r+1})\otimes_{R_{r+1}}\F_p} \ar[r] & {e_r^*H(\X_r/R_r)\otimes_{R_r} \F_p}}$
	are surjective.
	\label{keypfB}
\end{enumerate}
One of the miraculous properties of the cohomology sequence (\ref{finiteleveldRseq}) is that
it is compatible with base change, in the sense that $H(\X_r/R_r)\otimes_{R_r}\F_p$
may be functorially identified with the hypercohomology $H^{\bullet}(\o{\X}_r/\F_p)$
of the two term complex $d:\O_{\o{\X}_r}\rightarrow \omega_{\o{\X}_r/\F_p}$
on 
the special fiber $\o{\X}_r:=\X_r\times_{R_r} \F_p$, where $\omega_{\o{\X}_r/\F_p}$
is the relative dualizing sheaf; see Lemma \ref{ReductionCompatibilities}.
Thus, we obtain a functorial identification of
$e_r^*H(\X_r/R_r)\otimes_{R_r} \F_p$ 
with the short exact sequence
\begin{equation}
	\xymatrix{
		0\ar[r] & {e_r^*H^0(\o{\X}_r,\omega_{\o{\X}_r/\F_p})} \ar[r] &
		{e_r^*H^1(\o{\X}_r/\F_p)}\ar[r] & {e_r^*H^1(\o{\X}_r,\O_{\o{\X}_r})} \ar[r] & 0
	}\label{CharpSeq}
\end{equation}
arising from the degeneration of the ``Hodge to de~Rham" spectral seqence 
that computes the hypercohomology $H^1(\o{\X}_r/\F_p)$.
Thus, in order to prove (\ref{keypfA}) and (\ref{keypfB}) above, we are reduced to a problem
about the cohomology (\ref{CharpSeq}) of the characteristc $p$ schemes $\o{\X}_r$.

Now the fact that $\o{\X}_r$ is {\em reduced} plays a critical role, as it allows us
to use Rosenlicht's explicit description of the relative dualizing sheaf for a reduced curve
to determine the structure of the exact sequence (\ref{CharpSeq}).
Rosenlicht's theory allows us to identify $H^0(\o{\X}_r,\omega_{\o{\X}_r/\F_p})$
with a certain subspace of meromorphic differentials on the normalization $\nor{\o{\X}}_r$
of $\o{\X}_r$.  This normalization is a disjoint union of smooth and proper ``Igusa" curves,
and has two ``privileged" irreducible components:
namely the component $I_r^{\infty}$ which meets the $\infty$-cuspidal section of
$\X_r$, and $I_r^0$, which meets the $0$-cuspidal section.  Using
the description of the correspondences $U_p$ and $U_p^*$ on $\nor{\o{\X}}_r$
due to Ulmer \cite{Ulmer}, and recorded in Proposition \ref{UlmerProp},
together with the fact that pullback by Frobenius
kills differential forms in characteristic $p$, we calculate that these correspondences
``contract" the space of meromorphic differentials on $\nor{\o{\X}}_r$
onto the components $I_r^{\star}$.  More precisely, we prove in \S\ref{OrdStruct}
that pullback of meromorphic
differentials along the canonical closed immersions $I_r^{\star}\hookrightarrow \nor{\o{\X}}_r$
induce natural isomorphisms
\begin{equation}
	e_r^*H^0(\o{\X}_r,\omega_{\o{\X}_r/\F_p}) \simeq H^0(I_r^{\infty},\Omega^1_{I_r^{\infty}}(\SS))^{V_{\ord}}
	\quad\text{and}\quad
	e_rH^0(\o{\X}_r,\omega_{\o{\X}_r/\F_p}) \simeq H^0(I_r^{0},\Omega^1_{I_r^0}(\SS))^{V_{\ord}},
	\label{UpRescues}
\end{equation} 
where the superscript of ``$V_{\ord}$" means the {\em $V$-ordinary} part of cohomology, by definition
the maximal subspace on which the Cartier operator $V$ acts invertibly. 
Applying Grothendieck duality---which swaps the idempotent $e_r$ with $e_r^*$ and the
Cartier operator with pullback by absolute Frobenius---to the second of these isomorphisms
yields functorial identifications 
\begin{equation}
	e_r^*H^1(\o{\X}_r,\O_{\o{\X}_r})\simeq 
	{e_rH^0(\o{\X}_r,\omega_{\o{\X}_r/\F_p}})^{\vee}\simeq
	\left(H^0(I_r^0,\Omega^1_{I_r^0}(\SS))^{V_{\ord}}\right)^{\vee}\simeq
	H^1(I_r^0,\O_{I_r^0}(-\SS))^{F_{\ord}}\label{GDualIden}
\end{equation}
where the superscript of ``$F_{\ord}$" denotes the maximal subspace on which pullback by absolute
Frobenius is invertible.  From the functoriality of (\ref{GDualIden}), we deduce that
pullback by absolute Frobenius is an {\em isomorphism} on $e_r^*H^1(\o{\X}_r,\O_{\o{\X}_r})$,
and it follows (recalling again that pullback by Frobenius kills differentials) that 
(\ref{CharpSeq}) is {\em functorially split} by Frobenius.  In this way, and via the isomorphisms
(\ref{UpRescues})--(\ref{GDualIden}), the structure of the cohomology exact sequence (\ref{CharpSeq})
is entirely captured by the space of $V$-ordinary meromorphic differential forms
on the Igusa towers $\{I_r^{\star}\}_{r}$ for $\star=0,\infty$.

Nakajima's beautiful {equivariant Deuring--Shafarevich formula} (Proposition \ref{Nakajima} below),
applied to the $\Delta_1/\Delta_r$-cover $I_r^{\star}\rightarrow I_1^{\star}$,
allows us to conclude that the right side of each isomorphism
in (\ref{UpRescues}) is {\em free} of rank $d$ as a module over the group ring $\F_p[\Delta_1/\Delta_r]$.
We recall the context and statement of this key result in \S\ref{CarterOp},
and apply it to 
the cohomology of the Igusa tower $I_r^{\star}$ in \S\ref{IgusaTower}.
It then follows from the identifications (\ref{UpRescues})--(\ref{GDualIden})
that the flanking terms of the exact sequence (\ref{CharpSeq}) are likewise free of rank 
$d$ over $\F_p[\Delta_1/\Delta_r]$,
and since---as observed in the discussion above---the cohomology sequence (\ref{CharpSeq})
is functorially split by Frobenius, the middle term $e_r^*H^1(\o{\X}_r/\F_p)$ of this sequence
is then free of rank $2d$ over $\F_p[\Delta_1/\Delta_r]$, which establishes the key claim (\ref{keypfA})
above.

To prove (\ref{keypfB}), and thus complete the proof of Theorem \ref{MTA},
we again use the established Frobenius splitting of (\ref{CharpSeq}) and the identification
of its flanking terms with the cohomology of the Igusa tower
provided by (\ref{UpRescues})--(\ref{GDualIden}), together with Grothendieck duality,
to reduce the asserted surjectivity of $\rho_*\otimes 1$
to the following claim: the canonical pullback maps
\begin{subequations}
\begin{equation}
	\xymatrix@1{
	{\rho^* : H^0(I_{r+1}^{0},\Omega^1_{I_{r+1}^0}(\SS))} \ar[r] & {H^0(I_r^{0},\Omega^1_{I_r^0}(\SS))}
	}\label{IgDiffPB}
\end{equation}
\begin{equation}
\xymatrix@1{
	{\rho^* : H^1(I_{r+1}^{\infty},\O_{I_{r+1}^{\infty}}(-\SS))} \ar[r] & {H^0(I_r^{\infty},\O_{I_r^{\infty}}(-\SS))}
	}\label{IgLinePB}
\end{equation}	
\end{subequations}
attached to the degeneracy mapping $\rho:I_{r+1}^{\star}\rightarrow I_r^{\star}$
are both {\em injective}.
The injectivity of (\ref{IgDiffPB}) is clear, as $\rho$ is generically \'etale, 
while we prove that (\ref{IgLinePB}) is injective in Proposition \ref{IgusaStructure} and Lemma \ref{MW}
by interpreting classes in $H^1$ as line bundles and using the fact that the degeneracy mapping $\rho$ {\em totally} ramifies over every supersingular point.  

With Theorem \ref{MTA} established,
we show as part of our commutative algebra formalism in Lemma \ref{LambdaDuality}
that the proof of Theorem \ref{dRdualityThm} may then be reduced to the existence of 
certain autoduality pairings on $e_r^*H^1(\X_r/R_r')$ that are compatible
in a precise sense with change in $r$ (see (\ref{pairingchangeinr}) for the exact condition).
Using the fact that the canonical $R_r'$-lattice $H^1(\X_r/R_r')$  in the de Rham cohomology 
of the generic fiber $X_r$ over $K_r'$ is preserved by the standard cup-product autoduality, 
we are able to ``twist" the restriction of this pairing to $e_r^*H^1(\X_r/R_r')$
by an approprite power of $U_p^*$ and the Atkin--Lehner involution
(which explains why we must work over $R_r'=R_r[\mu_N]$ rather than $R_r$ itself)
to obtain a perfect self-pairing that satisfies the required compatibility condition.

\subsection{Notation}\label{Notation}

If $\varphi:A\rightarrow B$ is any map of rings, we will often write $M_B:=M\otimes_{A} B$ 
for the $B$-module induced from an $A$-module $M$ by extension of scalars.
When we wish to specify $\varphi$, we will write $M\otimes_{A,\varphi} B$.
Likewise, if $\varphi:T'\rightarrow T$ is any morphism of schemes, for any $T$-scheme $X$
we denote by $X_{T'}$ the base change of $X$ along $\varphi$.
If $f:X\rightarrow Y$ is any morphism of $T$-schemes,
we will write $f_{T'}: X_{T'}\rightarrow Y_{T'}$
for the morphism of $T'$-schemes obtained from $f$ by base change along $\varphi$.
When $T=\Spec(R)$ and $T'=\Spec(R')$ are affine, we abuse notation and write
$X_{R'}$ or $X\times_{R} R'$ for $X_{T'}$.
We will frequently work with schemes over a discrete valuation ring $R$.
We will often write $\X,\Y,\ldots$ for schemes over $\Spec(R)$,
and will generally use $X,Y,\ldots$ (respectively $\o{\X},\o{\Y},\ldots$) 
for their generic (respectively special) fibers.

The following notation will be in effect throughout this article and its sequel \cite{CaisHida2}.
We always assume that $p > 2$ is a fixed prime and $N$ is a fixed positive integer
with $p\nmid N$ and $Np>4$.  We set $K_r:=\Q_p(\mu_{p^r})$ and $R_r:=\Z_p[\mu_{p^r}]$,
and put $K_r':=K_r(\mu_N)$ and $R_r':=R_r[\mu_N]$.  We choose, once and for all,
a compatible sequence $\{\varepsilon^{(r)}\}_{r\ge 0}$ of primitive $p^r$-th
roots of unity.
As in \S\ref{resultsintro}, for $r\ge 1$, we denote by $X_r:=X_1(Np^r)$ 
the canonical model over $\Q$ with rational cusp at $i\infty$
of the modular curve arising as the quotient of the extended upper half-plane
by the congruence subgroup $\Upgamma_1(Np^r)$.
There are two natural degeneracy mappings $\rho,\sigma:X_{r+1}\rightrightarrows X_r$
of curves over $\Q$ induced by the self-maps of the upper half-plane $\rho:\tau\mapsto \tau$ and 
$\sigma:\tau\mapsto p\tau$.
We write $J_r:=\Pic^0_{X_r/\Q}$ for the Jacobian of $X_r$ over $\Q$
and $\H_r(\Z)$
for the $\Z$-subalgebra of $\End_{\Q}(J_r)$ generated by the
Hecke operators $\{T_{\ell}\}_{\ell\nmid Np}$, $\{U_{\ell}\}_{\ell|Np}$
and the Diamond operators $\{\langle u\rangle\}_{u\in (\Z/p^r\Z)^{\times}}$
and 
$\{\langle v \rangle_N\}_{v\in (\Z/N\Z)^{\times}}$.
We define $\H_r(\Z)^{*}$ similarly, using instead the ``transpose"
Hecke and diamond operators, and set $\H_r:=\H_r(\Z)\otimes_{\Z}\Z_p$
and $\H_r^*:=\H_r(\Z)^*\otimes_{\Z}\Z_p$; see Definitions \ref{CorrDef} and \ref{Def:HeckeModuliProblem},
and the surrounding discussion.
As usual, $e_r\in \H_r$ and $e_r^*\in \H_r^*$ are the idempotents of these
semi-local $\Z_p$-algebras corresponding to the Atkin operators $U_p$ and $U_p^*$,
respectively. 
We put $e:=(e_r)_r$ and $e^*:=(e_r^*)_r$
for the induced idempotents of the ``big" $p$-adic Hecke 
algebras $\H:=\varprojlim_r \H_r$ and $\H^*:=\varprojlim_r \H_r^*$, with the projective limits 
formed using the transition mappings induced by the maps on Jacobians $J_r\rightrightarrows J_{r'}$
for $r'\ge r$ arising (via Picard functoriality) from $\sigma$ and $\rho$, respectively.
We will consistently view $\H^*$ (respectively $\H$) as a $\Lambda$-module
via the $\Z_p$-linear map $\Lambda\hookrightarrow \H^*$ (respectively $\Lambda\hookrightarrow \H$)
sending $u\in \Delta_1$ to $\langle u\rangle^*$ (respectively $\langle u\rangle$).
Let $w_r$ be the Atkin--Lehner 
``involution" of $X_r$ over $\Q(\mu_{Np^r})$ corresponding to a choice
of primitive $Np^r$-th root of unity as in Appendix \ref{tower} and simply write $w_r$
for the automorphism $\Alb(w_r)$ of $J_r$ over $\Q(\mu_{Np^r})$ induced by Albanese
functoriality.
We note that for any Hecke operator $T\in \H_r(\Z)$
one has the relation 
$w_rT=T^*w_r$ as endomorphisms of $J_r$ over $\Q(\mu_{Np^r})$; see
\cite[pg. 336]{Tilouine}, \cite[2.1.8]{OhtaEichler}, or \cite[Chapter 2, \S5.6 (c)]{MW-Iwasawa}.

\subsection{Acknowledgements}

It is a pleasure to thank 
Brian Conrad, Adrian Iovita, Joseph Lipman, 
and Romyar Sharifi 
for enlightening conversations and correspondence,
and Doug Ulmer for supplying the proof
of Lemma \ref{UlmerLemma}.
This paper owes a great deal to the work of Masami Ohta, to whom this paper is dedicated, 
and I heartily thank him
for graciously hosting me during a visit to Tokai University in August, 2009.

\tableofcontents

\section{\texorpdfstring{{d}e~Rham Cohomology of Modular Curves in characteristic $p$}
{{d}e~Rham Cohomology of Modular Curves in characteristic p}}\label{Prelim}

We keep the notation of \S\ref{Notation}, and
write $\X_r$ for the Katz--Mazur model of $X_1(Np^r)$
over $R_r:=\Z_p[\mu_{p^r}]$ provided by Definition \ref{XrDef}.
Due to Proposition \ref{XrCptRepresentability},
the scheme $\X_r$ is regular and proper flat of pure relative dimension one
over $R_r$ with fibers that are geometrically connected
and reduced thanks to Proposition \ref{redXr}; it is therefore
a {\em curve} in the sense of Definition \ref{curvedef}
due to Corollary \ref{curvecorollary}.

Denote by $\o{\X}_r:=\X_r\times_{R_r}\F_p$ the 
special fiber of $\X_r$.
As $\o{\X}_r$ is again a curve in the sense of 
Definition \ref{curvedef}, the relative dualizing sheaf $\omega_{\o{\X}_r/\F_p}$
exists and is a line bundle on $\o{\X}_r$, and there is a canonical
two-term complex $\mathrm{d}:\O_{\o{\X}_r}\rightarrow \omega_{\o{\X}_r/\F_p}$
whose hypercohomology provides 
a well-behaved version of de~Rham cohomology
for non-smooth curves such as $\o{\X}_r$; see Appendix
\ref{GD}.  In particular,  thanks to Proposition \ref{HodgeFilCrvk},
the first hypercohomology of this complex sits in a functorial short-exact sequence
of $\F_p$-vector spaces
\begin{equation}
	\xymatrix{
		0 \ar[r] & {H^0(\o{\X}_r,\omega_{\o{X}_r/\F_p})} \ar[r] & {H^1(\o{\X}_r/\F_p)} \ar[r] & 
		{H^1(\o{\X}_r,\O_{\o{\X}_r})}\ar[r] & 0
	}\label{ExSeq:SpecialFiber}
\end{equation}
which we denote by $H(\o{\X}_r/\F_p)$.

In this section, we will use Rosenlicht's explicit description of the dualizing sheaf 
to compute the {\em ordinary part} of the cohomology $H(\o{\X}_r/\F_p)$
in terms of the de Rham cohomology of the {\em Igusa tower}, which we analyze in \S\ref{IgusaTower}.
The crucial ingredient in our analysis is Nakajima's beautiful {\em equivariant Deuring--Shafarevich formula},
which we recall first.

\subsection{The Cartier operator}\label{CarterOp}

Fix a perfect field $k$ of characteristic $p > 0$ and write $\varphi:k\rightarrow k$ for the $p$-power Frobenius map. 
In this section, we recall the basic theory of the Cartier operator 
for a smooth and proper curve over $k$.  As we will only need the theory
in this limited setting, we will content ourselves with a somewhat {\em ad hoc} formulation of it.
Our exposition follows \cite[\S10]{SerreTopology}, but 
the reader may consult \cite[\S5.5]{Oda} or \cite{CartierNouvelle} for a more general treatment.

Let $X$ be a smooth and proper curve over $k$
and write $F:X\rightarrow X$ 
for the absolute Frobenius map; 
it is finite and flat and is a morphism over the endomorphism of $\Spec(k)$ induced by $\varphi$.
Moreover, if $\L$ is a line bundle on $X$, then one has a canonical isomorphism 
$F^*\L\simeq \L^{\otimes p}$ of line bundles.\footnote{This useful description of $F^*\L$ justifies our choice to use the absolute Frobenius map
rather than the {\em relative} Frobenius map $F_{X/k}:X\rightarrow X^{(p)}$, which
does not enjoy any analogous property.}
Let $D$ be an effective Cartier (=Weil) divisor on $X$ over $k$, and write $\O_X(-D)$ for 
the coherent (invertible) ideal sheaf determined by $D$. 
The pullback map $F^*:\O_{X}\rightarrow {F}_*\O_{X}$ carries the ideal sheaf
 $\O_X(-nD) \subseteq \O_{X}$ into ${F}_*\O_X(-npD)$, so we obtain a canonical $\varphi$-semilinear pullback map on cohomology
 \begin{equation}
 	\xymatrix{
		{F^*:H^1(X,\O_X(-nD))} \ar[r] & H^1(X,\O_X(-npD)).
		}\label{Fpullback}
 \end{equation}
By Grothendieck--Serre duality, (\ref{Fpullback}) gives a $\varphi^{-1}$-semilinear ``trace" 
map of $k$-vector spaces
\begin{equation}
	\xymatrix{
		{V:={F}_*:H^0(X,\Omega^1_{X/k}(npD))} \ar[r] & {H^0(X,\Omega^1_{X/k}(nD))}.
		}\label{cartier}
\end{equation}
which, thanks to the very construction of the duality isomorphism \cite[3.4.10]{GDBC},
coincides with the map induced by Grothendieck's trace morphism on dualizing sheaves
attached to the finite map $F$ \cite[2.7.26]{GDBC}.

\begin{proposition}\label{CartierOp}
	Let $X/k$ be a smooth and proper curve, $D$ an effective Cartier divisor on $X$,
	and $n$ a nonnegative integer; for $x\in \R$ we write $\lceil x\rceil$ for the least
	integer $m$ satisfying $m\ge x$.
	\begin{enumerate}
		\item There is a unique $\varphi^{-1}$-linear endomorphism $V:={F}_*$
		of $H^0(X,\Omega^1_{X/k}(nD))$ which is dual, via Grothendieck-Serre
		duality, to pullback by absolute Frobenius on $H^1(X,\O_X(-nD))$.
		\label{CartierExists}	
		\item The map $V$ ``improves poles" in the sense that it factors through the canonical inclusion
	  	\begin{equation*}
	  		\xymatrix{
				{H^0(X,\Omega^1_{X/k}(\lceil\frac{n}{p}\rceil D))} \ar@{^{(}->}[r] & 
				{H^0(X,\Omega^1_{X/k}(nD))}
			}.
	  	\end{equation*}
		\label{CartierImproves}
		
		\item If $\rho:Y\rightarrow X$ is any finite morphism of smooth proper curves over $k$, 
		and $\rho^*D$ is the pullback of $D$ to $Y$, then
		the induced pullback and trace maps
		\begin{equation*}
			\xymatrix{
				H^0(Y,\Omega^1_{Y/k}(n\rho^*D)) \ar@<0.5ex>[r]^-{\rho_*} & 
				\ar@<0.5ex>[l]^-{\rho^*} H^0(X,\Omega^1_{X/k}(nD))
			}
		\end{equation*}
		commute with $V$.
		\label{CartierCommutes}
		
		\item Assume that $k$ is algebraically closed, and that $x$ is a closed point of $X$.
		Choose a uniformizer $t\in \O_{X,x}$, and for any meromorphic differential 
		$\eta$ on $X$, denote by $\eta_x$ the image of $\eta$ under the map
		$ \Omega^1_{k(X)/k}\hookrightarrow \Omega^1_{k(X)/k}\otimes_{\O_{X,x}}\wh{\O}_{X,x} = k(\!(t)\!)\,\mathrm{d}t$.
		\begin{equation*}
			\text{If}\ \eta_x = \sum_{n} a_n t^n\,\mathrm{d}t\quad \text{then}\quad V(\eta)_x = \sum_{n\equiv -1\bmod p} \varphi^{-1}(a_n) t^{(n+1)/p-1}
			\,\mathrm{d}t.
		\end{equation*}
		\label{Vformula}
	
		\item With hypotheses and notation as in $(\ref{Vformula})$, we have
	$
			\res_x(V\eta)^p = \res_x(\eta)
	$
		where $\res_x$ is the canonical ``residue at $x$ map" on meromorphic 
		differentials, determined by the condition $\res_x(\eta) := a_{-1}$.
		\label{CartierResidue}
	\end{enumerate}
\end{proposition}

\begin{proof}
	Both (\ref{CartierExists}) and (\ref{CartierImproves}) follow from our discussion, while
	(\ref{CartierCommutes}) follows (via duality) from the fact that the $p$-power
	map commutes with any ring homomorphism in characteristic $p$. 
	
	To prove (\ref{Vformula}), we work locally at $x$ and use the fact noted above that
	$V$ is induced by Grothendieck's trace map 
	$\Tr_F : F_*\Omega^1_{X/k}\rightarrow \Omega^1_{X/k}$ 
	attached to $F$.  We may find a Zariski open neighborhood $U$
	of $x$ in $X$ that admits a finite \'etale map $U\rightarrow \Aff^1_k$
	carrying $x$ to the origin, and since the formation of $\Tr_F$
	is compatible with \'etale localization, we may reduce to checking the proposed
	formula for $X=\Aff^1_k$ at the origin.  As $\Tr_F$ is compatible with base change,
	we may further reduce to the case $k=\F_p$.  Again invoking compatibility
	with base change, this map is the reduction modulo $p$ of the Grothendieck
	trace mapping attached to the standard lift of Frobenius on $\Aff^1_{\Z_p}$
	given at the level of coordinate rings by the $\Z_p$-algebra map
	$\Phi:\Z_p[T]\rightarrow \Z_p[T]$ sending $T$ to $T^p.$
	Appealing now to the explicit description of Grothendieck's trace mapping afforded by 
	\cite[2.7.41]{GDBC}, we find that for any $b \in \Z_p(\!(T)\!) $, one has
	\begin{equation}
		\Tr_F(b\,\mathrm{d}{\Phi(T)}) = \Tr_{\Phi}(b)\, \mathrm{d}T, \label{TraceKey}
	\end{equation}
	where $\Tr_{\Phi}: \Z_p(\!(T)\!)\rightarrow \Z_p(\!(T)\!)$
	is the standard ring-theoretic trace mapping attached to $\Phi$ (which realizes
	$\Z_p(\!(T)\!)$ as a free module of rank $p$ over itself).  The computation
	\begin{equation*}
		\Tr_{\Phi}(T^j) = \begin{cases} 0 & \text{if}\ j\not\equiv 0\bmod p\\ pT^{j/p} & \text{otherwise}\end{cases}
	\end{equation*}
	is a standard exercise in linear algebra, whence we obtain from (\ref{TraceKey})
	the identity
	$$\Tr_F( \sum_j a_jT^j pT^{p-1}\mathrm{d}T) = \sum_{j\equiv 0\bmod p} p a_j T^{j/p},$$
 	which, after canceling the factor of $p$ on either side (remember that we are now working
	in a rank-1 free $\Z_p(\!(T)\!)$-module!), is equivalent to the desired formula.

	 Finally, (\ref{CartierResidue}) is an immediate consequence of (\ref{Vformula})
	 and the given explicit description of $\res_x$.
\end{proof}

\begin{remark}\label{poletrace}
	As a sort of complement to Proposition \ref{CartierOp} (\ref{CartierResidue}), 
	we have the following: assume that $\rho: Y\rightarrow X$ is
	a finite and {\em generically \'etale} morphism of smooth curves
	over an algebraically closed field $k$.  Then 
	\begin{equation}
		\sum_{y\in \rho^{-1}(x)} \res_y(\eta) = \res_x(\rho_*\eta).\label{TateFormula}
	\end{equation}
	Indeed, as $\rho$ is generically \'etale, we have $\Omega^1_{k(Y)/k} = k(Y)\otimes_{k(X)}\Omega^1_{k(X)/k}$,
	so the claimed formula follows from \cite[Theorem 4]{TateResidues}
	and the explicit description of Grothendieck's trace morphism 
	provided by \cite[2.7.41]{GDBC}, together with the facts that the isomorphism $H^1(X,\Omega^1_{X/k})\rightarrow k$
	induced by the residue map coincides with the negative of Grothendieck's trace isomorphism,
	and the latter is compatible with compositions; see Appendix B and Corollary 3.6.6 of \cite{GDBC}.
\end{remark}

We recall the following (generalization of a) well-known lemma of Fitting:

\begin{lemma}\label{HW}
		Let $A$ be a commutative ring, $\varphi$ an automorphism of $A$,
		and $M$ an $A$-module equipped with a $\varphi$-semilinear
		endomorphism $F:M\rightarrow M.$
		Assume that one of the following holds:
		\begin{enumerate}
			\item $M$ is a finite length $A$-module.\label{finlen}
			\item $A$ is a complete noetherian adic ring, with ideal of definition $I\subseteq A$, and $M$ is a finitely generated $A$-module.\label{top}
		\end{enumerate}
		Then there is a unique direct sum decomposition 
		\begin{equation}
			M = M^{F_{\ord}} \oplus M^{F_{\nil}},\label{FittingDecomp}
		\end{equation}
		where $M^{F_{\ord}}$ 
		is the maximal $F$-stable
		submodule of $M$ on which $F$ is bijective, and $M^{F_{\nil}}$
		is the maximal $F$-stable submodule of $M$ on which $F$
		is $($topologically$)$ nilpotent.  The assignment $M\rightsquigarrow M^{F_{\star}}$
		for $\star=\ord, \nil$ is an exact functor on the category of $($left$)$
		$A[F]$-modules verifying $(\ref{finlen})$ or $(\ref{top})$.
\end{lemma}

\begin{proof}
	For the proof in case (\ref{finlen}), we refer to \cite[\Rmnum{6}, 5.7]{LazardGroups},
	and just note that one has:
	\begin{equation*}
		M^{F_{\ord}}:=\bigcap_{n\ge 0} \im(F^n)\quad\text{and}\quad
		M^{F_{\nil}}:=\bigcup_{n\ge 0} \ker(F^n),
	\end{equation*}
	where one uses that $\varphi$ is an automorphism to know that the image and kernel 
	of $F^n$ are $A$-submodules of $M$.
	It follows immediately from this that the association $M\rightsquigarrow M^{F_{\star}}$
	is a functor from the category of left $A[F]$-modules of finite $A$-length
	to itself.  It is an exact functor because the canonical inclusion $M^{F_{\star}}\rightarrow M$
	is an $A[F]$-direct summand.  
	In case (\ref{top}), our hypotheses ensure that $M/I^nM$ is a noetherian and Artinian
	$A$-module, and hence of finite length, for all $n$.  Our assertions in this situation
	then follow immediately from (\ref{finlen}), via the uniqueness of (\ref{FittingDecomp}),
	together with fact that $M$ is finite as an $A$-module, and hence $I$-adically complete (as $A$ is).
\end{proof}

We apply \ref{HW} to the $k$-vector space $M:=H^0(X,\Omega^1_{X/k})$
equipped with the $\varphi^{-1}$ semilinear map $V$:

\begin{definition}\label{ordnil}
	The $k[V]$-module $H^0(X,\Omega^1_{X/k})^{V_{\ord}}$ is called the
	{\em $V$-ordinary subspace} of holomorphic differentials on $X$.  
	The integer $\gamma_X:=\dim_k H^0(X,\Omega^1_{X/k})^{V_{\ord}}$
	is called the {\em Hasse-Witt invariant} of $X$.
\end{definition}

\begin{remark}\label{DualityOfFVOrd}
	Let $D$ be any effective Cartier divisor.
	Since $V:={F}_*$ and $F:=F^*$ are adjoint under the canonical perfect 
	$k$-pairing between $H^0(X,\Omega^1_{X/k}(D))$ and $H^1(X,\O_X(-D))$, this pairing
	restricts to a perfect duality pairing 
	\begin{equation}
		\xymatrix{
			{H^0(X,\Omega^1_{X/k}(D))^{V_{\ord}} \times H^1(X,\O_X(-D))^{F_{\ord}}} \ar[r] & {k}
		}.\label{DualityOfFVOrdMap}
	\end{equation}
	In particular (taking $D=0$) we also have $\gamma_X=\dim_k H^1(X,\O_X)^{F_{\ord}}$.
\end{remark}

The following ``control lemma" is a manifestation of the fact that the Cartier
operator improves poles (Proposition \ref{CartierOp}, (\ref{CartierImproves})): 

\begin{lemma}\label{sspoles}
	Let $X$ be a smooth and proper curve over $k$ and $D$ an effective Cartier divisor 
	on $X$. Considering $D$ as a closed subscheme of $X$, we write $D_{\red}$ for the associated 
	reduced closed subscheme.  
	\begin{enumerate}
		\item For all positive integers $n$, the canonical morphism
			\begin{equation*}
				H^0(X,\Omega^1_{X/k}(D_{\red})) \rightarrow H^0(X,\Omega^1_{X/k}(nD))
			\end{equation*}
			induces a natural isomorphism on $V$-ordinary subspaces.\label{VControl}
			
		\item For each positive integer $n$, the canonical map
			\begin{equation*}
				H^1(X, \O_X(-nD)) \rightarrow H^1(X,\O_X(-D_{\red}))
			\end{equation*}
			induces a natural isomorphism on $F$-ordinary subspaces.\label{FControl} 
		
		\item The identifications in $(\ref{VControl})$ and $(\ref{FControl})$ are canonically 
		$k$-linearly dual, via Remark $\ref{DualityOfFVOrd}$. 	
		
	\end{enumerate}
\end{lemma}

\begin{proof}
	This follows immediately from Proposition \ref{CartierOp} (\ref{CartierImproves})
	and Remark \ref{DualityOfFVOrd}.
\end{proof}

Now let $\pi:Y\rightarrow X$ be a finite map of smooth, proper and geometrically connected curves over $k$
that is generically \'etale and Galois with group $G$ that is a $p$-group.  
Let $D_X$ be any effective Cartier divisor on $X$ over $k$ with support containing the ramification locus of $\pi$,
and put $D_Y=\pi^*D_X$.   As in Lemma \ref{sspoles},  denote by $D_{X,\red}$ and $D_{Y,\red}$ the underlying reduced closed subschemes;
as $D_{Y,\red}$ is $G$-stable, the $k$-vector spaces 
$H^0(Y,\Omega^1_{Y/k}(nD_{Y,\red}))$ and $H^1(Y,\O_Y(-nD_{Y,\red})$
are canonically $k[G]$-modules for any $n\ge 1$.  The following {\em equivariant Deuring--Shafarevich
formula} of Nakajima 
is the key to the proofs of our structure theorems for $\Lambda$-modules:

\begin{proposition}[Nakajima]\label{Nakajima}
	Assume that $\pi$ is ramified, 
	let $\gamma_X$ be the Hasse-Witt invariant of $X$ and set $d:=\gamma_X-1+\deg (D_{X,\red})$.
	Then for each positive integer $n$:
	\begin{enumerate}
		\item The $k[G]$-module 
		 $H^0(Y,\Omega^1_{Y/k}(nD_{Y,\red}))^{V_{\ord}}$ is free of rank $d$ and independent
		 of $n$.\label{NakajimaOne}
		
		\item The $k[G]$-module $H^1(Y,\O_Y(-nD_{Y,\red}))^{F_{\ord}}$ is naturally
		isomorphic to the contragredient of $H^0(Y,\Omega^1_{Y/k}(nD_{Y,\red}))^{V_{\ord}}$;
		as such, it is $k[G]$-free of rank $d$ and independent of $n$.
		\label{NakajimaTwo}
	\end{enumerate}
\end{proposition}

\begin{proof}
	The independence of $n$ is simply Lemma \ref{sspoles};
	using this, the first assertion is then equivalent to Theorem 1 of \cite{Nakajima}.
	The second assertion is immediate from Remark \ref{DualityOfFVOrd},
	once one notes that for $g\in G$ one has the identity $g_*=(g^{-1})^*$
	on cohomology (since $g_*g^*=\deg g = \id$), so $g^*$ and $(g^{-1})^*$
	are adjoint under the duality pairing (\ref{DualityOfFVOrdMap}).
\end{proof}

\subsection{The Igusa tower}\label{IgusaTower}

We apply Proposition \ref{Nakajima} to the {\em Igusa tower}.
Recall (Definition \ref{IgusaDef}) that for $r\ge 0$ we write $\Ig_r$
for the compactified moduli scheme classifying
Elliptic curves over $\F_p$-schemes equipped with 
a level-$p^r$ Igusa structure 
and a $\mu_N$-structure as in Definition \ref{IgusaStrDef}.
Then $\Ig_r$ is a smooth, projective and geometrically irreducible
curve over $\F_p$ by Proposition \ref{Pr:IgusaSmoothCrv},
and comes equipped with a 
canonical degeneracy map 
$\pr:\Ig_{r}\rightarrow \Ig_1$ defined moduli-theoretically by (\ref{Vmapsch}) that is
finite \'etale outside\footnote{We will frequently 
write simply $\SS$ for the divisor $\SS_r$ on $\Ig_r$
when $r$ is clear from context.
}
$\SS:=\SS_r$ and totally (wildly) ramified
over $\SS_1$. In this way, $\Ig_r$ is a branched cover
of $\Ig_1$ with group $\Delta_1/\Delta_r$, so that
the cohomology groups $H^0(\Ig_r,\Omega^1_{\Ig_r/\F_p}(\SS))$
and $H^1(\Ig_r,\O_{\Ig_r}(-\SS))$ are naturally
$\F_p[\Delta_1/\Delta_r]$-modules.
Note that we have $\Ig_0=X_1(N)_{\F_p}$, the compactified modular curve over $\F_p$
classifying elliptic curves with a $\mu_N$-structure.

\begin{proposition}\label{IgusaStructure}
	Let $r\ge 1$ be an integer, write $\gamma$ for the $p$-rank of
	$\Pic^0_{X_1(N)/\F_p}$, and set $\delta:=\deg\SS$.
\begin{enumerate}	
	\item The $\F_p[\Delta_1/\Delta_r]$-modules $H^0(\Ig_r,\Omega^1_{\Ig_r/\F_p}(\SS))^{V_{\ord}}$
	and $H^1(\Ig_r,\O_{\Ig_r}(-\SS))^{F_{\ord}}$ are both free of rank $d:=\gamma+\delta-1$.
	Each is canonically isomorphic to the contragredient of the other.
	\label{IgusaFreeness}
	 
	\item For any positive integer $s\le r$, the canonical trace mapping
	associated to $\pr:\Ig_r\rightarrow \Ig_s$ induces natural isomorphisms of $\F_p[\Delta_1/\Delta_s]$-modules 
	\begin{subequations}
		\begin{equation*}
			\xymatrix{
				{\pr_*:H^0(\Ig_r, \Omega^1_{\Ig_r/\F_p}(\SS))^{V_{\ord}}
				\tens_{\F_p[\Delta_1/\Delta_r]} \F_p[\Delta_1/\Delta_s]} \ar[r]^-{\simeq} &
				{H^0(\Ig_s, \Omega^1_{\Ig_s/\F_p}(\SS))^{V_{\ord}}}
			}
		\end{equation*}
		\begin{equation*}
			\xymatrix{
				{\pr_*:H^1(\Ig_r, \O_{\Ig_r}(-\SS))^{F_{\ord}}
				\tens_{\F_p[\Delta_1/\Delta_r]} \F_p[\Delta_1/\Delta_s]} \ar[r]^-{\simeq} &
				{H^1(\Ig_s, \O_{\Ig_s}(-\SS))^{F_{\ord}}}
			}
		\end{equation*}
	\end{subequations}	
	\label{IgusaControl}
\end{enumerate}
\end{proposition}

In order to prove Proposition \ref{IgusaStructure}, we require the following Lemma
({\em cf.} \cite[p. 511]{MW-Analogies}):

\begin{lemma}\label{MW}
	Let $\pi: Y\rightarrow X$ be a finite and generically \'etale morphism of
	smooth proper and geometrically irreducible curves over a field $k$.
	If there is a geometric point of $X$ over which $\pi$ is totally ramified then the induced map
	of $k$-group schemes $\Pic(\pi):\Pic_{X/k}\rightarrow \Pic_{Y/k}$ has trivial scheme-theoretic kernel.
\end{lemma}

\begin{proof}
	The hypotheses and the conclusion are preserved under extension of $k$, so we may
	assume that $k$ is algebraically closed.  We fix a $k$-point $x\in X(k)$ over which
	$\pi$ is totally ramified, and let $y\in Y(k)$ be the unique $k$-point of $Y$ over $x$.
	To prove that $\Pic_{X/k}\rightarrow \Pic_{Y/k}$
	has trivial kernel, it suffices to prove that the map of groups  
	$\pi_R^*:\Pic(X_R)\rightarrow \Pic(Y_R)$
	is injective for every 
	$k$-algebra $R$.  
	We fix such a $k$-algebra, and denote by $x_R\in X_R(R)$ and $y_R\in Y_R(R)$
	the points obtained from $x$ and $y$ by base change.  Let $\L$ be a line bundle
	on $X_R$ whose pullback to $Y_R$ is trivial; our claim is that we may choose
	a trivialization $\pi^*\L\xrightarrow{\simeq} \O_{Y_R}$ of $\pi^*\L$
	over $Y_R$ which descends to $X_R$.  In other words, by descent theory,
	we assert that we may choose a trivialization of $\pi^*\L$ with the property that the
	two pullback trivializations under the canonical projection maps
	\begin{equation}
		\xymatrix{
			{Y_R \times_{X_R} Y_R} \ar@<-0.5ex>[r]_-{\pr_2}\ar@<0.5ex>[r]^-{\pr_1} & {Y_R}
			}\label{TwoPullback}
	\end{equation}
	coincide.
	
	We first claim that the $k$-scheme $Z:=Y\times_X Y$ is connected and generically reduced.
	Since $\pi$ is totally ramified over $x$, 
	there is a unique geometric point $(y,y)$ of $Z$ mapping to $x$ under the canonical map 
	$Z\rightarrow X$.  Since this map is moreover finite flat (because $\pi:Y\rightarrow X$
	is finite flat due to smoothness of $X$ and $Y$), every connected component of $Z$ is finite flat  	
	onto $X$ and so passes through $(y,y)$.  Thus, $Z$ is connected.
	On the other hand, $\pi:Y\rightarrow X$ is generically
	\'etale by hypothesis, so there exists a dense open subscheme $U\subseteq X$
	over which $\pi$ is \'etale.  Then $Z\times_X U$ is \'etale---and hence smooth---over $U$
	and the open immersion $Z\times_X U\rightarrow Z$ is schematically dense 
	as $U\rightarrow X$ is schematically dense and $\pi$ is finite
	and flat.  As $Z$ thus contains a $k$-smooth and dense subscheme,
	it is generically reduced.
	
	Fix a choice $e$ of $R$-basis of the fiber $\L(x_R)$ of $\L$ at $x_R$.
	As any two trivializations of $\pi^*\L$ over $Y_R$ differ by an element of 
	$R^{\times}$, we may uniquely choose a trivialization $\pi^*\L\simeq \O_{Y_R}$
	with the property that the induced isomorphism
	\begin{equation}
		\xymatrix{ 
			{\L(x_R)\simeq (\pi^*\L)(y_R)}\ar[r]^-{\simeq} & {\O_{Y_R}(y_R)\simeq R}
			}\label{TrivOnYunit}
	\end{equation}
	carries $e$ to $1$.  The obstruction to the two pullback trivializations
	under (\ref{TwoPullback}) being equal is a global unit on $Y_R\times_{X_R} Y_R$.
	But since $Y_R\times_{X_R} Y_R = (Y\times_X Y)_R$, we have by flat base change
	\begin{equation*}
		H^0(Y_R\times_{X_R} Y_R, \O_{Y_R\times_{X_R} Y_R}) = H^0(Y\times_X Y,\O_{Y\times_X Y})\otimes_k R=R
	\end{equation*}
	where the last equality rests on the fact that $Y\times_X Y$ is connected, generically reduced, 
	and proper over $k$.
	Thus, the obstruction to the two pullback trivializations being equal is an element of $R^{\times}$,
	whose value may be calculated at any point of $Y_R\times_{X_R} Y_R$.  By our choice (\ref{TrivOnYunit})
	of trivialization of $\pi^*\L$, the value of this obstruction at the point $(y_R,y_R)$ is 1,
	and hence the two pullback trivializations coincide as desired.
\end{proof}

\begin{proof}[Proof of Proposition $\ref{IgusaStructure}$]
	Since $\pr:\Ig_r\rightarrow \Ig_s$ is a finite branched cover with group $\Delta_s/\Delta_r$
	and totally wildly ramified over $\SS_s$,
	we may apply Proposition \ref{Nakajima}, which gives (\ref{IgusaFreeness}).

	To prove (\ref{IgusaControl}), we work over $k:=\o{\F}_p$ and
	argue as follows.  Since $\pr:\Ig_r\rightarrow \Ig_{s}$
	is of degree $p^{r-s}$ and totally ramified over $\SS_{s}$, we have $\pr^*\SS_{s}=p^{r-s}\cdot\SS$;
	it follows that pullback induces a map
	\begin{equation}
		\xymatrix{
				{H^1(\Ig_{s},\O_{\Ig_{s}}(-\SS_{s}))}	\ar[r]^-{\pr^*} & 	{H^1(\Ig_r,\O_{\Ig_r}(-\SS))}
			}\label{pbH1poles}
	\end{equation}
	which we claim is {\em injective}.
	To see this,
	we observe that the long exact cohomology sequence attched to
	the short exact sequence of sheaves on $\Ig_r$ 
	\begin{equation*}
		\xymatrix{
			0\ar[r] & {\O_{\Ig_r}(-\SS)} \ar[r] & {\O_{\Ig_r}} \ar[r] & {\O_{\SS}} \ar[r] & 0
		}
	\end{equation*}
	(with $\O_{\SS}$ a skyscraper sheaf supported on $\SS$) 
	yields a commutative diagram with exact rows
	\begin{equation}
	\begin{gathered}
		\xymatrix@C=15pt{
			0\ar[r] & {H^0(\Ig_{s},\O_{\Ig_{s}})} \ar[r]\ar[d] & 
			{H^0(\Ig_{s},\O_{\SS_{s}})} \ar[r]\ar[d] & 
			{H^1(\Ig_{s},\O_{\Ig_{s}}(-\SS_{s}))} \ar[r]\ar[d] & 
			{H^1(\Ig_{s},\O_{\Ig_{s}})} \ar[r]\ar[d] & 0\\
			0\ar[r] & {H^0(\Ig_r,\O_{\Ig_r})} \ar[r] & 
			{H^0(\Ig_r,\O_{\SS})} \ar[r] & {H^1(\Ig_r,\O_{\Ig_r}(-\SS))} \ar[r] & 
			{H^1(\Ig_r,\O_{\Ig_r})} \ar[r] & 0
		}
	\end{gathered}	
		\label{pbH1inj}
	\end{equation}
	The leftmost vertical arrow is an isomorphism because $\Ig_r$ is geometrically connected for all $r$.
	Since $\SS$ is reduced, we have
		$H^0(\Ig_r,\O_{\SS})=k^{\deg\SS}$
	for all $r$, so since $\pr:\Ig_r\rightarrow \Ig_{s}$ totally ramifies over every 
	supersingular point, the second vertical arrow in 
	(\ref{pbH1inj}) is also an isomorphism.  
	Now the rightmost vertical map in (\ref{pbH1inj}) is	
	identified with the map on Lie algebras
	${\Lie \Pic^0_{\Ig_{s}/k}} \rightarrow {\Lie\Pic^0_{\Ig_r/k}}$
	induced by $\Pic^0(\pr)$, 
	which is injective thanks to Lemma \ref{MW} and the left-exactness of the functor $\Lie$.
	An easy diagram chase 
	using (\ref{pbH1inj}) then shows that (\ref{pbH1poles}) is injective, as claimed.

	Using again the equality $\pr^*(\SS_s)=p^{r-s}\cdot\SS_r$, 
	pullback of meromorphic differentials yields a mapping	
	\begin{equation}
		\xymatrix{
			{H^0(\Ig_s,\Omega^1_{\Ig_s/k}(\SS))} \ar[r] & {H^0(\Ig_r, \Omega^1_{\Ig_r/k}(p^{r-s}\cdot\SS))}
		}\label{diffPBinj}
	\end{equation}
	which is injective since $\pr:\Ig_r\rightarrow \Ig_s$ is separable.
	
	Dualizing the injective maps (\ref{pbH1poles}) and (\ref{diffPBinj}), we see that the canonical
	trace mappings
	\begin{subequations}
		\begin{equation}
			\xymatrix{
				{H^0(\Ig_r,\Omega^1_{\Ig_r/k}(\SS))} \ar[r]^-{\pr_*} & 	
				{H^0(\Ig_{s},\Omega^1_{\Ig_{s}/k}(\SS))}
			}
		\end{equation}
		\begin{equation}
			\xymatrix{
				{H^1(\Ig_r,\O_{\Ig_r}(-p^{r-s}\cdot \SS))} \ar[r]^-{\pr_*} & {H^1(\Ig_{s},\O_{\Ig_{s}}(-\SS))}
			}
		\end{equation}
	\end{subequations}	
	are surjective for all $r\ge s\ge 1$.
	Passing to $V$- (respectively $F$-) ordinary parts and using
	Lemma \ref{sspoles} (\ref{VControl}), we conclude that the canonical trace 
	mappings attached to $\Ig_r\rightarrow \Ig_s$ induce {\em surjective} maps
	as in Proposition \ref{IgusaStructure} (\ref{IgusaControl}).  By (\ref{IgusaFreeness}), 
	these mappings are then surjective
	mappings of free $\F_p[\Delta_1/\Delta_s]$-modules of the same rank, and are hence
	isomorphisms.  
\end{proof}

The group $\F_p^{\times}$ acts naturally on the Igusa curve $\Ig_1$ through the
{\em diamond operators} (see below Definition \ref{IgusaDef}), and the eigenspaces
for the induced action on $H^0(\Ig_1,\Omega^1_{\Ig_1/\F_p}(\SS))$
are intimately connected with mod $p$ cusp forms:

\begin{proposition}\label{MFmodp}
	Let $S_k(\Upgamma_1(N);\F_p)$ be the space of weight $k$ cuspforms for $\Upgamma_1(N)$ over $\F_p$,
	and for $0\le j<p$ denote
	 by $H^0(\Ig_r,\Omega^1_{\Ig_1/\F_p}(\SS))^{(j)}$ the subspace
	of $H^0(\Ig_r,\Omega^1_{\Ig_1/\F_p}(\SS))$ on which $\F_p^{\times}$ acts through the
	$j$-th power of the Teichm\"uller character.
	For each $k$ with $ 2 < k < p+1$, there are canonical isomorphisms
	of $\F_p$-vector spaces
	\begin{equation}
	\addtocounter{equation}{1}
		S_k(\Upgamma_1(N);\F_p) \simeq H^0(\Ig_1,\Omega^1_{\Ig_1/\F_p})^{(k-2)} \simeq  H^0(\Ig_1,\Omega^1_{\Ig_1/\F_p}(\SS))^{(k-2)} 
		\tag{$\arabic{section}.\arabic{subsection}.\arabic{equation}_k$}
		\label{WtkIsom}
	\end{equation}	
	which are equivariant for the Hecke operators and diamond operators $\langle \cdot\rangle_N$, with $U_p$ acting as usual on modular forms
	and as the Cartier operator $V$ 
	on differential forms.  For $k=2,$ $p+1$, we have the following commutative diagram:
	\begin{equation*}
		\xymatrix{
			{S_2(\Upgamma_1(N);\F_p)} \ar[r]^-{\simeq} \ar@{^{(}->}[d]_-{\cdot A} & 
			{H^0(\Ig_1,\Omega^1_{\Ig_1/\F_p})^{(0)}} \ar@{^{(}->}[d] \\
			{S_{p+1}(\Upgamma_1(N);\F_p)} \ar[r]^-{\simeq} & {H^0(\Ig_1,\Omega^1_{\Ig_1/\F_p}(\SS))^{(0)}}  \\
			}
	\end{equation*}
	where $A$ is the Hasse invariant.
\end{proposition}

\begin{proof}
	This follows from Propositions 5.7--5.10 of
	\cite{tameness}, using Lemma \ref{CharacterSpaces} (keeping in mind Remark \ref{GrossCompat});
	we note that our forward reference to Lemma \ref{CharacterSpaces} does not result in circular reasoning.
\end{proof}

\begin{remark}\label{dMFmeaning}
	For each $k$ with $2\le k\le p+1$, let us write $d_k:=\dim_{\F_p} S_k(\Upgamma_1(N);\F_p)^{\ord}$
	for the $\F_p$-dimension of the subspace of weight $k$ level $N$ cuspforms over $\F_p$
	on which\footnote{In this characteristic $p$ setting, it is standard to write $U_p$
	for the Hecke operator $T_p$, as the effect of these operators on $q$-expansions
	agrees in characteristic $p$; {\em cf.} \cite[\S4]{tameness}.} $U_p:=T_p$ acts invertibly.
	As in Proposition \ref{IgusaStructure} (\ref{IgusaFreeness}), 
	let $\gamma$ be the $p$-rank of the Jacobian of
	$X_1(N)_{\F_p}$ and $\delta:=\deg\SS$.  It follows immediately from Proposition
	\ref{MFmodp} that we have the equality
	\begin{equation}
		d :=\gamma+\delta-1 = \sum_{k=3}^{p+1} d_k.
	\end{equation}
\end{remark}

\subsection{Structure of the ordinary part of \texorpdfstring{$H^0(\o{\X}_r,\omega_{\o{\X}_r/\F_p})$}{
H0(Xr,omega)}
}\label{OrdStruct}

As in the introduction to \S\ref{Prelim}, we write $\X_r$
for the regular proper model of $X_1(Np^r)$ over $R_r$
given by Definition \ref{XrDef}, and
we denote by $\o{\X}_r:=\X_r\times_{R_r}\F_p$ the 
special fiber of $\X_r$.
There is a natural ``semilinear" action of  
$\Gamma=\Gal(\Q_p(\mu_{p^{\infty}})/\Q_p)$ on the scheme
$\X_r$ given by (\ref{gammamaps}), which records the fact that its generic fiber 
is defined over $\Q_p$; this action induces a ``geometric inertia
action" of $\Gamma$ on the special fiber $\o{\X}_r$ over $\F_p$.

By Proposition \ref{redXr}, the curve $\o{\X}_r$  
is a ``disjoint union with crossings at the supersingular points"
of Igusa curves $I_{(a,b,u)}:=\Ig_{\max(a,b)}$ indexed by
triples $(a,b,u)$ with $a,b$ nonnegative integers satisfying $a+b=r$
and $u$ ranging over all elements of $(\Z/p^{\min(a,b)}\Z)^{\times}$.
Using Rosenlicht's explicit description of the relative dualizing sheaf
$\omega_{\o{\X}_r/\F_p}$ as a certain sheaf of meromorphic differentials
on the normalization $\nor{\o{\X}}_r$ of $\o{\X}_r$,
we now relate the ordinary 
part of the cohomology $H(\o{\X}_r/\F_p)$ given by (\ref{ExSeq:SpecialFiber}) 
to the de Rham cohomology of the Igusa tower.

For notational ease, as in Remark \ref{MWGood} we write $I_r^{\infty}:=I_{(r,0,1)}$ and 
$I_r^0:=I_{(0,r,1)}$ for the two ``good" components of $\o{\X}_r$,
each of which is a copy of $\Ig_r$.
For $\star=0,\infty$, we denote by $i_r^{\star}: I_r^{\infty}\hookrightarrow \o{\X}_r$ the canonical closed immersion,
and for $s\le r$, we 
write simply $\pr:I_r^{\star}\rightarrow I_s^{\star}$ for the 
the degeneracy map (\ref{Vmapsch}) on Igusa curves.
The diamond operator and Hecke correspondences on $\X_r$ and $I_r^{\star}\simeq \Ig_r$
are defined in Appendix \ref{tower} (see especially Definition 
\ref{Def:HeckeModuliProblem}). Each of these correspondences $T$
induces endomorphisms $T$ and $T^*$
of $H^0(\o{\X}_r,\omega_{\o{\X}_r})$ via 
Lemma \ref{ReductionCompatibilities} (\ref{BaseChngDiagram})
and (\ref{HeckeDef}) (see below),
and of
$H^0(I_r^{\star},\Omega^1_{I_r^{\star}}(\SS))$ via (\ref{HeckeDef}).
There are thus two possible ways of viewing each of these cohomology groups
as modules over the group ring $\F_p[\Delta/\Delta_r]$ of diamond operators at $p$:
either via the ``usual" action, or the adjoint action.
We give $e_r^*H^0(\o{\X}_r,\omega_{\o{\X}_r})$ (respectively 
$e_rH^0(\o{\X}_r,\omega_{\o{\X}_r})$) the structure induced by
the adjoint (respectively usual) diamond operator action,
which induces the {\em same} $\F_p[\Delta_1/\Delta_r]$-module
structure arising from the natural $\H_r^*$ (respectively $\H_r$) module structure
and the $\Lambda$-algebra structure of $\H^*$ (respectively $\H$)
fixed in \S\ref{Notation}.
On the other hand, 
in what follows we will {\em always} give $H^0(I_r^{\star},\Omega^1_{I_r^{\star}}(\SS))$
the $\F_p[\Delta/\Delta_r]$-module structure 
determined by 
requiring that $ u\in \Delta$ act as $\langle u\rangle$ (and {\em not} the adjoint $\langle u\rangle^* = \langle u^{-1}\rangle$).

\begin{proposition}\label{charpord}
	For each positive integer $r$, pullback of differentials along		
	$i_r^{0}$ $($respectively $i_r^{\infty}$$)$ induces a natural 
	isomorphism of $\F_p[\Delta/\Delta_r]$-modules
		\begin{equation}
				\xymatrix{
					{e_r^*H^0(\o{\X}_r,\omega_{\o{\X}_r})} \ar[r]^-{\simeq}_-{(i_r^{0})^*} &
					{H^0(I_r^{0},\Omega^1_{I_r^{0}}(\SS))^{V_{\ord}}}
					},
		\ \text{resp.}\ 
				\xymatrix{
					{e_rH^0(\o{\X}_r,\omega_{\o{\X}_r})} \ar[r]^-{\simeq}_-{(i_r^{\infty})^*} &
					{H^0(I_r^{\infty},\Omega^1_{I_r^{\infty}}(\SS))^{V_{\ord}}}
					}.
					\label{I'compIsom}			
		\end{equation}
	that is equivariant for the Hecke operators $T_{\ell}^*$, $U_{\ell}^*$
	and $\langle v\rangle_N^*$
	$($respectively $T_{\ell}$, $U_{\ell}$
	and $\langle v\rangle_N$$)$ and 		
	$\Gamma$-equivariant for the geometric inertia action 
	on $\o{\X}_r$ and the action $\gamma\mapsto \langle \chi(\gamma)^{-1}\rangle^{*}$ 
	$($respectively the trivial action$)$
	on 
	$H^0(I_r^{\star},\Omega^1_{I_r^{\star}}(\SS))$
	for $\star=0$ $($resp. $\star=\infty$$)$.	
	The isomorphisms $(\ref{I'compIsom})$ 
	induce identifications that are
	compatible with change in $r$:
	the diagrams formed
	by taking 
	the interior or the exterior arrows
	\begin{equation}
		\begin{gathered}
			\xymatrix@R=30pt{
				{e_r^*H^0(\o{\X}_r,\omega_{\o{\X}_r})} 
				\ar@<-0.5ex>[r]_-{\langle p\rangle_N^{r} (i_r^0)^*}
				\ar@<0.5ex>[r]^-{F_*^r (i_r^0)^*} \ar@<-0.5ex>[d]_-{\pr_*} &
				{H^0(I_r^{0},\Omega^1_{I_r^{0}}(\SS))^{V_{\ord}}} \ar@<0.5ex>[d]^-{\pr_*}  \\
				{e_{s}^*H^0(\o{\X}_{s},\omega_{\o{\X}_{s}})} 
				\ar@<-0.5ex>[r]_-{F_*^s (i_s^0)^*}
				\ar@<0.5ex>[r]^-{\langle p\rangle_N^{s} (i_s^0)^*}
				\ar@<-0.5ex>[u]_-{\ps^*} &
				{H^0(I_s^{0},\Omega^1_{I_s^{0}}(\SS))^{V_{\ord}}} \ar@<0.5ex>[u]^-{\pr^*} 
			}
		\quad\raisebox{-24pt}{and}\quad
			\xymatrix@R=30pt{
				{e_rH^0(\o{\X}_r,\omega_{\o{\X}_r})} \ar@<-0.5ex>[r]_-{(i_r^{\infty})^*}
				\ar@<0.5ex>[r]^-{F_*^r (i_r^{\infty})^*}
				\ar@<-0.5ex>[d]_-{\ps_*} &
				{H^0(I_r^{\infty},\Omega^1_{I_r^{\infty}}(\SS))^{V_{\ord}}} \ar@<0.5ex>[d]^-{\pr_*}  \\
				{e_{s}H^0(\o{\X}_{s},\omega_{\o{\X}_{s}})} \ar@<-0.5ex>[r]_-{F_*^s (i_s^{\infty})^*}
				\ar@<0.5ex>[r]^-{(i_s^{\infty})^*}
				\ar@<-0.5ex>[u]_-{\pr^*} &
				{H^0(I_s^{\infty},\Omega^1_{I_s^{\infty}}(\SS))^{V_{\ord}}} \ar@<0.5ex>[u]^-{\pr^*} 
			}			
		\end{gathered}\label{rCompatDiagrams}
		\end{equation}
	are all commutative for $s\le r$.
	The same assertions hold true if we replace $\o{\X}_r$ with $\nor{\o{\X}}_r$ and $\Omega^1_{I_r^{\star}}(\SS)$
	with $\Omega^1_{I_r^{\star}}$ throughout.
		\end{proposition}

	We will prove Proposition \ref{charpord} via a series of Lemmas.  To proceed with the proof, 
	we first fix some notation and ideas.  To begin with, 
	we may and do work over $k:=\o{\F}_p$.  If $X$ is any $\F_p$-scheme, 
	we will abuse notation slightly by 
	again writing $X$ for the base change of $X$ to $k$,
	and we write $F:X\rightarrow X$ for the base change of the absolute Frobenius of
	$X$ over $\F_p$ to $k$.
	
	Let $\nor{\o{\X}}_r\rightarrow \o{\X}_r$ be the
	normalization map; by Proposition \ref{redXr}, we know that 
	$\nor{\o{\X}}_r$ is the disjoint union of proper smooth and irreducible
	Igusa curves $I_{(a,b,u)}=\Ig_{\max(a,b)}$ indexed by triples $(a,b,u)$
	with $a,b$ nonnegative integers satisfying $a+b=r$ and 
	$u\in(\Z/p^{\min(a,b)}\Z)^{\times}$.  Rosenlicht's explicit description
	of the relative dualizing sheaf (see Proposition \ref{Rosenlicht})
	provides a functorial
	injection of $k$-vector spaces
	\begin{equation}
		 \xymatrix{
		 	{H^0(\o{\X}_r,\omega_{\o{\X}_r})}\ar@{^{(}->}[r] & 
			{H^0(\nor{\o{\X}}_r,\underline{\Omega}^1_{k(\nor{\o{\X}}_r)})\simeq 
			\prod\limits_{(a,b,u)}\Omega^1_{k(I_{(a,b,u)})}}
			}\label{dualizing2prod}
	\end{equation}
	with image precisely those elements $(\eta_{(a,b,u)})$ of the product that 
	are holomorphic outside the supersingular points and
	satisfy
	$\sum \res_{x_{(a,b,u)}}(s\eta_{(a,b,u)}) =0$
	for each supersingular point $x\in \o{\X}_r(k)$ and all $s\in \O_{\o{\X}_r,x}$, where
	$x_{(a,b,u)}$ is the unique point of $I_{(a,b,u)}$ lying
	over $x$ and the sum is over all triples $(a,b,u)$ as above.
	We henceforth identify $H^0(\o{\X}_r,\omega_{\o{\X}_r})$ with a subspace of
	the space of meromorphic differential forms on $\nor{\o{\X}}_r$ via
	(\ref{dualizing2prod}), and for any meromorphic differential $\eta$, we denote by
	$\eta_{(a,b,u)}$ the $(a,b,u)$-component of $\eta$.

	The correspondence $U_p:=(\pi_1,\pi_2)$ 
	on $\X_r$ is given by the 
	degeneracy maps $\pi_1,\pi_2:\Y_r\rightrightarrows \X_r$ of (\ref{Upcorr}) and 
	yields, as in (\ref{HeckeDef}), endomorphisms $U_p:=(\pi_1)_*\circ\pi_2^*$ and $U_p^*:=(\pi_2)_*\circ\pi_1^*$
	of $H^0(\X_r,\omega_{\X_r/R_r})$;
	we will again denote by $U_p$ and $U_p^*$ the induced endomorphisms $U_p\otimes 1$ and $U_p^*\otimes 1$ of
	\begin{equation*}
		H^0(\o{\X}_r,\omega_{\o{\X}_r}) \simeq H^0(\X_r,\omega_{\X_r/R_r})\otimes_{R_r} k,
	\end{equation*}
	where the isomorphism is
	the canonical one of Lemma \ref{ReductionCompatibilities} (\ref{BaseChngDiagram}).	
	By the functoriality of normalization,
	we have an induced correspondence $U_p:=(\nor{\o{\pi}}_1,\nor{\o{\pi}}_2)$
	on $\nor{\o{\X}}_r$,
	and we write $U_p$ and $U_p^*$ for the resulting endomorphisms (\ref{HeckeDef})
	of $H^0(\nor{\o{\X}}_r,\underline{\Omega}^1_{k(\nor{\o{\X}}_r)})$.
	By Lemma \ref{ReductionCompatibilities} (\ref{PTBCCompat}), 
	the map (\ref{dualizing2prod}) is then $U_p$ and $U_p^*$-equivariant.  
	The  Hecke correspondences away from $p$
	and the diamond operators act on the source of (\ref{dualizing2prod})
	via ``reduction modulo $p$" and on the target via the induced
	correspondences in the usual way (\ref{HeckeDef}), and 
	the map (\ref{dualizing2prod}) is compatible with these
	actions 
	thanks to Lemma 
	\ref{ReductionCompatibilities} (\ref{PTBCCompat}).
	  Similarly, 
	the semilinear ``geometric inertia" action of $\Gamma$
	on $\X_r$ induces a linear action on $\nor{\o{\X}}_r$ as in Proposition
	\ref{AtkinInertiaCharp}, and the map (\ref{dualizing2prod}) is equivariant
	with respect to these actions. 		
	
	With these preliminaries out of the way, our first task is to describe the action of $U_p$ and $U_p^*$ on the space of meromorphic
	differentials on $\nor{\o{\X}}_r$:
	
	\begin{lemma}\label{Lem:UpOnMero}
	 For {\em any} meromorphic differential $\eta = (\eta_{(a,b,u)})$ on $\nor{\o{\X}}_r$, we have
	\begin{subequations}
	\begin{equation}
			\left({U_p}\eta \right)_{(a,b,u)} = \begin{cases}
				 F_*\eta_{(r,0,1)} &:\quad (a,b,u)=(r,0,1)\\
				\pr_*\eta_{(a+1,b-1,u)} &:\quad 0 < b \le a \\
				\sum\limits_{\substack{u'\in (\Z/p^{a+1}\Z)^{\times} \\ u'\equiv u\bmod p^{a}}} 
				\langle u'\rangle \eta_{(a+1,b-1,u)}
				 &:\quad r\ \text{odd},\ a=b-1\\
				 \sum\limits_{\substack{u'\in (\Z/p^{a+1}\Z)^{\times} \\ u'\equiv u\bmod p^{a}}} 
				\rho^*\langle u'\rangle \eta_{(a+1,b-1,u')}
				 &:\quad r\ \text{even},\ a=b-2\\				 
				\sum\limits_{\substack{u'\in (\Z/p^{a+1}\Z)^{\times}\\ u'\equiv u\bmod p^a}} 
				\pr^*\eta_{(a+1,b-1,u')} &:\quad 0 \le a < b-2\\
			\end{cases}\label{Up1}
	\end{equation}
	and
	\begin{equation}
			\left(U_p^*\eta \right)_{(a,b,u)} = \begin{cases}
				\langle p\rangle_N^{-1}F_*\eta_{(0,r,1)} &:\quad (a,b,u)=(0,r,1)\\
				\pr_*\eta_{(a-1,b+1,u)} &:\quad 0 < a < b \\
				\langle u\rangle^{-1}\pr_*\eta_{(a-1,b+1,u)} &:\quad r\ \text{even},\ b=a\\				 
				 \sum\limits_{\substack{u'\in (\Z/p^{b+1}\Z)^{\times} \\ u'\equiv u\bmod p^{b}}} 
				\langle u'\rangle^{-1} \eta_{(a-1,b+1,u')}
				 &:\quad r\ \text{odd},\ b=a-1\\
				\sum\limits_{\substack{u'\in (\Z/p^{b+1}\Z)^{\times}\\ u'\equiv u\bmod p^b}} 
				\pr^*\eta_{(a-1,b+1,u')} &:\quad 0 \le b < a-1\\
			\end{cases}\label{Up2}
	\end{equation}
	\end{subequations}
\end{lemma}

\begin{proof}
	This is an easy exercise using the definition of $U_p$, the explicit description of 
	the maps $\nor{\o{\pi}}_1$ and $\nor{\o{\pi}}_2$
	given in Proposition \ref{UlmerProp}, and the fact that $F^*$ kills any global 
	meromorphic differential form on a scheme of characteristic $p$.	
\end{proof}

	Observe that for $0 < b \le r$, the $(a,b,u)$-component of 
	${U_p}\eta$ depends only on the $(a+1,b-1,u')$-components of $\eta$ for varying $u'$, and similarly
	for $0<a\le r$ the $(a,b,u)$-component of $U_p^*\eta$ depends only on the $(a-1,b+1,u')$-components 
	of $\eta$.  This fact is crucial for our purposes, and manifests itself clearly in the following:

\begin{corollary}\label{Cor:UpHighPower}
	For any $n\ge r\ge 1$, and any meromorphic differential $\eta = (\eta_{(a,b,u)})$ on $\nor{\o{\X}}_r$, we have
	\begin{subequations}
	\begin{equation}
		\left(U_p^n\eta \right)_{(a,b,u)} = \begin{cases}
				\pr_*^b F_*^{n-b}\eta_{(r,0,1)} &:\quad  b \le a \\
				\sum\limits_{\substack{u'\in (\Z/p^{b}\Z)^{\times}\\ u'\equiv u\bmod p^a}} 
				\langle u'\rangle \pr_*^a F_*^{n-b}\eta_{(r,0,1)} &:\quad a < b 			
				\end{cases}\label{Upn1}
	\end{equation}
	and
	\begin{equation}
		\left({U_p^*}^n\eta \right)_{(a,b,u)} = \begin{cases}
				\pr_*^a\langle p\rangle_N^{a-n}F_*^{n-a}\eta_{(0,r,1)} &:\quad  a < b \\
				\sum\limits_{\substack{u'\in (\Z/p^{a}\Z)^{\times}\\ u'\equiv u\bmod p^b}} 
				\langle u'\rangle^{-1} \pr_*^b\langle p\rangle_N^{a-n} F_*^{n-a}\eta_{(0,r,1)} &:\quad b \le a				
				\end{cases}\label{Upn2}
	\end{equation}
	\end{subequations}
\end{corollary}

\begin{proof}
	This follows immediately from Lemma \ref{Lem:UpOnMero} by induction.
\end{proof}	

\begin{lemma}\label{Cor:PBisWD}
	Pullback of meromorphic differential forms
	along  $i_r^{\star}:I_r^{\star}\hookrightarrow \nor{\o{\X}}_r$ for $\star=\infty$ $($respectively $\star=0$$)$
	carries $e_rH^0(\o{\X}_r,\omega_{\o{\X}_r})$ $($respectively $e_r^*H^0(\o{\X}_r,\omega_{\o{\X}_r})$$)$
	into $H^0(I_r^{\star},\Omega^1_{I_r^{\star}}(\SS))^{V_{\ord}}$.  Moreover, the resulting maps
	$(\ref{I'compIsom})$ intertwine $U_p$ $($respectively $U_p^*$$)$ on source with $F_*$ $($respctively
	$\langle p\rangle_N^{-1} F_*$$)$ on target.
\end{lemma}

\begin{proof}
	We will treat the case of $\star=\infty$; the case $\star=0$ is similar.
 	We first observe that pullback of meromorphic differentials along 
		$i_r^{\infty}:I_r^{\infty}\hookrightarrow \nor{\o{\X}}_r$ is given by projection
		\begin{equation}
			\xymatrix@C=35pt{
			{H^0(\nor{\o{\X}}_r,\underline{\Omega}^1_{k(\nor{\o{\X}}_r)})\simeq 
			\prod\limits_{(a,b,u)} H^0(I_{(a,b,u)},\u{\Omega}^1_{k(I_{(a,b,u)})})} \ar[r]^-{\proj_{(r,0,1)}} & 
			{H^0(I_r^{\infty},\u{\Omega}^1_{k(I_r^{\infty})})}
			}\label{PBisProj}
		\end{equation}
		onto the $(r,0,1)$-component. 
		Thanks to Rosenlicht's description of the image of (\ref{dualizing2prod}), 
		any section $\eta=(\eta_{(a,b,u)})$ of $H^0(\o{\X}_r,\omega_{\o{\X}_r})$
		has the property that $\eta_{(a,b,u)}$ is holomorphic outside $\SS$ and, 
		due to \cite[Lemma 5.2.2]{GDBC}, has poles along $\SS$ of order bounded
		by a constant $N_r$ depending only on $r$.
		Thus, the composition
		of pullback along $i_r^{\infty}$ with (\ref{dualizing2prod}) induces a map
		\begin{equation}
			\xymatrix{
				{(i_r^{\infty})^*:e_rH^0(\o{\X}_r,\omega_{\o{\X}_r})} \ar[r] & H^0(I_r^{\infty},\Omega^1_{I_r^{\infty}}(N_r\cdot\SS))
				}\label{Map:irinfpullback}
		\end{equation}	
		which moreover satisfies $(i_r^{\infty})^*\circ U_p = F_* \circ (i_r^{\infty})^*$ 
		thanks to the description (\ref{Up1}) of Lemma \ref{Lem:UpOnMero}.
		Since $U_p$ is invertible on the source of (\ref{Map:irinfpullback}),
		it follows at once from this intertwining relation that $F_*$ is invertible on its image, and we conclude from Lemma \ref{sspoles} (\ref{VControl})
		that (\ref{Map:irinfpullback}), has image contained in $H^0(I_r^{\infty},\Omega^1_{I_r^{\infty}}(\SS))^{V_{\ord}}$,
		as claimed.
\end{proof}

In order to prove that the pullback maps (\ref{I'compIsom}) resulting from Lemma \ref{Cor:PBisWD}
are isomorphisms, we will construct the inverse mappings.
For any $r>0$ and for $\star=\infty, 0$ we now define maps
$		\xymatrix@1{
			{\gamma_r^{\star}: H^0(I_r^{\star},\Omega^1_{I_r^{\star}}(\SS))^{V_{\ord}}} \ar[r] & 
			{H^0(\nor{\o{\X}}_r,\underline{\Omega}^1_{k(\nor{\o{\X}}_r)})}
			}
$
	by
	\begin{subequations}
		\begin{equation}
			(\gamma_{r}^{\infty}(\nu))_{(a,b,u)}: = \begin{cases}
				\pr_*^bF_*^{-b}\nu &:\quad  b\le a\\
		\sum\limits_{\substack{u'\in (\Z/p^{b}\Z)^{\times}\\ u'\equiv u\bmod p^a}} 
				\langle u'\rangle \pr_*^a F_*^{-b}\nu &:\quad a < b \\
				\end{cases}\label{betadef}
		\end{equation}
		and
		\begin{equation}
		(\gamma_{r}^{0}(\nu))_{(a,b,u)}: = \begin{cases}
				\pr_*^a\langle p\rangle_N^{a}F_*^{-a}\nu &:\quad  a < b\\
		\sum\limits_{\substack{u'\in (\Z/p^{a}\Z)^{\times}\\ u'\equiv u\bmod p^b}} 
				\langle u'\rangle^{-1} \pr_*^b \langle p\rangle_N^{a} F_*^{-a}\nu_{(0,r,1)} &:\quad b \le a				

				\end{cases}\label{gammadef}
		\end{equation}
	\end{subequations}
	These maps are well-defined because $F_*=V$ is invertible on the $V$-ordinary subspace.
	
	\begin{lemma}\label{Lem:gammaMaps}
	The maps $\gamma_r^{\star}$ 
	satisfy 
	$U_p\circ \gamma_r^{\infty}=\gamma_r^{\infty}\circ F_*$  and $U_p^*\circ \gamma_r^{0}=\gamma_r^0\circ \langle p\rangle_N^{-1}F_*$.
	Via $(\ref{dualizing2prod})$, the map $\gamma_r^{\infty}$ $($respectively $\gamma_r^0$$)$ has image contained in 
	$e_rH^0(\o{\X}_r,\omega_{\o{\X}_r})$ $($respectively $e_r^*H^0(\o{\X}_r,\omega_{\o{\X}_r})$$)$.
	\end{lemma}	

\begin{proof}
	The asserted intertwining relations follow readily from Lemma \ref{Lem:UpOnMero} 
	and the definitions (\ref{betadef})--(\ref{gammadef}).
	Since $F_*$ is invertible on the source of $\gamma_r^{\star}$,
	to complete the proof it then suffices to show that $\gamma_r^{\star}$
	has image in $H^0(\o{\X}_r,\omega_{\o{\X}_r})$. 
	
	To do this, we again appeal to the explicit description of the image of (\ref{dualizing2prod}) afforded by Rosenlicht's theory,
	and must check that for		
	any supersingular point $x$ on $\o{\X}_r$, and any local section $s\in \O_{\o{\X}_r,x}$,
	the sum of the residues of $s\gamma^{\star}(\nu)$
	at all $k$-points of $\nor{\o{\X}}_r$ lying over $x$ is zero.  
	Using (\ref{betadef}), we calculate that for $\star=\infty$ this sum is equal to
	\begin{align}
		\sum_{b \le a} 
		\sum_{u\in (\Z/p^b\Z)^{\times}}\res_{x_{(a,b,u)}}(s\pr_*^bF_*^{-b}\nu)
		+\sum_{a < b} \sum_{u\in (\Z/p^b\Z)^{\times}}
		\res_{x_{(a,b,u)}}(s\langle u\rangle \pr_*^aF_*^{-b}\nu)
	\label{residuecalc1}
	\end{align}
	where $x_{(a,b,u)}$ denotes the unique point of the $(a,b,u)$-component of $\nor{\o{\X}}_r$ over $x$,
	and the outer sums range over all nonnegative integers $a,b$ with $a+b=r$.
	We claim that for any meromorphic differential $\omega$ on $I_{(a,b,u)}$ and any 
	supersingular point $y$ of $I_{(a,b,u)}$ over $x$, we have
	\begin{subequations}
	\begin{equation}
		 \res_y(\omega) = \res_y(\langle v \rangle\omega)\label{DiamondFix}
	\end{equation}
	for all $v\in \Z_p^{\times}$, and, if in addition $\omega$ is $V$-ordinary,
	\begin{equation}
		\res_y(s\omega) = s(x)\res_y(\omega)\label{FunctionPullOut}
	\end{equation}
	\end{subequations}
	Indeed, (\ref{DiamondFix}) is a consequence of (\ref{TateFormula}),
	using the fact that the automorphism
	$\langle v\rangle$ of $I_{(a,b,u)}$ fixes every supersingular point, 
	while (\ref{FunctionPullOut}) is deduced by thinking about formal expansions of differentials
	at $y$ ({\em cf.} Proposition \ref{CartierOp} (\ref{Vformula}) and (\ref{CartierResidue}))
	and using the fact that a $V$-ordinary meromorphic differential has at worst simple poles 
	thanks to Lemma \ref{sspoles}.  Via (\ref{DiamondFix})--(\ref{FunctionPullOut}), we reduce the
	sum (\ref{residuecalc1}) to
	\begin{align}
		\sum_{a + b = r} 
		\sum_{u\in (\Z/p^b\Z)^{\times}} s(x)\res_{x_{(a,b,u)}}(\pr_*^{\min(a,b)}F_*^{-b}\nu)
		&= \sum_{a + b = r} 
		\varphi(p^b)s(x)\res_{x_{(a,b,1)}}(\pr_*^{\min(a,b)}F_*^{-b}\nu)\nonumber\\
		&=s(x)\res_{x_{(r,0,1)}}(\nu) - s(x)\res_{x_{(r,0,1)}}(F_*^{-1}\nu)
		\label{twoterm}
	\end{align}
	where the first equality above follows from the fact
	that for {\em fixed} $a,b$, the points $x_{(a,b,u)}$ for varying $u\in (\Z/p^{\min(a,b)}\Z)^{\times}$
	are all identified with the {\em same} point on $\Ig(p^{\max(a,b)})$, and the second 
	equality is a consequence of (\ref{TateFormula}), since $\rho:I_{(r,0,1)}\rightarrow I_{(r-1,0,1)}$
	is generically \'etale and carries the point $x_{(r,0,1)}$ to $x_{(r-1,0,1)}$.
	As $\nu$ is $V$-ordinary, there exists a $V$-ordinary meromorphic differential $\xi$
	on $I_r^{\infty}$ with $\nu=F_*\xi$; substituting this expression for $\nu$ into (\ref{twoterm})
	and applying Proposition \ref{CartierOp} (\ref{CartierResidue}), we conclude that (\ref{twoterm}) is zero,
	as desired.
	That $\gamma_r^{0}$ has image in $H^0(\o{\X}_r,\omega_{\o{\X}_r/k})$ 
	follows from a nearly identical calculation, and we omit the details.		
\end{proof}

From Lemmas \ref{Cor:PBisWD} and \ref{Lem:gammaMaps} we thus obtain natural maps
\begin{subequations}
\begin{equation}
	\xymatrix{
			{H^0(I_r^{\infty},\Omega^1_{I_r^{\infty}}(\SS))^{V_{\ord}}} \ar@<0.5ex>[r]^-{\gamma_r^{\infty}} & 
			\ar@<0.5ex>[l]^-{(i_r^{\infty})^*}{e_rH^0(\o{\X}_r,\omega_{\o{\X}_r})}
			}\label{Map:gammaInf}
\end{equation}	
and
\begin{equation}	
	\xymatrix{
			{H^0(I_r^{0},\Omega^1_{I_r^{0}}(\SS))^{V_{\ord}}} \ar@<0.5ex>[r]^-{\gamma_r^{0}} & 
			\ar@<0.5ex>[l]^-{(i_r^0)^*}{e_r^*H^0(\o{\X}_r,\omega_{\o{\X}_r})}
			}\label{Map:gamma0}
\end{equation}
\end{subequations}
intertwining $F_*$ and $\langle p\rangle_N^{-1} F_*$ on the left with $U_p$ and $U_p^*$ on the right, respectively.

\begin{lemma}\label{Lemma:InverseIsoms}
	The maps $(\ref{Map:gammaInf})$ and $(\ref{Map:gamma0})$ are inverse isomorphisms.	
\end{lemma}		

\begin{proof}
	From the very definitions (\ref{betadef})
	of $\gamma_r^{\infty}$ and the description (\ref{PBisProj}) of pullback as projection,
	it is clear that $(i_r^{\infty})^*\circ \gamma_r^{\infty}=\id$ on $H^0(I_r^{\infty},\Omega^1_{I_r^{\infty}})^{V_{\ord}}$.
	Since $U_p$ 
	is an isomorphism on $e_rH^0(\o{\X}_r,\omega_{\o{\X}_r})$
	to prove that the composition in the other
	order is the identity, it suffices to prove 
	that $U_p^n \circ \gamma_r^{\infty}\circ (i_r^{\infty})^* = U_p^n$ 
	on $e_rH^0(\o{\X}_r,\omega_{\o{\X}_r})$
	for any $n>r$.
	Using the intertwining relations established by Lemmas \ref{Cor:PBisWD} and \ref{Lem:gammaMaps},
	this amounts to proving that 
	$\gamma_r^{\infty} \circ (i_r^{\infty})^* \circ U_p^n = U_p^n$
	for any $n >r$, which is clear thanks to Corollary \ref{Cor:UpHighPower},
	the description (\ref{PBisProj}), and the definition  (\ref{betadef}).
	The proof that (\ref{Map:gamma0}) are inverse isimorphisms is similar,
	and is left to the reader.
\end{proof}		
	
		As discussed after Definition \ref{Def:HeckeModuliProblem},
		$i_r^{\infty}$ is compatible with the $\H_r$-correspondences,
		and it follows that 
		the isomorphism (\ref{I'compIsom}) resulting from Lemma \ref{Lemma:InverseIsoms}
		is $\H_r$-equivariant (with $U_p$ acting on the target via $F_*$).  Similarly, 
		since the ``geometric inertia" action (\ref{gammamaps}) of 
		$\Gamma$ on $\X_r$ is compatible via $i_r^{\infty}$ with the trivial action
		on $I_r^{\infty}$ by Proposition \ref{AtkinInertiaCharp},		
		the isomorphism (\ref{I'compIsom}) is equivariant for these actions of $\Gamma$.
		A nearly identical analysis shows that $(i_r^{0})^*$ is equivariant
		for the actions of $T_{\ell}^*$, $U_{\ell}^*$ (with $U_p^*$ acting on the target 
		as $\langle p\rangle_N^{-1} F_*$) and $\langle v\rangle_N^*$,
		and that it intertwines $\langle u\rangle^*$ on source with $\langle u\rangle$
		on target ({\em cf.} the preliminary discussion above Proposition \ref{charpord}).
		Likewise, $(i_r^0)^*$ is $\Gamma$-equivariant for the
		action of $\Gamma$ on $I_r^{0}$ via $\langle \chi(\cdot)\rangle^{-1}$.
		The commutativity of
		the four diagrams in (\ref{rCompatDiagrams}) is an immediate consequence
		of the descriptions of the degeneracy mappings $\o{\pr},\o{\ps}$ on $\nor{\o{\X}}_r$ furnished by Proposition
		\ref{pr1desc} and the explication (\ref{PBisProj}) of pullback by $i_r^{\star}$ in terms
		of projection.  
		
		Finally, that the assertions of Proposition \ref{charpord} all hold if $\o{\X}_r$ and
		 $\Omega^1_{I_r^{\star}}(\SS)$
		are replaced by $\nor{\o{\X}}_r$ and $\Omega^1_{I_r^{\star}}$, respectively, follows from a
	 	similar---but much simpler---argument.  The point is that the maps $\gamma_r^{\star}$ 
		of (\ref{betadef})--(\ref{gammadef})
		visibly carry $H^0(I_r^{\star},\Omega^1_{I_r^{\star}})^{V_{\ord}}$ into 
		$H^0(\nor{\o{\X}}_r,\Omega_{\nor{\o{\X}}_r}^1)$, from which it follows 
		via our argument that they induce the claimed isomorphisms.	
		This completes the proof of Proposition \ref{charpord}.

Since $\o{\X}_r$ is a proper and geometrically connected curve over $\F_p$, Proposition \ref{HodgeFilCrvk} 
(\ref{HodgeDegenerationField}) provides short exact sequences of 
$\F_p[\Delta/\Delta_r]$-modules with linear $\Gamma$ and $\H_r^*$ (respectively $\H_r$)-actions
\begin{subequations}
\begin{equation}
	\xymatrix{
		0\ar[r] & {e_r^*H^0(\o{\X}_r,\omega_{\o{\X}_r/\F_p})} \ar[r] & {e_r^*H^1(\o{\X}_r/\F_p)} \ar[r] &
		{e_r^*H^1(\o{\X}_r,\O_{\o{\X}_r})} \ar[r] & 0
	}\label{sesincharp1}
\end{equation}
respectively
\begin{equation}
	\xymatrix{
		0\ar[r] & {e_rH^0(\o{\X}_r,\omega_{\o{\X}_r/\F_p})} \ar[r] & {e_rH^1(\o{\X}_r/\F_p)} \ar[r] &
		{e_rH^1(\o{\X}_r,\O_{\o{\X}_r})} \ar[r] & 0
	}\label{sesincharp2}
\end{equation}
\end{subequations}
which are canonically $\F_p$-linearly dual to each other.  We likewise have such exact sequences
in the case of $\nor{\o{\X}}_r$; note that since $\nor{\o{\X}}_r$ is smooth,
the short exact sequence $H(\nor{\o{\X}}_r/\F_p)$ is simply the Hodge filtration
of $H^1_{\dR}(\nor{\o{\X}}_r/\F_p)$.

In what follows, we give $H^1(I_r^{\star},\O(-\SS))$
the structure of $\F_p[\Delta/\Delta_r]$-module 
by decreeing that $u\in \Delta$ act
through the {\em adjoint} diamond operator $\langle u\rangle^*=\langle u^{-1}\rangle$.
The canonical Grothendieck--Serre duality between 
$H^1(I_r^{\star},\O(-\SS))$ and $H^0(I_r^{\star},\Omega^1_{I_r^{\star}}(\SS))$,
then identifies the former with the contragredient of the latter; 
{\em cf.} the discussion preceeding Proposition \ref{charpord}.

\begin{corollary}\label{SplitIgusa}
		The absolute Frobenius morphism of $\o{\X}_r$ over $\F_p$ induces a natural
		$\F_p[\Delta/\Delta_r]$-linear, $\Gamma$-compatible,
		and $\H_r^*$ (respectively $\H_r$) equivariant splitting of $(\ref{sesincharp1})$
		$($respectively $(\ref{sesincharp2})$$)$.
		For each $r$ there are natural isomorphisms of split short exact sequences 
		of $\F_p[\Delta/\Delta_r]$-modules
		\begin{subequations}
		\begin{equation}
			\xymatrix@C=15pt{
				0\ar[r] & {e_r^*H^0(\o{\X}_r,\omega_{\o{\X}_r/\F_p})} \ar[r]\ar[d]_-{F_*^r (i_r^0)^*}^-{\simeq} & 
				{e_r^*H^1(\o{\X}_r/\F_p)} \ar[r]\ar[d]^-{\simeq} &
				{e_r^*H^1(\o{\X}_r,\O_{\o{\X}_r})} \ar[r] & 0\\
				0 \ar[r] & {H^0(I_r^{0},\Omega^1(\SS))^{V_{\ord}}} \ar[r] &
				{H^0(I_r^{0},\Omega^1(\SS))^{V_{\ord}}\oplus
				H^1(I_r^{\infty},\O(-\SS))^{F_{\ord}}} \ar[r] & 
				{H^1(I_r^{\infty},\O(-\SS))^{F_{\ord}}} \ar[r]\ar[u]_-{\VDual{(i_r^{\infty})^*}}^-{\simeq} & 0
			}\label{LowerIsom1}
		\end{equation}
		\begin{equation}
			\xymatrix@C=15pt{
				0\ar[r] & {e_rH^0(\o{\X}_r,\omega_{\o{\X}_r/\F_p})} 
				\ar[r]\ar[d]_-{F_*^r (i_r^{\infty})^*}^-{\simeq} & 
				{e_rH^1(\o{\X}_r/\F_p)} \ar[r]\ar[d]^-{\simeq} &
				{e_rH^1(\o{\X}_r,\O_{\o{\X}_r})} \ar[r] & 0\\
				0 \ar[r] & {H^0(I_r^{\infty},\Omega^1(\SS))^{V_{\ord}}} \ar[r] &
				{H^0(I_r^{\infty},\Omega^1(\SS))^{V_{\ord}}\oplus
				H^1(I_r^{0},\O(-\SS))^{F_{\ord}}} \ar[r] & 
				{H^1(I_r^{0},\O(-\SS))^{F_{\ord}}} \ar[r]
				\ar[u]_-{\VDual{(i_r^{0})^*}\langle p\rangle_N^{-r}}^-{\simeq} & 0
			}\label{UpperIsom1}
		\end{equation}
		\end{subequations}
		that are compatible with the actions of $\Gamma$
		and of the Hecke operators 
		$T_{\ell}^*$, $U_{\ell}^*$ and $\langle v\rangle_N^*$
		$($respectively $T_{\ell}^*$, $U_{\ell}^*$ and $\langle v\rangle_N^*$$)$.
	The identification
		$(\ref{LowerIsom1})$  $($respectively $(\ref{UpperIsom1})$$)$ is moreover
		compatible with change in $r$ using the trace mappings attached
		to $\pr: I_r^{\star}\rightarrow I_{r-1}^{\star}$ and to 
		$\o{\pr}:\o{\X}_r\rightarrow \o{\X}_{r-1}$ $($respectively 
		$\o{\ps}:\o{\X}_r\rightarrow \o{\X}_{r-1}$$)$.
		The same statements hold true if we replace $\o{\X}_r$, $\Omega^1_{I_r^{\star}}(\SS)$,
		and $\O_{I_r^{\star}}(-\SS)$ with $\nor{\o{\X}}_r$, $\Omega^1_{I_r^{\star}}$, 
		and $\O_{I_r^{\star}}$, respectively.
\end{corollary}

\begin{proof}
		Pullback by the absolute Frobenius endomorphism of $\o{\X}_r$
		induces an endomorphism of (\ref{sesincharp1}) which 
		kills $H^0(\o{\X}_r,\omega_{\o{\X}_r/\F_p})$
		and so yields a morphism of $\F_p[\Delta/\Delta_r]$-modules
		\begin{equation}
			\xymatrix{
				{e_r^*H^1(\o{\X}_r,\O_{\o{\X}_r})}\ar[r] & {{e_r^*}H^1(\o{\X}_r/\F_p)} 
			}\label{H1Splitting}
		\end{equation}
		that 
		is $\Gamma$ and $\H_r^*$-compatible 
		and projects to the endomorphism $F^*$ of $e_r^*H^1(\o{\X}_r,\O_{\o{\X}_r})$.
		On the other hand, Proposition \ref{charpord} and Grothendieck--Serre duality (see Remark \ref{DualityOfFVOrd})
		 give a natural $\Gamma$ and $\H_r^*$-equivariant 
		isomorphism of $\F_p[\Delta/\Delta_r]$-modules
		\begin{equation}
			\xymatrix{
				{H^1(I_r^{\infty},\O_{I_r^{\infty}}(-\SS))^{F_{\ord}}}\ar[r]^-{\VDual{(i_r^{\infty})^*}} &  
				{e_r^*H^1(\o{\X}_r,\O_{\o{\X}_r})}
			}.\label{IsomOnH1}
		\end{equation}
		As this isomorphism intertwines $F^*$ on source and target, 
		we deduce that $F^*$ acts invertibly on 
		${e_r^*H^1(\o{\X}_r,\O_{\o{\X}_r})}$. We may therefore pre-compose (\ref{H1Splitting}) with $(F^*)^{-1}$
		to obtain a canonical splitting of (\ref{sesincharp1}), which by construction is
		$\F_p[\Delta/\Delta_r]$-linear and compatible with $\Gamma$ and $\H_r^*$.  
		The existence of (\ref{LowerIsom1}) as well
		as its asserted compatibility with $\Gamma$ and adjoint Hecke operators
		 and with change in $r$ now follows
		immediately from Proposition \ref{charpord} and duality ({\em cf.} again Remark \ref{DualityOfFVOrd}).
		The corresponding assertions for the exact sequence (\ref{sesincharp2}) and the 
		diagram (\ref{UpperIsom1}) are proved similarly, and we leave the details to the reader.
		A nearly identical argument shows that the same assertions hold true when $\o{\X}_r$, 
		$\Omega^1_{I_r^{\star}}(\SS)$,
		and $\O_{I_r^{\star}}(-\SS)$ are replaced by $\nor{\o{\X}}_r$,
		 $\Omega^1_{I_r^{\star}}$, and $\O_{I_r^{\star}}$, respectively.
\end{proof}

\begin{corollary}\label{FreenessInCharp}
	The exact sequences
	$(\ref{sesincharp1})$ and $(\ref{sesincharp2})$ are split short exact sequences of 
	free $\F_p[\Delta_1/\Delta_r]$-modules 
	whose terms have $\F_p[\Delta_1/\Delta_r]$-ranks
	$d$, $2d$, and $d$, respectively, for $d$ as in 
	Remark $\ref{dMFmeaning}$.
	For $s\le r$, the degeneracy maps $\pr,\ps:\X_r\rightrightarrows \X_s$ 
	induce natural isomorphisms of
	exact sequences
	\begin{align*}
		&\xymatrix{
			{\pr_*:e_r^*H(\o{\X}_r/\F_p) \tens_{\F_p[\Delta_1/\Delta_s]} \F_p[\Delta_1/\Delta_r]} \ar[r]^-{\simeq} & 
			{e_s^*H(\o{\X}_s/\F_p)}
			}\\
		&\xymatrix{
			{\ps_*:e_rH(\o{\X}_r/\F_p) \tens_{\F_p[\Delta_1/\Delta_s]} \F_p[\Delta_1/\Delta_r]} \ar[r]^-{\simeq} & 
			{e_sH(\o{\X}_s/\F_p)}
			}	
	\end{align*}
	that are $\Gamma$ and $\H_r^*$ $($respectively $\H_r$$)$ equivariant.
\end{corollary}

\begin{proof}
	This follows immediately from Proposition \ref{IgusaStructure} and Corollary \ref{SplitIgusa}.
\end{proof}

To formulate an analogue of Corollary \ref{FreenessInCharp}
in the case of $\nor{\o{\X}}_r$, we proceed as follows.
Denote by $\tau:\F_p^{\times} \rightarrow \Z_p^{\times}$ the Teichm\"uller character,
and for any $\Z_p[\F_p^{\times}]$-module $M$ 
and any $j\in \Z/(p-1)\Z$, let
\begin{equation*}
	M^{(j)}:=\{m\in M\ :\ d\cdot m = \tau(d)^jm\ \text{for all}\ d\in \F_p^{\times}\}
\end{equation*}
be the subspace of $M$ on which $\F_p^{\times}$ acts via $\tau^j$. 
As $\#\F_p^{\times}=p-1$ is a unit in $\Z_p^{\times}$, the submodule $M^{(j)}$
is a direct summand of $M$.
Explicitly, the idenitity of $\Z_p[\F_p^{\times}]$ admits the  decomposition
\begin{equation}
	1 = \sum_{j\in \Z/(p-1)\Z} f_j\quad\text{with}\quad 
	f_j:=\frac{1}{p-1}\sum_{g\in \F_p^{\times}} \tau^{-j}(g)\cdot g
	\label{GpRngIdem}
\end{equation}
into mutually orthogonal idempotents $f_j$, and we have $M^{(j)}=f_jM$.
In applications, we will consistently need to remove the trivial eigenspace
$M^{(0)}$ from $M$, as this eigenspace in the $p$-adic Galois representations 
we consider is not potentially crystalline at $p$.  We will write
\begin{equation}
	f':=\sum_{\substack{j\in \Z/(p-1)\Z \\ j\neq 0}} f_j\label{TeichmullerIdempotent}
\end{equation}
for the idempotent of $\Z_p[\F_p^{\times}]$ corresponding to projection away
from the 0-eigenspace for $\F_p^{\times}$.  

Applying these considerations
to the identifications of split exact sequences in Corollary \ref{SplitIgusa}, which
are compatible with the given diamond operator action of $\Delta=\Z_p^{\times}\simeq\F_p^{\times}\times \Delta_1$ 
on both rows, we obtain a corresponding identification of split exact sequences
of $\tau^j$-eigenspaces, for each $j\bmod p-1$.
The following is a generalization of \cite[Proposition 8.10 (2)]{tameness}:

\begin{lemma}\label{CharacterSpaces}
	Let $j$ be an integer with $j\not\equiv 0\bmod p-1$.  For each $r$, there are canonical isomorphisms
	\begin{equation}
		\xymatrix{
			{H^0(I_r^{\star},\Omega^1_{I_r^{\star}})^{(j)}}\ar[r]^-{\simeq} & {H^0(I_r^{\star},\Omega^1_{I_r^{\star}}(\SS))^{(j)}}
		}\qquad\text{and}\qquad
		\xymatrix{
			{H^1(I_r^{\star},\O(-\SS))^{(j)}}\ar[r]^-{\simeq} & {H^1(I_r^{\star},\O)^{(j)}}
		}\label{IgusaEigen}
	\end{equation}
	The normalization map 
	$\nu:\nor{\o{\X}}_r\rightarrow \o{\X}_r$ induces a natural isomorphism of split
	exact sequences
	\begin{equation}
	\begin{gathered}
		\xymatrix{
				0\ar[r] & {{e_r^*}H^0(\o{\X}_r,\Omega^1_{\nor{\o{\X}}_r})^{(j)}} \ar[r]\ar[d]_-{\nu_*}^-{\simeq} & 
				{{e_r^*}H^1_{\dR}(\nor{\o{\X}}_r/\F_p)^{(j)}} \ar[r]\ar[d]^-{\simeq} &
				{{e_r^*}H^1(\nor{\o{\X}}_r,\O_{\nor{\o{\X}}_r})^{(j)}} \ar[r] & 0		\\
				0\ar[r] & {e_r^*H^0(\o{\X}_r,\omega_{\o{\X}_r/\F_p})^{(j)}} \ar[r] & 
				{e_r^*H^1(\o{\X}_r/\F_p)^{(j)}} \ar[r] &
				{e_r^*H^1(\o{\X}_r,\O_{\o{\X}_r})^{(j)}} \ar[r]\ar[u]_-{\nu^*}^-{\simeq} & 0			
		}\label{splitholo}
	\end{gathered}
	\end{equation}
	where the central vertical arrow is deduced from the outer two vertical arrows via the
	splitting of both rows by the Frobenius endomorphism.
	The same assertions hold if we replace $e_r^*$ with $e_r$.	
\end{lemma}

\begin{proof}
	The first map in (\ref{IgusaEigen}) is injective, as it is simply the canonical inclusion.
	To see that it is an isomorphism, 
	we may work over $k:=\o{\F}_p$.  If $\eta$ is {\em any} meromorphic differential on 
	$I_r^{\star}$ on which $\F_p^{\times}$
	acts via the character $\tau^j$, 
	then since the diamond operators fix every supersingular point on $I_r^{\star}$ we have
	\begin{equation*}
		\res_x(\eta) = \res_x(\langle u\rangle\eta) = \tau^j(u)\res_x(\eta)
	\end{equation*}
	for any $x\in \SS(k)$ and all $u\in \F_p^{\times}$ thanks to Remark \ref{poletrace}.  As
	$j\not\equiv 0\bmod p-1$, so $\tau^j$ is nontrivial, we must therefore have $\res_x(\eta)=0$ for all 
	supersingular points $x$.  If in addition $\eta$ is holomorphic
	outside $\SS$ with at worst simple poles along $\SS$, then $\eta$ must be holomorphic everywhere, 
	so the first map in (\ref{IgusaEigen}) is
	surjective, as desired. As duality interchanges $\langle u\rangle$ with $\langle u\rangle^*$,
	the second mapping in (\ref{IgusaEigen}) is dual to the first
	via our conventions on the $\F_p[\Delta/\Delta_r]$-structures
	of $H^0(I_r^{\star},\Omega^1_{I_r^{\star}}(\SS))$ and $H^1(I_r^{\star},\O(-\SS))$, and is therefore
	an isomorphism as well.

	Now for each $j\not\equiv 0\bmod p-1$,
	we have a commutative diagram
	\begin{equation}
	\begin{gathered}
		\xymatrix{
			{e_r^*H^0(\nor{\o{\X}}_r,\Omega^1_{\nor{\o{\X}}_r})^{(j)}} 
			\ar@{^{(}->}[r]^-{\nu_*}\ar[d]_-{(i_r^0)^*}^-{\simeq} &
			{e_r^*H^0(\o{\X}_r,\omega_{\o{\X}_r})^{(j)}} \ar[d]^-{(i_r^0)^*}_-{\simeq} \\
			{{H^0(I_r^0,\Omega^1_{I_r^0})^{(j)}}^{V_{\ord}}} \ar@{^{(}->}[r]^-{\simeq} & 
			{{H^0(I_r^0,\Omega^1_{I_r^0}(\SS))^{(j)}}^{V_{\ord}}}
		}
		\label{linkingDiagram}
	\end{gathered}
	\end{equation}
	of $\F_p[\Delta_1/\Delta_r]$-modules 
	in which
	the two vertical arrows are isomorphisms by Proposition \ref{charpord} and the bottom horizontal mapping is 
	an isomorphism as we have just seen.  We conclude that the top horizontal
	arrow of (\ref{linkingDiagram}) is an isomorphism as well.    
	Thus, the left vertical map in (\ref{splitholo}) is an isomorphism.
	A similar argument shows that the analogue of (\ref{linkingDiagram})
	with $e_r$ in place of $e_r^*$ and $I_r^{\infty}$, $i_r^{\infty}$
	in place of $I_r^0$ and $i_r^0$, respectively, is likewise a
	commutative diagram of isomorphisms.  Dualizing 
	this analogue then shows that
	the right vertical map of (\ref{splitholo}) is an isomorphism as well.
	 The diagram (\ref{splitholo}) then follows at once
	from the fact the both rows are canonically split by the Frobenius endomorphism, thanks to 
	Corollary \ref{SplitIgusa}. 	
	A nearly identical argument shows that the same assertions hold if we replace $e_r^*$
	with $e_r$ throughout.
	\end{proof}

If $A$ is any $\Z_p[\F_p^{\times}]$-algebra and $a\in A$, we will 
write $a':=f'a$ for the product of $a$ with the idempotent $f'$
of (\ref{TeichmullerIdempotent}), or equivalently the 
projection of $a$ to the complement 
of the trivial eigenspace for $\F_p^{\times}$.  We will apply this
to $A=\H_r,\,\H_r^*$, viewed as $\Z_p[\F_p^{\times}]$-algebras
via the diamond and adjoint diamond operators, respectively.

\begin{proposition}\label{NormalizationCoh}
	For each $r$ there are isomorphisms of split short exact sequences
	of $\F_p[\Delta_1/\Delta_r]$-modules
	\begin{subequations}
		\begin{equation}
		\begin{gathered}
			\xymatrix@C=15pt{
				0\ar[r] & {{e_r^*}'H^0(\o{\X}_r,\Omega^1_{\nor{\o{\X}}_r})} 
				\ar[r]\ar[d]_-{F_*^r (i_r^0)^*}^-{\simeq} & 
				{{e_r^*}'H^1_{\dR}(\nor{\o{\X}}_r/\F_p)} \ar[r]\ar[d]^-{\simeq} &
				{{e_r^*}'H^1(\nor{\o{\X}}_r,\O_{\nor{\o{\X}}_r})} \ar[r] & 0\\
				0 \ar[r] & {f'H^0(I_r^{0},\Omega^1)^{V_{\ord}}} \ar[r] &
				{f'H^0(I_r^{0},\Omega^1)^{V_{\ord}}\oplus
				f'H^1(I_r^{\infty},\O)^{F_{\ord}}} \ar[r] & 
				{f'H^1(I_r^{\infty},\O)^{F_{\ord}}} \ar[r]\ar[u]_-{\VDual{(i_r^{\infty})^*}}^-{\simeq} & 0
			}\label{LowerIsom2}
		\end{gathered}
		\end{equation}
		\begin{equation}
		\begin{gathered}
			\xymatrix@C=15pt{
				0\ar[r] & {e_r'H^0(\nor{\o{\X}}_r,\Omega^1_{\nor{\o{\X}}_r})} 
				\ar[r]\ar[d]_-{F_*^r (i_r^{\infty})^*}^-{\simeq} & 
				{e_r'H^1_{\dR}(\nor{\o{\X}}_r/\F_p)} \ar[r]\ar[d]^-{\simeq} &
				{e_r'H^1(\nor{\o{\X}}_r,\O_{\nor{\o{\X}}_r})} \ar[r] & 0\\
				0 \ar[r] & {f'H^0(I_r^{\infty},\Omega^1)^{V_{\ord}}} \ar[r] &
				{f'H^0(I_r^{\infty},\Omega^1)^{V_{\ord}}\oplus
				f'H^1(I_r^{0},\O)^{F_{\ord}}} \ar[r] & 
				{f'H^1(I_r^{0},\O)^{F_{\ord}}} \ar[r]
				\ar[u]_-{\VDual{(i_r^{0})^*}\langle p\rangle_N^{-r}}^-{\simeq} & 0
			}\label{UpperIsom2}
		\end{gathered}
		\end{equation}
		\end{subequations}
		Setting $d':=\sum_{k=3}^p d_k$ where $d_k:=\dim_{\F_p} S_k(\Upgamma_1(N);\F_p)^{\ord}$
		as in Remark $\ref{dMFmeaning}$,
		the terms in the top rows of $(\ref{LowerIsom2})$ and $(\ref{UpperIsom2})$
		are free over $\F_p[\Delta_1/\Delta_r]$ of ranks $d'$, $2d'$, and $d'$.		
		The identification
		$(\ref{LowerIsom2})$  $($respectively $(\ref{UpperIsom2})$$)$ is
		equivariant for the actions of $\Gamma$ and the Hecke operators
		$T_{\ell}^*$, $U_{\ell}^*$, and $\langle v\rangle_N^*$
		$($respectively $T_{\ell}$, $U_{\ell}$, and $\langle v\rangle_N$$)$, and
		compatible with change in $r$ using the trace mappings attached
		to $\pr: I_r^{\star}\rightarrow I_s^{\star}$ and to 
		$\o{\pr}:\o{\X}_r\rightarrow \o{\X}_{s}$ $($respectively 
		$\o{\ps}:\o{\X}_r\rightarrow \o{\X}_{s}$$)$.
\end{proposition}

\begin{proof}
	This follows immediately from Corollaries \ref{SplitIgusa}--\ref{FreenessInCharp}
	and Lemma \ref{CharacterSpaces}, using the fact that the group ring $\F_p[\Delta_1/\Delta_r]$
	is local, so any projective $\F_p[\Delta_1/\Delta_r]$-module 
	is free.
\end{proof}

\section{Ordinary \texorpdfstring{$\Lambda$}{Lambda}-adic de Rham cohomology}\label{results}

In this section, we will state and prove our main results as described in \S\ref{resultsintro}.
Throughout, we will keep the notation of \S\ref{resultsintro} and the introduction of \S\ref{Prelim}.

\subsection{The formalism of towers}\label{TowerFormalism}

In this preliminary section, we set up a general commutative algebra framework 
for dealing with the various projective limits of cohomology modules
that we will encounter.

\begin{definition}
	A {\em tower of rings} is an inductive system $\scrA:=\{A_r\}_{r\ge 1}$ of local rings with local
	transition maps.  A {\em morphism of towers} $\scrA\rightarrow \scrA'$
	is a collection of local ring homomorphisms $A_r\rightarrow A_r'$ which are compatible
	with change in $r$.
	A {\em tower of $\scrA$-modules} $\scrM$
	consists of the following data:
	\begin{enumerate}
		\item For each integer $r\ge 1$, an $A_r$-module $M_r$.
	  	\item A collection of $A_r$-module homomorphisms
			$\varphi_{r,s}:M_r\rightarrow M_{s}\otimes_{A_{s}} A_r $
		for each pair of integers $r\ge s\ge 1$, which are compatible 
		in the obvious way under composition.		
	\end{enumerate}
	A {\em morphism of towers of $\scrA$-modules} $\scrM\rightarrow \scrM'$
	is a collection of $A_r$-module homomorphisms $M_r\rightarrow M_r'$ which
	are compatible with change in $r$ in the evident manner.  For a tower of 
	rings $\scrA=\{A_r\}$, we will write $A_{\infty}$ for the inductive limit,
	and for
	a tower of $\scrA$-modules $\scrM=\{M_r\}$ and any $A_{\infty}$-algebra $B$ we set
	\begin{equation*}
		M_B := \varprojlim_r \left( M_r\otimes_{A_r} B\right)\quad\text{and write simply}\quad
		M_{\infty}:=M_{A_{\infty}},
	\end{equation*}
	with the projective limit taken with respect to the induced transition maps.
\end{definition}

\begin{lemma}\label{Technical}
	Let $\scrA=\{A_r\}_{r\ge 0}$ be a tower of rings
	and suppose that $I_r\subseteq A_r$ is a sequence of proper 
	ideals
	such that 
	the image of $I_{r}$ in $A_{r+1}$ is contained in $I_{r+1}$ for all $r$.	
	Write $I_{\infty}:=\varinjlim I_r$
	for the inductive limit, and set $\o{A}_r:=A_r/I_r$ for all $r$.
	Let $\scrM=\{M_r,\pr_{r,s}\}$ be a tower of $\scrA$-modules
	equipped with an action\footnote{That is, a homomorphism of 
	groups $\Delta_1\rightarrow \Aut_{\scrA}(\scrM)$,
	or equivalently, an $A_r$-linear action of $\Delta_1$ on $M_r$
	for each $r$ that is compatible with change in $r$.
	}
	of $\Delta_1$ by $\scrA$-automorphisms.  Suppose that 
	$M_r$ is free of finite rank over $A_r$ for all $r$, and that 
	$\Delta_r$ acts trivially on $M_r$.
	Let $B$ be an $A_{\infty}$-algebra, and observe that $M_B$
	is canonically a module over the completed group ring $\Lambda_B$.
	Assume that 
	that the following
	conditions hold for $r>0$:
	\begin{enumerate}
	\setcounter{equation}{1}
		\renewcommand{\theenumi}{\theequation{\rm\alph{enumi}}}
		{\setlength\itemindent{10pt} 
			\item $\o{M}_r:=M_r/I_rM_r$ is a free $\o{A}_r[\Delta_1/\Delta_r]$-module of rank 
			$d$ that is independent of $r$.\label{freehyp}}
		{\setlength\itemindent{10pt} 
		\item For all $s\le r$ the induced maps 
		$\xymatrix@1{
				{\overline{\pr}_{r,s}: \o{M}_r}\ar[r] & 
				{\o{M}_{s}\otimes_{\o{A}_{s}} \o{A}_{r}}
				}$
		are surjective.\label{surjhyp}} 
	\end{enumerate}
	Then:
	\begin{enumerate}
	 	\item $M_r$ is a free $A_r[\Delta_1/\Delta_r]$-module of rank $d$ for all $r$.\label{red2pfree}
	 
		\item The induced maps of $A_r[\Delta_1/\Delta_{s}]$-modules
		\begin{equation*}
			\xymatrix{
				{M_r \otimes_{A_r[\Delta_1/\Delta_r]} A_r[\Delta_1/\Delta_{s}]} \ar[r] & {M_{s}\otimes_{A_{s}} A_r}
			}
		\end{equation*}
		are isomorphisms for all $r\ge s$.\label{red2psurj}
		
		\item $M_B$ is a finite free $\Lambda_{B}$-module of rank $d$.	\label{MBfree}
		
		\item For each $r$, the canonical map 
		\begin{equation*}
			\xymatrix{
				{M_B\otimes_{\Lambda_B} B[\Delta_1/\Delta_r]} \ar[r] & {M_r\otimes_{A_r} B}
				}
		\end{equation*}
		is an isomorphism of $B[\Delta_1/\Delta_r]$-modules.
		\label{ControlLimit}

		\item If $B'$ is any 
		$B$-algebra, 
		then the canonical map
		\begin{equation*}
			\xymatrix{
				{M_B\otimes_{\Lambda_B} \Lambda_{B'}} \ar[r] & {M_{B'}}
				}
		\end{equation*}
		is an isomorphism of finite free $\Lambda_{B'}$-modules.\label{CompletedBaseChange}
		
	\end{enumerate}
\end{lemma}

\begin{proof}
	For notational ease, let us put $\Lambda_{A_r,s}:=A_r[\Delta_1/\Delta_{s}]$ for all pairs of nonnegative integers
	$r,s$.  Note that $\Lambda_{A_r,s}$ is a local $A_r$-algebra, so the 
	ideal $\widetilde{I}_r:=I_r\Lambda_{A_r,r}$ is 
	contained in the radical of $\Lambda_{A_r,r}$. 
	
	Let us fix $r$ and choose a generator $\gamma$ of $\Delta_1/\Delta_r$.
	The module $M_r$ is obviously finite
	over $\Lambda_{A_r,r}$ (as it is even finite over $A_r$), so by hypothesis (\ref{freehyp})
	we may choose $m_1,\ldots,m_{d}\in M_r$
	with the property that the images $\o{m}_i$ of the $m_i$ in $\o{M}_r=M_r/\widetilde{I}_rM_r$ 
	freely generate $\o{M}_r$ as an 
	$\o{A}_r[\Delta_1/\Delta_r]=\Lambda_{A_r,r}/\widetilde{I}_r$-module.   
	By Nakayama's Lemma \cite[Corollary to Theorem 2.2]{matsumura}, we conclude
	that $m_1,\ldots,m_{d}$ generate $M_r$ as a $\Lambda_{A_r,r}$-module.  
	It follows that $\{\gamma^i m_j\}_{i,j}$ for $i=1,\ldots, r$ and $j=1,\ldots, d$
	generate $M_r$ as an $A_r$-module.  We claim that these elements {\em freely}
	generate $M_r$ as an $A_r$-module.  To see this, we first observe that
	as $M_r$ is free of finite rank over $A_r$
	by hypothesis, we may select an $A_r$-basis $e_1,\ldots, e_N$ of $M_r$.
	The images $\{\o{e}_i\}_i$ of the $e_i$ in $\o{M}_r$ are then an $\o{A}_r$-basis;
	as the same is true of $\{\gamma^i \o{m}_j\}_{i,j}$, we necessarily have
	$N=rd$.  Choosing a reindexing the $e_i$'s, we obtain an $A_r$-basis $\{e_{ij}\}_{i,j}$
	for $1\le i \le r$ and $1\le j\le d$ of $M_r$.  Now consider the $A_r$-linear map  
	$\psi:M\rightarrow M$ sending $e_{ij}$ to $\gamma^im_j$.  
	As $m_1,\ldots, m_d$ generate $M_r$ as an $A_r[\Delta_1/\Delta_r]$-module, 
	necessarily $\{\gamma^im_j\}_{i,j}$ generates $M_r$ as an $A_r$-module and $\psi$
	is surjective.  By \cite[Theorem 2.4]{matsumura}, we conclude that 
	$\psi$ is an isomorphism and $\{\gamma^i m_j\}_{i,j}$ freely generates 
	$M_r$ as an $A_r$-module, as claimed.  
	It then follows that $m_1,\ldots,m_{d}$ freely generate $M_r$ over $\Lambda_{A_r,r}$, giving 
	(\ref{red2pfree}).
	
	To prove (\ref{red2psurj}), note that our assumption (\ref{surjhyp}) that the maps $\overline{\pr}_{r,s}$ are 
	surjective for all $r\ge s$ implies that the same is true of the maps $\pr_{r,s}$ 
	(again by Nakayama's Lemma) and hence that the induced map of 
	$\Lambda_{A_r,s}$-modules in (\ref{red2psurj}) is surjective.
	As this map is then a surjective map of free $\Lambda_{A_r,s}$-modules of the same
	rank $d$, it must be an isomorphism.  
	
	Since the kernel of the canonical surjection $\Lambda_{A_r,r}\twoheadrightarrow \Lambda_{A_r,s}$
	lies in the radical of $\Lambda_{A_r,r}$, we deduce by Nakayama's Lemma that 
	any lift to $M_r$ of a $\Lambda_{A_r,s}$-basis of $M_{s}\otimes_{A_{s}} A_r$
	is a $\Lambda_{A_r,r}$-basis of $M_r$.
	It follows
	that for all $r\ge 0$, we may choose a $B[\Delta_1/\Delta_r]$-basis 
	$e_{1r},\ldots, e_{dr}$ of the free $B[\Delta_1/\Delta_r]$-module $M_r\otimes_{A_r} B$
	with the property that $e_{ir}$ is carried to $e_{is}$ under the 
	induced map $M_r\otimes_{A_r} B\rightarrow M_s\otimes_{A_s} B$ for all $s\le r$.
	These choices yield isomorphisms $B[\Delta_1/\Delta_r]^{\oplus d}\simeq M_r\otimes_{A_r} B$
	which are compatible with change in $r$ using the canonical projections on each factor
	on the left side and the induced transition maps on the right.  As projective limits
	commute with finite direct sums, it follows that
	$M_B$ is a free $\Lambda_{B}$-module of rank 
	$d$ for any 
	$A_{\infty}$-algebra $B$.  
	
	We have seen that the mappings $\rho_{r,s}$ are surjective for all $r\ge s$,
	and it follows that the canonical map $M_B\twoheadrightarrow M_r\otimes_{A_r} B$
	is surjective as well.  Since the mapping of (\ref{ControlLimit}) is obtained from this surjection
	by extension of scalars (keeping in mind the natural identification 
	$(M_r\otimes_{A_r} B) \otimes_{\Lambda_B} B[\Delta_1/\Delta_r] \simeq M_r \otimes_{A_r} B$),
	we deduce that (\ref{ControlLimit}) is likewise surjective.
	From (\ref{red2pfree}) and (\ref{MBfree}) we conclude that
	the mapping in (\ref{ControlLimit}) is a surjection of free $B[\Delta_1/\Delta_r]$-modules
	of the same rank and is hence an isomorphism as claimed.

	It remains to
	prove (\ref{CompletedBaseChange}). 
	Extending scalars,
	the canonical maps
	$M_B\twoheadrightarrow M_r\otimes_{A_r} B$ induce surjections
	\begin{equation*}
		\xymatrix{
			{M_{B}\otimes_{\Lambda_B} \Lambda_{B'}} \ar@{->>}[r] & 
			{(M_r\otimes_{A_r} B)\otimes_{\Lambda_B} \Lambda_{B'} \simeq M_r\otimes_{A_r} B'}
			}
	\end{equation*}
	that are compatible in the evident manner with change in $r$.  Passing to inverse
	limits gives the mapping $M_{B}\otimes_{\Lambda_{B}}\Lambda_{B'}\rightarrow M_{B'}$ 
	of (\ref{CompletedBaseChange}).  Due to (\ref{MBfree}),
	this is then a map of finite free $\Lambda_{B'}$-modules of the same rank, 
	so to check that it is an isomorphism it suffices by Nakayama's Lemma
	to do so after applying $\otimes_{\Lambda_{B'}} B'[\Delta_1/\Delta_r]$, which is
	an immediate consequence of (\ref{ControlLimit}). 
\end{proof}

We record the following elementary commutative algebra fact,
which will be of use to us in the sequel to this paper \cite{CaisHida2}:

\begin{lemma}\label{fflatfreedescent}
	Let $A\rightarrow B$ be a local homomorphism of local rings which makes $B$ into
	a flat $A$-algebra, and let $M$ be an arbitrary $A$-module.
	Then $M$ is a free $A$-module of finite rank if and only if $M\otimes_A B$
	is a free $B$-module of finite rank.
\end{lemma}

\begin{proof}
	First observe that since $A\rightarrow B$ is local and flat, it is faithfully flat.
	We write $M=\varinjlim M_{\alpha}$ as the direct limit of its finitely generated $A$-submodules,
	whence $M\otimes_A B = \varinjlim (M_{\alpha}\otimes_A B)$ with each of $M_{\alpha}\otimes_A B$
	naturally a finitely generated $B$-submodule of $M\otimes_A B$.  Assume that $M\otimes_A B$
	is finitely generated as a $B$-module. Then there exists $\alpha$ with 
	$M_{\alpha}\otimes_A B\rightarrow M\otimes_A B$ surjective, and as $B$ is faithfully flat over $A$,
	this implies that $M_{\alpha}\rightarrow M$ is surjective, whence $M$ is finitely generated over $A$.
	Suppose in addition that $M\otimes_A B$ is free as a $B$-module.  In particular, $M\otimes_A B$
	is $B$-flat, which implies by faithful flatness of $B$ over $A$ that $M$ is $A$-flat
	(see, e.g. \cite[Exercise 7.1]{matsumura}).  Then $M$ is a finite flat module over the local ring $A$,
	whence it is free as an $A$-module by \cite[Theorem 7.10]{matsumura}.
\end{proof}

Finally, we analyze duality for towers with $\Delta_1$-action.  

\begin{lemma}\label{LambdaDuality}
	With the notation of Lemma $\ref{Technical}$, 
	let
	$\scrM:=\{M_r,\pr_{r,s}\}$ and $\scrM':=\{M_r',\pr_{r,s}'\}$ be two 
	towers of $\scrA$-modules with $\Delta_1$-action satisfying $(\ref{freehyp})$ and $(\ref{surjhyp})$.
	Suppose that for each $r$ there exist $A_r$-linear perfect duality pairings
	\begin{equation}
		\xymatrix{
			{\langle \cdot,\cdot \rangle_{r}:M_r\times M_r'} \ar[r] & A_r
		}\label{pairinghyp}
	\end{equation}
	with respect to which $\delta$ is self-adjoint for all $\delta\in \Delta_1$,
	and which satisfy the compatibility condition\footnote{By abuse of notation,
	for any map of rings $A\rightarrow B$ and any $A$-bilinear pairing of $A$-modules 
	$\langle\cdot,\cdot\rangle:M\times M'\rightarrow A$, we again write 
	$\langle\cdot,\cdot\rangle: M_B\times M_B'\rightarrow B$ for the $B$-bilinear pairing induced
	by extension of scalars.}
	\begin{equation}
		\langle\pr_{r,s}m, \pr_{r,s}'m'\rangle_{s} = 
		\sum_{\delta\in \Delta_{s}/\Delta_{r}} \langle m,\delta^{-1} m'\rangle_{r}
		\label{pairingchangeinr}
	\end{equation}
	for all $r\ge s$.  Then
 	for each $r$, the pairings
		$
			\xymatrix@1{
				{(\cdot,\cdot)_{r}: M_r \times M_r'} \ar[r] &  \Lambda_{A_r,r}				}
		$
		defined by
		\begin{equation*}
			(m,m')_{r} := \sum_{\delta\in \Delta_1/\Delta_r} \langle m , \delta^{-1} m'\rangle_r \cdot \delta
		\end{equation*}
		are $\Lambda_{A_r,r}$-bilinear and perfect, and  compile to give a $\Lambda_B$-linear
		perfect pairing
		\begin{equation*}
			\xymatrix{
				{(\cdot,\cdot)_{\Lambda_B}: M_B \times M_B'} \ar[r] & {\Lambda_B}
				}.
		\end{equation*}
		In particular, $M_B'$ and $M_B$ are canonically $\Lambda_B$-linearly dual to eachother. 
\end{lemma}

\begin{proof}
	An easy reindexing argument shows that $(\cdot,\cdot)_{r}$ is $\Lambda_{A_r,r}$-linear
	in the right factor, from which it follows that it is also $\Lambda_{A_r,r}$-linear
	in the left due to our assumption that $\delta\in \Delta_1$ is self-adjoint with respect to
	$\langle\cdot, \cdot\rangle_{r}$.  To prove that 
	$(\cdot,\cdot)_{r}$ is a perfect duality pairing, we analyze the $\Lambda_{A_r,r}$-linear map
	\begin{equation}
		\xymatrix@C=45pt{
			{M_r} \ar[r]^-{m\mapsto (m,\cdot)_r} & {\Hom_{\Lambda_{A_r,r}}(M_r',\Lambda_{A_r,r})}
		}.\label{GroupRingDuality}
	\end{equation}
	Due to Lemma \ref{Technical}, both $M_r$ and $M_r'$ are free $\Lambda_{A_r,r}$-modules,
	necessarily of the same rank by the existence of the perfect $A_r$-duality pairing (\ref{pairinghyp}).
	It follows that (\ref{GroupRingDuality}) is a homomorphism of free $\Lambda_{A_r,r}$-modules
	of the same rank. To show that it is an isomorphism it therefore suffices to prove it is surjective,
	which may be checked after extension of scalars along the augmentation map
	$\Lambda_{A_r,r}\twoheadrightarrow A_r$ by Nakayama's Lemma.
	Consider the diagram
	\begin{equation}
	\begin{gathered}
		\xymatrix@C=38pt{
		{M_r \tens_{\Lambda_{A_r,r}} A_r} \ar[r]^-{(\ref{GroupRingDuality})\otimes 1}
		\ar[d]_-{\rho_{r,1}\otimes 1}^-{\simeq} & 
		{\Hom_{\Lambda_{A_r,r}}(M_r',\Lambda_{A_r,r})\tens_{\Lambda_{A_r,r}} A_r} \ar[r]^-{\xi}_-{\simeq} &
		{\Hom_{A_r}(M_r'\tens_{\Lambda_{A_r,r}} A_r, A_r)}  \\
		{M_1\tens_{A_1} A_r} \ar[rr]^-{\simeq} & &{\Hom_{A_r}(M_1'\tens_{A_1} A_r, A_r)}
		\ar[u]_-{(\rho_{r,1}'\otimes 1)^{\vee}}^-{\simeq}
		}
	\end{gathered}
	\label{XiDiagram}
	\end{equation}
	where $\xi$ is the canonical map sending $f\otimes \alpha$ to $\alpha(f\otimes 1)$,
	and the bottom horizontal arrow is obtained by $A_r$-linearly extending the canonical
	duality map $m\mapsto \langle m,\cdot\rangle_1$.  On the one hand, the vertical
	maps in (\ref{XiDiagram}) are isomorphisms thanks to Lemma \ref{Technical} (\ref{red2psurj}),
	while the map $\xi$ and the bottom horizontal arrow are isomorphisms
	because arbitrary extension of scalars commutes with linear duality 
	of finitely generated {\em free} modules.\footnote{Quite generally, for any ring $R$, any $R$-modules 
	$M$, $N$, and any $R$-algebra $S$, the canonical map
	\begin{equation*}
	\xymatrix{
		{\xi_M:\Hom_R(M,N)\otimes_R S} \ar[r] & {\Hom_S(M\otimes_R S, N\otimes_R S)}
		}
	\end{equation*}
	sending $f\otimes s$ to $s(f\otimes \id_S)$ is an isomorphism if $M$ is finite and free over $R$.
	Indeed, the map $\xi_R$ is visibly an isomorphism, and one checks that $\xi_{M_1\oplus M_2}$
	is naturally identified with $\xi_{M_1}\oplus \xi_{M_2}$.
	}
	On the other hand, this diagram commutes because (\ref{pairingchangeinr}) guarantees the relation
	\begin{equation*}
	       \xymatrix@C=15pt{
	       {\langle \rho_{r,1}m , \rho_{r,1}'m' \rangle_1} \ar@^{=}[r]^-{(\ref{pairingchangeinr})} & 	    	       
	       {\displaystyle\sum_{\delta\in \Delta_1/\Delta_r} \langle  m,\delta^{-1}m'\rangle_r
	       \equiv (m,m')_r \bmod{I_{\Delta}}}
	       }  
	\end{equation*}
	where $I_{\Delta}=\ker(\Lambda_{A_r,r}\twoheadrightarrow A_r)$ is the augmentation ideal.
	We conclude that (\ref{GroupRingDuality}) is an isomorphism, as desired.
	The argument that the corresponding map with the roles of $M_r$ and $M_r'$ interchanged
	is an isomorphism proceeds {\em mutatis mutandis}.
	
	Using the definition
	of $(\cdot,\cdot)_r$ and (\ref{pairingchangeinr}), one has more generally that
	\begin{equation*}
		(\rho_{r,s}m,\rho_{r,s}'m')_{s} \equiv (m,m')_r \bmod 
		\ker(\Lambda_{A_r,r}\twoheadrightarrow \Lambda_{A_r,s}) 
	\end{equation*}
	for all $r\ge s$.  In particular, the pairings $(\cdot,\cdot)_r$ induce, by extension 
	of scalars, a $\Lambda_B$-bilinear pairing
	\begin{equation*}
		\xymatrix{
			{(\cdot,\cdot)_{\Lambda_B}: M_B\times M_{B}'} \ar[r] & {\Lambda_B}
			}
	\end{equation*}
	which satisfies the specialization property
	\begin{equation}
		(\cdot,\cdot)_{\Lambda_B} \equiv (\cdot, \cdot)_{r} \bmod \ker(\Lambda_B \twoheadrightarrow \Lambda_{B,r}).
		\label{pairingspecialize}
	\end{equation}
	From $(\cdot,\cdot)_{\Lambda_B}$ we obtain in the usual way duality morphisms
	\begin{equation}
		\xymatrix@C=45pt{
			{M_B} \ar[r]^-{m\mapsto (m,\cdot)_{\Lambda_B}} & \Hom_{\Lambda_B}(M_{B}',\Lambda_B)
		}\quad\text{and}\quad
		\xymatrix@C=45pt{
			{M_B'} \ar[r]^-{m'\mapsto (\cdot,m')_{\Lambda_B}} & \Hom_{\Lambda_B}(M_{B},\Lambda_B)
		}\label{LambdaDualityMaps}
	\end{equation}
	which we wish to show are isomorphisms.  Due to Lemma \ref{Technical} (\ref{MBfree}),
	each of (\ref{LambdaDualityMaps}) is a map of finite free $\Lambda_B$-modules of the same rank,
	so we need only show that these mappings are surjective.  As the kernel of 
	$\Lambda_B\twoheadrightarrow \Lambda_{B,r}$ is contained in the radical of $\Lambda_B$,
	we may by Nakayama's Lemma check 
	such surjectivity after extension of scalars along $\Lambda_B\twoheadrightarrow \Lambda_{B,r}$
	for any $r$, where it follows from (\ref{pairingspecialize}) and the fact that $M_r$ and $M_{r}'$
	are free $\Lambda_{A_r,r}$-modules, so that the extension of scalars of the perfect
	duality pairing $(\cdot,\cdot)_r$ along the canonical map 
	$\Lambda_{A_r,r}\rightarrow \Lambda_{B,r}$ is again perfect.	
\end{proof}

\subsection{Ordinary families of de Rham cohomology}\label{ordfamdR}

Let $\{\X_r/T_r\}_{r\ge 0}$ be the tower of modular curves 
introduced in \S\ref{tower}.  
As $\X_r$ is regular and proper flat over $T_r=\Spec(R_r)$ with geometrically reduced fibers, it is a curve
in the sense of Definition \ref{curvedef} (thanks to Corollary \ref{curvecorollary})
which moreover satisfies the hypotheses of Proposition \ref{HodgeIntEx}.
Abbreviating
\begin{align}
	& H^0(\omega_{r}):=H^0(\X_{r},\omega_{\X_{r}/S_{r}}),  &  & H^1_{\dR,r}:= H^1(\X_{r}/R_{r}), &  & H^1(\O_{r}):=H^1(\X_{r},\O_{\X_{r}}),\label{shfcoh}
\end{align}  
Proposition \ref{HodgeIntEx} (\ref{CohomologyIntegral}) provides a canonical short exact sequence
$H(\X_r/R_r)$ of finite 
free $R_r$-modules
\begin{equation}
	\xymatrix{
		0\ar[r] & {H^0(\omega_r)} \ar[r] & {H^1_{\dR,r}} \ar[r] & {H^1(\O_r)} \ar[r] & 0
	}\label{HodgeFilIntAbbrev}
\end{equation}
which recovers the Hodge filtration of $H^1_{\dR}(X_r/K_r)$ after inverting $p$.

The Hecke correspondences on $\X_r$ induce, via Proposition \ref{HodgeIntEx} (\ref{CohomologyFunctoriality})
and the discussion surrounding Definition \ref{CorrDef},
canonical actions of $\H_r$ and $\H_r^*$ on $H(\X_r/R_r)$ via $R_r$-linear endomorphisms.
In particular, $H(\X_r/R_r)$ is canonically a short exact sequence of $\Lambda$-modules
via the diamond and adjoint diamond operator maps $\Lambda\hookrightarrow \H$ and
$\Lambda\hookrightarrow \H^*$ as in \S\ref{Notation}.
Similarly, pullback along (\ref{gammamaps}) yields $R_r$-linear morphisms
$H((\X_r)_{\gamma}/R_r)\rightarrow H(\X_r/R_r)$ for each $\gamma\in \Gamma$;
using the fact that hypercohomology commutes with flat base change 
(by \v{C}ech theory), we obtain an action of $\Gamma$ on $H(\X_r/R_r)$
which is $R_r$-semilinear over the canonical action of $\Gamma$ on $R_r$
and which commutes with the actions of $\H_r$ and $\H_r^*$ as the Hecke operators are defined over 
$K_0=\Q_p$.

For $r\ge s$, we will need to work with the base change $\X_s\times_{T_s} T_r$,
which is a curve over $T_r$ thanks to Proposition \ref{curveproperties}.
Although $\X_s\times_{T_s} T_r$ need no longer be regular as  $T_{r}\rightarrow T_{s}$ is not smooth when $r> s$,
we claim that it is necessarily {\em normal}.  
Indeed, this follows from the more general assertion:

\begin{lemma}\label{normalcrit}
		Let $V$ be a discrete valuation ring and $A$ a finite type Cohen-Macaulay $V$-algebra with smooth generic fiber and geometrically reduced
		special fiber.  Then $A$ is normal.
\end{lemma} 

\begin{proof}
	We claim that $A$ satisfies Serre's ``$R_1+S_2$"-criterion for normality \cite[Theorem 23.8]{matsumura}.  As $A$ is assumed
	to be CM, by definition of Cohen-Macaulay $A$ verifies $S_i$ for all $i\ge 0$, so we need only show that each localization of $A$ at
	a prime ideal of codimension 1 is regular.  Since $A$ has geometrically reduced special fiber, this special fiber is in particular
	smooth at its generic points.  As $A$ is flat over $V$ (again by definition of CM), we deduce that the (open) $V$-smooth locus   
	in $\Spec A$ contains the generic points of the special fiber and hence contains all codimension-1 points (as the generic
	fiber of $\Spec A$ is assumed to be smooth).  Thus $A$ is $R_1$, as desired.
\end{proof}

We conclude that $\X_{s}\times_{T_s} T_r$ is a normal curve, and 
we obtain from Proposition \ref{HodgeIntEx}
a canonical short exact sequence of finite free $R_r$-modules $H(\X_{s}\times_{T_s}T_r/R_r)$
which recovers the Hodge filtration of $H^1_{\dR}(X_{s}/K_r)$ after inverting $p$.
As hypercohomology commutes with flat base change and the formation of the 
relative dualizing sheaf and the structure sheaf are 
compatible with arbitrary base change, we have a natural isomorphism of short exact sequences of free $R_r$-modules
\begin{equation}
	H(\X_{s}\times_{T_s} T_r/R_r)	\simeq  H(\X_{s}/R_{s})\otimes_{R_{s}} R_r. \label{bccompat}
\end{equation}
In particular, we have $R_r$-linear actions of $\H_{s}^*$, $\H_s$ and an $R_r$-semilinear action of
$\Gamma$ on $H(\X_s\times_{T_s} T_r/R_r)$.  These actions moreover commute with one another.

Consider now the canonical degeneracy map $\pr: \X_r\rightarrow \X_{s}\times_{T_s} T_r$ of curves over $T_r$
induced by (\ref{rdegen}).  
As $\X_r$ and $\X_{s}\times_{T_s} T_r$ are normal and proper curves over $T_r$,
we obtain from Proposition \ref{HodgeIntEx} (\ref{CohomologyFunctoriality}) 
canonical trace mappings of short exact sequences
\begin{equation}
		\xymatrix{
		{\pr_* : H(\X_{r}/R_r)} \ar[r] & {H(\X_{s}\times_{T_s}T_r/R_r)
		\simeq H(\X_{s}/R_{s})\otimes_{R_{s}}R_r}
	}\label{trmap}
\end{equation}
which recover the usual trace mappings on de Rham cohomology after inverting $p$;
as such, these mappings are Hecke and $\Gamma$-equivariant, and
compatible with change in $r,s$ in the obvious way.
Tensoring the maps (\ref{trmap}) over $R_r$ with $R_{\infty}$, 
we obtain projective systems of finite free $R_{\infty}$-modules 
with semilinear $\Gamma$-action and commuting, linear 
$\H^*:=\varprojlim_r \H_r^*$ action.

\begin{definition}\label{limitmods}
	We write 
	\begin{align*}
			&H^0(\omega):=\varprojlim_r \left(H^0(\omega_r) \tens_{R_r} {R}_{\infty}\right),  
			&& H^1_{\dR}:=\varprojlim_r \left(H^1_{\dR,r}\tens_{R_r} {R}_{\infty} \right),
			&&  H^1(\O) := \varprojlim_r \left(H^1(\O_r)\tens_{R_r} {R}_{\infty}\right)
	\end{align*}
	for the limits with respect to the maps induced by $\pr_*$,
	which are naturally 
	$\Lambda_{R_{\infty}}={R}_{\infty}[\![\Delta_1]\!]$-modules via the adjoint 
	diamond operators 
	and are equipped with a semilinear $\Gamma$-action 
	and a linear $\H^*$-action.  
\end{definition}

\begin{theorem}\label{main}
	Let $e^*$ be the idempotent of $\H^*$ associated to $U_p^*$ and let $d$ be
	the positive integer defined as in Proposition $\ref{IgusaStructure}$ $(\ref{IgusaFreeness})$.
	Then $e^*H^0(\omega)$, $e^*H^1_{\dR}$ and $e^*H^1(\O)$ are free $\Lambda_{R_{\infty}}$-modules
	of ranks $d$, $2d$, and $d$ respectively, and there is a canonical short exact sequence
	of free $\Lambda_{R_{\infty}}$-modules with linear $\H^*$-action and $R_{\infty}$-semilinear
	$\Gamma$-action
	\begin{equation}
		\xymatrix{
			0\ar[r] & {e^*H^0(\omega)} \ar[r] & {e^*H^1_{\dR}} \ar[r] & {e^*H^1(\O)} \ar[r] & 0
			}.\label{mainthmexact}
	\end{equation}
	For each positive integer $r$, applying $\otimes_{\Lambda_{R_{\infty}}} R_{\infty}[\Delta_1/\Delta_r]$
	to $(\ref{mainthmexact})$ yields the short exact sequence
	\begin{equation}
		\xymatrix{
			0\ar[r] & {e^*H^0(\omega_r)\tens_{R_r} R_{\infty}} \ar[r] & 
			{e^*H^1_{\dR}\tens_{R_r} R_{\infty}} \ar[r] & 
			{e^*H^1(\O)\tens_{R_r} R_{\infty}} \ar[r] & 0
			},\label{mainthmexact2}	
	\end{equation}
	compatibly with the actions of $\H^*$ and $\Gamma$.	 
\end{theorem}

\begin{proof}
	
Applying $e^*$ to the short exact sequence $H(\X_r/R_r)$ yields a short exact sequence
\begin{equation}
	\xymatrix{
		0\ar[r] & {e^*H^0(\omega_r)}\ar[r] & {e^*H^1_{\dR,r}} \ar[r] & {e^*H^1(\O_r)} \ar[r] & 0
		}\label{hitwithidem}
\end{equation}
of $R_r[\Delta_1/\Delta_r]$-modules with linear $\H_r^*$-action and $R_r$-semilinear $\Gamma$-action in which each term is
free as an $R_r$-module.\footnote{Indeed, $e^*M$ is a direct summand of $M$ for any $\H_r^*$-module $M$,
and hence $R_r$-projective ($=R_r$-free) if $M$ is.}
Similarly, for each pair of nonnegative integers $r\ge s$,
the trace mappings (\ref{trmap}) induce a commutative diagram with exact rows
\begin{equation}
\begin{gathered}
	\xymatrix{
		0\ar[r] & {e^*H^0(\omega_r)} \ar[r]\ar[d]_-{\pr_*} & 
		{e^*H^1_{\dR,r}} \ar[r]\ar[d]_-{\pr_*} & {e^*H^1(\O_r)} 
		\ar[r]\ar[d]^-{\pr_*} & 0\\
		0\ar[r] & {e^*H^0(\omega_{s})\otimes_{R_{s}}R_r} \ar[r] & 
		{e^*H^1_{\dR,{s}}\otimes_{R_{s}}R_r} \ar[r] & {e^*H^1(\O_{s})\otimes_{R_{s}}R_r} 
		\ar[r] & 0
	}
	\end{gathered}
	\label{piecetogether}
\end{equation}
We will apply Lemma \ref{Technical} with $A_r=R_r$, $I_r=(\varepsilon^{(r)}-1)$, $B=R_{\infty}$ 
and with $M_r$ each one of the terms in (\ref{hitwithidem}).
In order to do this, we must check that the hypotheses (\ref{freehyp}) and 
(\ref{surjhyp}) are satisfied.

Applying $\otimes_{R_r} \F_p$ to the short exact sequence (\ref{hitwithidem})
and using the fact that the idempotent $e^*$ commutes with tensor products,
we obtain, thanks to Lemma \ref{ReductionCompatibilities} (\ref{BaseChngDiagram}),
the short exact sequence of $\F_p$-vector spaces (\ref{sesincharp1}).
By Corollary \ref{FreenessInCharp}, the three terms of (\ref{sesincharp1}) are free 
$\F_p[\Delta_1/\Delta_r]$-modules
of ranks $d$, $2d$, and $d$ respecvitely, so (\ref{freehyp})
is satisfied for each of these terms.  Similarly, applying $\otimes_{R_r} \F_p$
to the diagram (\ref{piecetogether}) 
yields a diagram which by Corollary \ref{SplitIgusa} is naturally isomorphic to
the diagram of $\F_p[\Delta_1/\Delta_r]$-modules with split-exact rows
\begin{equation*}
	\xymatrix@C=15pt{
		0 \ar[r] & {H^0(I_r^{\infty},\Omega^1(\SS))^{V_{\ord}}} \ar[r]\ar[d]_-{\pr_*} &
		{H^0(I_r^{\infty},\Omega^1(\SS))^{V_{\ord}}\oplus
		H^1(I_r^0,\O(-\SS))^{F_{\ord}}} \ar[r]\ar[d]|-{\pr_*\oplus \pr_*} & 
		{H^1(I_r^0,\O(-\SS))^{F_{\ord}}} \ar[r]\ar[d]^-{\pr_*} & 0 \\
		0 \ar[r] & {H^0(I_{s}^{\infty},\Omega^1(\SS))^{V_{\ord}}} \ar[r] &
		{H^0(I_{s}^{\infty},\Omega^1(\SS))^{V_{\ord}}\oplus
		H^1(I_{s}^0,\O(-\SS))^{F_{\ord}}} \ar[r] & 
		{H^1(I_{s}^0,\O(-\SS))^{F_{\ord}}} \ar[r] & 0
	}
\end{equation*}
Each of the vertical maps in this diagram is surjective due to Proposition \ref{IgusaStructure} 
(\ref{IgusaControl}), and we conclude that the hypothesis (\ref{surjhyp}) is satisfied
as well.  Furthermore, the vertical maps in (\ref{piecetogether}) are then surjective by Nakayama's
Lemma, so applying $\otimes_{R_r} R_{\infty}$ 
yields an inverse system of short exact sequences in which the first term satisfies the Mittag-Leffler
condition.  Passing to inverse limits is therefore (right) exact, and we obtain the short 
exact sequence (\ref{mainthmexact}).
\end{proof}

Due to Proposition \ref{HodgeIntEx} (\ref{CohomologyDuality}), the short exact sequence
(\ref{HodgeFilIntAbbrev}) is auto-dual with respect to the canonical cup-product pairing $(\cdot,\cdot)_r$
on $H^1_{\dR,r}$.  We extend scalars along $R_r\rightarrow R_r':=R_r[\mu_N]$, so that the Atkin-Lehner
``involution" $w_r$ is defined, and consider the ``twisted" pairing on ordinary parts
\begin{equation}
\xymatrix{
	{\langle \cdot,\cdot\rangle _r : ({e^*}H^1_{\dR,r})_{R_r'} \times ({e^*}H^1_{\dR,r})_{R_r'}} \ar[r] & {R_r'}
	}\qquad\text{given by}\qquad \langle x, y\rangle_r := (x, w_r {U_p^*}^r y).
	\label{TwistdRpairing}
\end{equation}
It is again perfect and satisfies $\langle T^* x,y\rangle =\langle x, T^* y \rangle$
for all $x,y\in (e^*H^1_{\dR,r})_{R_r'}$ and $T^*\in \H_r^*$.
\begin{proposition}\label{dRDuality}
	The pairings $(\ref{TwistdRpairing})$ compile to give a perfect $\Lambda_{R_{\infty}'}$-linear
	duality pairing 
	\begin{equation*}
		\xymatrix{
		{\langle\cdot,\cdot\rangle_{\Lambda_{R_{\infty}'}}:   
		({e^*}H^1_{\dR})_{\Lambda_{R_{\infty}'}} \hspace{-1ex}\times ({e^*}H^1_{\dR})_{\Lambda_{R_{\infty}'}}}
		\ar[r] & {\Lambda_{R_{\infty}'}}
		}\ \text{given by}\ 
		\langle x , y \rangle_{\Lambda_{R_{\infty}'}} :=
		\varprojlim_r\sum_{\delta\in \Delta_1/\Delta_r} \langle x_r, \langle \delta^{-1}\rangle^* y_r\rangle_r\cdot\delta
	\end{equation*}
	for $x=\{x_r\}_r$ and $y=\{y_r\}_r$ in $(e^*H^1_{\dR})_{\Lambda_{R_{\infty}'}}$.
	The pairing $\langle \cdot,\cdot \rangle_{\Lambda_{R_{\infty}}'}$  induces
	a canonical isomorphism 
	\begin{equation*}
		\xymatrix{
				0\ar[r] & {e^*H^0(\omega)(\langle\chi\rangle\langle a\rangle_N)_{\Lambda_{R_{\infty}'}}}
				\ar[r]\ar[d]^-{\simeq} & 
				{e^*H^1_{\dR}(\langle\chi\rangle\langle a\rangle_N)_{\Lambda_{R_{\infty}'}}} 
				\ar[r]\ar[d]^-{\simeq} & 
				{e^*H^1(\O)(\langle\chi\rangle\langle a\rangle_N)_{\Lambda_{R_{\infty}'}}} 
				\ar[r]\ar[d]^-{\simeq} & 0\\
	0\ar[r] & {(e^*H^1(\O))^{\vee}_{\Lambda_{R_{\infty}'}}} \ar[r] & 
	{({e^*H^1_{\dR}})^{\vee}_{\Lambda_{R_{\infty}'}}} \ar[r] & 
	{(e^*H^0(\omega))^{\vee}_{\Lambda_{R_{\infty}'}}} \ar[r] & 0
		}
	\label{mainthmdualityisom}
	\end{equation*}
	that is $\H^*$- and $(\Z/N\Z)^{\star}$-equivariant and compatible with the natural action 
	of $\Gamma \times \Gal(K_0'/K_0)\simeq \Gal(K_{\infty}'/K_0)$ 
	on the bottom row and the twist
	$\gamma\cdot m := \langle \chi(\gamma)\rangle\langle a(\gamma)\rangle_N \gamma m$
	of the natural action on the top, 	
	where 
	$a(\gamma) \in (\Z/N\Z)^{\times}$ is
	determined by
	$\zeta^{a(\gamma)}=\gamma\zeta$ for every $\zeta\in \mu_N(\Qbar_p)$.
\end{proposition}

\begin{proof}[Proof of Proposition $\ref{dRDuality}$]
	That $\langle\cdot,\cdot \rangle_{\Lambda_{R_{\infty}'}}$ is a perfect duality pairing
	follows easily from Lemma \ref{LambdaDuality}, using Theorem \ref{main} and the formalism of 
	\S\ref{TowerFormalism}, once we check that the twisted pairings (\ref{TwistdRpairing})
	satisfy the hypothesis (\ref{pairingchangeinr}).  By the definition (\ref{TwistdRpairing})
	of $\langle\cdot, \cdot\rangle_r$, this amounts to the computation
	\begin{align*}
		({\rho}_* x, w_r {U_p^*}^r{\rho}_* y)_r = (x, \rho^* w_r {U_p^*}^r{\rho}_*y)_{r+1}
		&=(x, w_{r+1} {U_p^*}^r \sigma^*{\rho}_*y)_{r+1} \\				
		&= \sum_{\delta\in \Delta_r/\Delta_{r+1}}(x, w_{r+1} {U_p^*}^{r+1} \langle \delta^{-1}\rangle^* y)_{r+1}
	\end{align*}
	where we have used Proposition \ref{ALinv} and the identity 
	$\sigma^*{\rho}_* = U_p^*\sum_{\delta\in \Delta_r/\Delta_{r+1}} \langle\delta^{-1}\rangle^*$
	on $H^1_{\dR,r+1}$, which follows from\footnote{The reader will check that our reliance
	on the sequel \cite{CaisHida2} of the present paper does not involve circular reasoning.} 
	Lemma 3.1.1, Lemma 3.1.5 and Proposition 2.2.5 of \cite{CaisHida2}.
	We obtain an isomorphism of short exact sequences of
	$\Lambda_{R_{\infty}'}$-modules as in (\ref{mainthmdualityisom}), which it remains to show is 
	$\Gamma\times\Gal(K_0'/K_0)$-equivariant
	for the specified actions.  For this, we compute that for $\gamma\in\Gal(K_{\infty}'/K_0)$,
	\begin{equation*}
		\langle \gamma x,\gamma y\rangle_{r} =
		(\gamma x, w_r {U_p^*}^r \gamma y)_r = 
		 (\gamma x,  \gamma w_r {U_p^*}^r \langle \chi(\gamma)^{-1} \rangle\langle a(\gamma)^{-1}\rangle_N  y)_r
		 =  \gamma \langle x,\langle \chi(\gamma)^{-1} \rangle\langle a(\gamma)^{-1}\rangle_N y\rangle_r,
	\end{equation*}
	where we have used Proposition \ref{ALinv} and the fact that the cup product is Galois-equivariant.
	It now follows easily from definitions that 
	\begin{equation*}
	\langle \gamma x,\gamma y\rangle_{\Lambda_{R_{\infty}'}} = 
	\langle \chi(\gamma)^{-1} \rangle \gamma \langle x, \langle a(\gamma)^{-1}\rangle_N y	
	\rangle_{\Lambda_{R_{\infty}'}},
	\end{equation*}
	and the claimed $\Gamma\times\Gal(K_0'/K_0)$-equivariance of  (\ref{mainthmdualityisom}) is equivalent to this.
\end{proof}

\begin{remark}
	For an open subgroup $H$ of $\scrG_K$ and any $H$-stable subfield 
	$F$ of $\C_K$, denote by $\Rep_F(H)$ the category of finite-dimensional
	$F$-vector spaces that are equipped with a continuous semilinear 
	action of $H$.  Recall \cite{DSen} that classical Sen theory provides a functor
	$\D_{\Sen}:\Rep_{\C_K}(\scrG_K)\rightarrow \Rep_{K_{\infty}}(\Gamma)$
	which is quasi-inverse to $(\cdot)\otimes_{K_{\infty}} \C_K$.  Furthermore,
	for any $W\in \Rep_{\C_K}(\scrG_K)$, there is a unique $K_{\infty}$-linear
	operator $\Theta_{D}$ on $D:=\D_{\Sen}(W)$ with the property that 
	$\gamma x = \exp(\log \chi(\gamma)\cdot \Theta_D)(x)$ for all $x\in D$
	and all $\gamma$ in a small enough open neighborhood of $1\in \Gamma$.
	
	We expect that for $W$ any specialization of $e^*H^1_{\et}$ along a continuous homomorphism 
	$\Lambda\rightarrow K_{\infty}$, there is a canonical isomorphism between $D:=\D_{\Sen}(W\otimes {\C_K})$
	and the corresponding specialization of $e^*H^1_{\dR}$,
	with the Sen operator $\Theta_D$
	induced by the Gauss-Manin connections on $H^1_{\dR,r}$.  In this way, we might think
	of $e^*H^1_{\dR}$ 
	as a $\Lambda$-adic avatar of ``$\D_{\Sen}(e^*H^1_{\et}\otimes_{\Lambda} \Lambda_{\O_{\C_K}})$."
	We hope to pursue these connections in future work.
\end{remark}

\subsection{Relation to ordinary \texorpdfstring{$\Lambda$}{Lambda}-adic modular forms}\label{ordforms}

In this section, we discuss the relation between $e^*H^0(\omega)$ 
and ordinary $\Lambda_{R_{\infty}}$-adic cuspforms as defined by Ohta \cite[Definition 2.1.1]{OhtaEichler}.

We begin with some preliminaries on modular forms.
For a ring $A$, a congruence subgroup $\Upgamma$, and a nonnegative integer $k$, we will write $S_k(\Upgamma;A)$ for the space of weight $k$ cuspforms for $\Upgamma$ over $A$; we put $S_k(\Upgamma):=S_k(\Upgamma;\Qbar)$.
If $\Upgamma',$ $\Upgamma$ are congruence subgroups and $\gamma\in \GL_2(\Q)$ satisfies
$\gamma^{-1}\Upgamma'\gamma\subseteq \Upgamma$, then there is a 
canonical injective ``pullback" map 
on modular forms
$\xymatrix@1{{\iota_{\gamma}:S_k(\Upgamma)} \ar@{^{(}->}[r] & {S_k(\Upgamma')}}$ 
given by $\iota_{\gamma}(f):=f\big|_{\gamma^{-1}}$. 
When $\Upgamma'\subseteq \Upgamma$, {\em unless specified to the contrary}, we will
always view $S_k(\Upgamma)$ as a subspace of $S_k(\Upgamma')$ via $\iota_{\id}$.
As $\gamma\Upgamma'\gamma^{-1}$ is necessarily of finite index in $\Upgamma$,
one also has a canonical ``trace" mapping
\begin{equation}
	\xymatrix{
		{\tr_{\gamma}:S_k(\Upgamma')} \ar[r] & {S_k(\Upgamma)}
		}
		\qquad\text{given by}\qquad
		\tr_{\gamma}(f):=\sum_{\delta\in \gamma^{-1}\Upgamma'\gamma\backslash\Upgamma} (f\big|_{\gamma})\big|_{\delta}
		\label{MFtrace}
\end{equation}
with the property that $\tr_{\gamma}\circ\iota_{\gamma}$ is multiplication by $[\Upgamma: \gamma^{-1}\Upgamma'\gamma]$
on $S_k(\Upgamma)$.

We put $\Upgamma_r:=\Upgamma_1(Np^r)$ and define
\begin{equation*}
		S_2^{\infty}(\Upgamma_r;R_r):=S_2(\Upgamma_r;R_r)\qquad\text{and}\qquad
		S_2^{0}(\Upgamma_r;R_r):=
		\{f\in S_2(\Upgamma_r; \Qbar_p)\ :\ f\big|_{w_r} \in S_2^{\infty}(\Upgamma_r;R_r) \},
	\end{equation*}
By definition, $S_2^{\star}(\Upgamma_r;R_r)$ for $\star=0,\infty$ are $R_r$-submodules
of $S_2(\Upgamma_r;K_r')$ that are carried isomorphically onto each other by the automorphism
$w_r$ of $S_2(\Upgamma_r;K_r')$.	Note that $S_2^{\star}(\Upgamma_r;R_r)$ is precisely the
$R_r$-submodule consisting of cuspforms whose formal expansion at the cusp $\star$
has coefficients in $R_r$.	
As the Hecke algebra $\H_r$ stabilizes 
$S_2^{\infty}(\Upgamma_r; R_r)$, it follows immediately from 
the intertwining relation $w_rT=T^*w_r$ for any $T\in \H_r$ (see \S\ref{Notation})
that
$S_2^0(\Upgamma_R;R_r)$ is stable under the action of $\H_r^*$ on $S_2(\Upgamma_r; K_r)$.
Furthermore, $\Gal(K_r'/K_0)$ acts on $S_2(\Upgamma_r;K_r')\simeq S_2(\Upgamma_r;\Q_p)\otimes_{\Q_p} K_r'$
through the second tensor factor, and this action leaves stable the $R_r$-submodule $S_2^{\infty}(\Upgamma_r; R_r)$.
The second equality of Proposition \ref{ALinv} then implies that $S_2^0(\Upgamma_r;R_r)$
is also a $\Gal(K_r'/K_0)$-stable $R_r$-submodule of $S_2(\Upgamma_r;K_r')$.  
A straightforward computation shows that 
the direct factor $\Gal(K_0'/K_0)$ of $\Gal(K_r'/K_0)$ acts trivially on $S_2^{\infty}(\Upgamma_r;R_r)$
and through $\langle a\rangle_N^{-1}$ on $S_2^0(\Upgamma_r;R_r)$.

We can interpret $S_2^{\star}(\Upgamma_r;R_r)$ geometrically as follows.
As in Remark \ref{MWGood}, for $\star= \infty, 0$
let $I_r^{\star}$ be the irreducible component 
of $\o{\X}_r$ passing through the cusp $\star$,
and denote by $\X_{r}^{\star}$ the complement
in $\X_r$ of all irreducible components of $\o{\X}_r$ distinct from $I_{r}^{\star}$.
By construction, $\X_r$ and $\X_r^{\star}$ have the same generic fiber $X_r\times_{\Q_p} K_r$.
Using Proposition \ref{redXr}, it is not hard to show that the diamond operators
induce automorphisms of $\X_r^{\star}$, and one checks via Proposition \ref{AtkinInertiaCharp}
that the ``semilinear" action (\ref{gammamaps}) of $\gamma\in \Gamma$ on $\X_r$
carries $\X_r^{\star}$ to $(\X_r^{\star})_{\gamma}$ for all $\gamma$.

\begin{lemma}\label{Edixhoven}
	Formal expansion at the $R_r$-point $\infty$ $($respectively $R_r'$-point $0$$)$ of $\X_r^{\star}$ 
	induces an isomorphism of $R_r$-modules
		\begin{equation}
			H^0(\X_r^{\infty},\Omega^1_{\X_r^{\infty}/R_r}) \simeq S_2^{\infty}(\Upgamma_r;R_r)
			\quad\text{respectively}\quad
			H^0(\X_r^{0},\Omega^1_{\X_r^{0}/R_r})(\langle a\rangle_N^{-1}) \simeq S_2^{0}(\Upgamma_r;R_r)
		\end{equation}
	which is equivariant for the natural actions of $\Gamma$ and $\H_r$ $($respectively $\H_r^*$$)$ on source
	and target and, in the case of the second isomorphism, intertwines the action of $\Gal(K_0'/K_0)$
	via $\langle a\rangle_N^{-1}$ on source with the natural action on the target.
\end{lemma}

\begin{proof}
	The proof is a straightforward
	adaptation of the proof of  \cite[Proposition 2.5]{EdixhovenComparison}.
\end{proof}	
	
Now $\X_r\rightarrow S_r$ is smooth outside the supersingular points, so there is a canonical
closed immersion $\iota_r^{\star}:\X_r^{\star}\hookrightarrow \X_r^{\sm}$.
Using Lemmas \ref{ConcreteDualizingDescription} and \ref{Edixhoven}, 
pullback of differentials along $\iota_r^{\star}$ gives a natural map 
\begin{equation}
	\xymatrix{
		{H^0(\X_r,\omega_{\X_r/T_r})\simeq H^0(\X_r^{\sm},\Omega^1_{\X_r^{\sm}/T_r})} \ar[r]^-{(\iota_r^{\star})^*} & 
		{H^0(\X_r^{\star},\Omega^1_{\X_r^{\star}/T_r}) \simeq S_2^{\star}(\Upgamma_r;R_r)}
		}\label{OmegasComparison}
\end{equation}
which is an isomorphism after inverting $p$ as $\X_r^{\sm}$ and $\X_r^{\star}$ have the same generic fiber.
In particular,
the map (\ref{OmegasComparison}) is injective, $\Gamma$ and $\H_r$ (respectively $\H_r^*$) equivariant
for $\star=\infty$ (respectively
$\star=0$), and in the case of $\star=0$ intertwines the action of $\Gal(K_0'/K_0)$ via 
$\langle a\rangle_N^{-1}$ on source with the natural action on the target.

\begin{remark}
	The image of $(\ref{OmegasComparison})$ for $\star=\infty$ 
	is naturally identified 
	with the space of weight $2$ cuspforms for $\Upgamma_r$ whose formal expansion
	at {\em every} cusp has $R_r$-coefficients. 
\end{remark}

Applying the idempotent $e$ (respectively $e^*$) to (\ref{OmegasComparison}) with $\star=\infty$
(respectively $\star=0$) gives an injective homomorphism
\begin{subequations}
\begin{equation}
	\xymatrix{
		{eH^0(\X_r,\omega_{\X_r/T_r})} \ar@{^{(}->}[r] & {eS_2^{\infty}(Np^r;R_r)}
		}\label{OmegasComparisonOrd}
\end{equation}
respectively
\begin{equation}
	\xymatrix{
		{e^*H^0(\X_r,\omega_{\X_r/T_r})(\langle a\rangle_N^{-1})} \ar@{^{(}->}[r] & {e^*S_2^0(Np^r;R_r)}
		}\label{OmegasComparisonOrd0}
\end{equation}
\end{subequations}
which is compatible with the canonical actions of $\Gamma$ and of $\H_r$ (respectively $\H_r^*$) on
source and target and in the case of (\ref{OmegasComparisonOrd}) is $\Gal(K_0'/K_0)$-equivariant.

\begin{proposition}\label{MFGeometryIsom}
	The mappings $(\ref{OmegasComparisonOrd})$ and $(\ref{OmegasComparisonOrd0})$ are isomorphisms.
\end{proposition}

\begin{proof}
	We treat the case of (\ref{OmegasComparisonOrd}); the proof that (\ref{OmegasComparisonOrd0})
	is an isomorphism goes through {\em mutatis mutandis}.
	We must show that (\ref{OmegasComparisonOrd}) is surjective.
	To do this, let $\nu\in e_rS_2^{\infty}(Np^r;R_r)$ be arbitrary.  Since (\ref{OmegasComparisonOrd})
	is an isomorphism after inverting $\varpi_r:=\varepsilon^{(r)}-1$, there exists a least nonnegative integer $d$
	such that $\varpi_r^d\nu$ is in the image of (\ref{OmegasComparisonOrd}).  
	Assume that $d\ge 1$, and let $\eta\in eH^0(\X_r,\omega_{\X_r/R_r})$
	be any element mapping to $\varpi_r^d \nu$.  For an irreducible component $I$ of $\o{\X}_r$,
	write $I^h$ for the complement of the supersingular points in $I$, and denote by
	$i_r^{\infty}:I_r^{\infty,h}\hookrightarrow \X_r^{\infty}$ the canonical immersion.
	We then have a commutative diagram
	\begin{equation}
	\begin{gathered}
		\xymatrix@C=40pt{
			{H^0(\o{\X}_r,\omega_{\o{\X}_r/R_r})} \ar[r]^-{(\ref{OmegasComparison})\bmod \varpi_r}\ar@{^{(}->}[d] & 
			{H^0(\X_r^{\infty},\Omega^1_{\X_r^{\infty}/R_r})\tens_{R_r} \F_p}\ar[d]^-{(i_r^{\infty})^*} \\
			{\displaystyle\prod\limits_{I\in \Irr(\o{\X}_r)} H^0(I^h,\Omega^1_{I^h/\F_p})}\ar[r]_-{\proj_{\infty}} & 
			{H^0(I_r^{\infty,h},\Omega^1_{I_r^{\infty,h}/\F_p})}
		}\label{etamaps20}
	\end{gathered}	
	\end{equation}
	where the left vertical mapping is deduced from (\ref{dualizing2prod}), 
	while the bottom map is simply projection.
	Our assumption that $d\ge 1$ implies that the image of $\o{\eta}:=\eta\bmod \varpi_r$ 
	under the composite of the right vertical and top horizontal maps in (\ref{etamaps20})
	is zero and hence, viewing
	$\o{\eta}=(\eta_{(a,b,u)})$ as a meromorphic differential on the normalization of $\o{\X}_r$,
	we have $\eta_{(r,0,1)}=\proj_{\infty}(\o{\eta})=0$.  
	Using the formula (\ref{Upn1}), we deduce that $U_p^n\o{\eta}=0$ for $n$ sufficiently large.
	But $U_p$ acts invertibly on $\eta$ (and hence on $\o{\eta}$) so we necessarily have that
	$\o{\eta}=0$ or what is the same thing that $\eta\bmod \varpi_r=0$.  We conclude that 
	$\varpi_r^{d-1}\nu$ is in the image of (\ref{OmegasComparisonOrd}), contradicting the minimality
	of $d$.  Thus $d=0$ and (\ref{OmegasComparisonOrd}) is surjective.
\end{proof}

For $s \le r$, Ohta shows \cite[2.3.4]{OhtaEichler} that the trace mapping 
$\tr_{\id}:S_k(\Upgamma_r; K_r)\rightarrow S_k(\Upgamma_s; K_s)\otimes_{K_s} K_r$
attached to the inclusion $\Upgamma_r\subseteq \Upgamma_s$
carries $S_k^0(\Upgamma_r;R_r)$ into $S_k^0(\Upgamma_s;R_s)\otimes_{R_s} R_r$, so that the projective limit
\begin{equation*} 
	\s_k^*(N,R_{\infty}) : = \varprojlim_{\tr_{\id}} \left( S_k^0(\Upgamma_r; R_r)\otimes_{R_r} R_{\infty}\right)
\end{equation*}
makes sense.  It is canonically a $\Lambda_{R_{\infty}}$-module, equipped with an action of $\H^*$,
a semilinear action of $\Gamma$, and a natural action of $\Gal(K_0'/K_0)$.
On the other hand,
let $eS(N;\Lambda_{R_{\infty}})\subseteq \Lambda_{R_{\infty}}[\![q]\!]$ 
be the space of ordinary $\Lambda_{R_{\infty}}$-adic
cuspforms of level $N$, as defined in \cite[2.5.5]{OhtaEichler}.
This space is equipped with an action of $\H$ via the usual formulae on formal
$q$-expansions (see, for example \cite[\S1.2]{WilesLambda}), as well as an action
of $\Gamma$ via its $q$-coefficient-wise action on $\Lambda_{R_{\infty}}[\![q]\!]$.

\begin{theorem}[Ohta]\label{OhtaThm}	
	Then there is a canonical isomorphism of $\Lambda_{R_{\infty}}$-modules 
	\begin{equation}
		\xymatrix{
			{eS(N;\Lambda_{R_{\infty}})(\langle \chi\rangle^{-1}\langle a\rangle_N^{-1})} \ar[r]^-{\simeq} & {e^*\s_2^*(N,R_{\infty})}
			}\label{LambdaForms}
	\end{equation}
	that intertwines the action of $T\in \H$ on the source with that of $T^*\in \H^*$
	on the target, for all $T\in \H$.  
	This isomorphism is $\Gal(K_{\infty}'/K_0)$-equivariant 
	for the natural action of $\Gal(K_{\infty}'/K_0)$ on the target,
	and the twisted action
	 $\gamma\cdot \scrF := \langle \chi(\gamma)\rangle^{-1}\langle a(\gamma)\rangle_N^{-1} 
	\gamma\scrF$ on the source.
\end{theorem}

\begin{proof}
	For the definition of the canonical map (\ref{LambdaForms}), as well as the proof that it is an isomorphism,
	see Theorem 2.3.6 and its proof in \cite{OhtaEichler}.  With the conventions of \cite{OhtaEichler},
	the claimed compatibility of (\ref{LambdaForms}) with Hecke operators is a consequence of \cite[2.5.1]{OhtaEichler},
	while the $\Gal(K_{\infty}'/K_0)$-equivariance of (\ref{LambdaForms}) follows from \cite[Proposition 3.5.6]{OhtaEichler}.
\end{proof}

\begin{corollary}\label{LambdaFormsRelation}
	There is a canonical isomorphism of $\Lambda_{R_{\infty}}$-modules 
	\begin{equation}
		 eS(N;\Lambda_{R_{\infty}})(\langle \chi\rangle^{-1})\simeq		e^* H^0(\omega) 
	\end{equation}
	that intertwines the action of $T\in \H$ on $eS(N;\Lambda_{R_{\infty}})$
	with $T^*\in \H^*$ on $e^*H^0(\omega)$, and
	is $\Gamma$-equivariant for the canonical action of $\Gamma$
	on the target and the twisted action $\gamma\cdot \scrF:=\langle \chi(\gamma)\rangle^{-1} \gamma\scrF$
	on the source.
\end{corollary}

\begin{proof}
	This follows immediately from Proposition \ref{MFGeometryIsom} and Theorem \ref{OhtaThm}.
\end{proof}

\setcounter{section}{0}
\setcounter{equation}{0}
\renewcommand{\theequation}{\thesection.\arabic{equation}}

\appendix
\section*{Appendices}

In the following appendices, 
we recall the technical geometric background used in our constructions.  Much (though not all)
of this material is contained in \cite{CaisDualizing}, \cite{CaisNeron}
and \cite{KM}.

\renewcommand{\thesubsection}{A} 
\numberwithin{theorem}{subsection}  
\numberwithin{equation}{subsection}

\subsection{Dualizing sheaves and cohomology}\label{GD}

We begin by describing a certain modification of the usual de Rham complex for non-smooth curves.  
The hypercohomology of this (two-term) complex is in general much better behaved than algebraic de Rham 
cohomology and will enable us to construct our $\Lambda$-adic de Rham cohomology.
We largely refer to \cite{CaisDualizing}, but remark that our
treatment here is different in some places and better suited to our purposes.

\begin{definition}\label{curvedef}
	A {\em curve} over a scheme $S$ is a morphism $f:X\rightarrow S$ 
	of finite presentation which is a flat local complete 
	intersection\footnote{That is, a {\em syntomic morphism} in the sense of
	Mazur \cite[II, 1.1]{FontaineMessing}.  Here, we use the definition of l.c.i. given 
	in \cite[Exp. \Rmnum{8}, 1.1]{SGA6}.}
	of pure relative dimension 1 with geometrically reduced fibers.
	We will often say that $X$ is a curve over $S$ or that $X$ is 
	a relative $S$-curve when $f$ is clear from context.
\end{definition}

\begin{proposition}\label{curveproperties}
	Let $f:X\rightarrow S$ be a flat morphism of finite type.  The following are equivalent:
	\begin{enumerate}
		\item The morphism $f:X\rightarrow S$ is a curve.\label{fiscrv}
		\item For every $s\in S$, the fiber $f_s:X_s\rightarrow \Spec k(s)$ is a curve.\label{fiberscrv}
		\item For every $x\in X$ with $s=f(x)$, the local ring 
		$\O_{X_s,x}$ is a complete intersection\footnote{That is, the quotient of a regular local ring by 
		a regular sequence.} and $f$ has geometrically reduced fibers of pure dimension 1.\label{localringcrv}
	\end{enumerate}
	Moreover, any base change of a curve is again a curve.
\end{proposition}

\begin{proof}
	Since $f$ is flat and of finite presentation, the definition of local complete 
	intersection that we are using ({\em i.e.} \cite[Exp. \Rmnum{8}, 1.1]{SGA6}) is equivalent
	to the definition given in \cite[$\mathrm{\Rmnum{4}}_4$, 19.3.6]{EGA} by \cite[Exp. \Rmnum{8}, 1.4]{SGA6};
	the equivalence of (\ref{fiscrv})--(\ref{localringcrv}) follows immediately.  
	The final statement of the proposition is an easy consequence of \cite[$\mathrm{\Rmnum{4}}_4$, 19.3.9]{EGA}.    
\end{proof}

\begin{corollary}\label{curvecorollary}
	Let $f:X\rightarrow S$ be a finite type morphism of pure relative dimension $1$.
	\begin{enumerate}
		\item If $f$ is smooth, then it is a curve.\label{smoothcrv}
		\item If $X$ and $S$ are regular and $f$ has geometrically reduced fibers
		then $f$ is a curve.\label{regcrv}
		\item If $f$ is a curve then it is Gorenstein and hence also Cohen-Macaulay.\label{crvCM}
	\end{enumerate}
\end{corollary}
			
\begin{proof}
	The assertion (\ref{smoothcrv}) is obvious, and (\ref{regcrv}) follows from the 
	fact that a closed subscheme of a regular scheme
	is regular if and only if it is defined (locally) by a regular sequence; {\em cf.} \cite[6.3.18]{LiuBook}.
	Finally, (\ref{crvCM})
	follows from Proposition \ref{curveproperties} (\ref{localringcrv}) and 
	the fact that every local ring
	that is a complete intersection is Gorenstein and hence Cohen-Macaulay  
	(see, e.g., Theorems 18.1 and 21.3 of \cite{matsumura}).  
\end{proof}

Fix a relative curve $f:X\rightarrow S$,
and assume that $S$ is noetherian.    
Since $f$ is CM by Corollary \ref{curvecorollary} (\ref{crvCM}), 
thanks to Theorem 3.5.1 and the discussion immediately following 3.5.2 in \cite{GDBC}, 
the {\em relative dualizing sheaf} for $X$ over $S$ (or for $f$),
denoted $\omega_{X/S}$ or $\omega_f$, exists.
As the fibers of $f$ are Gorenstein, $\omega_{X/S}$ is an invertible $\O_X$-module by \cite[V, Proposition 9.3, Theorem 9.1]{RD}. The formation of $\omega_{X/S}$ is moreover compatible
with arbitrary base change $S'\rightarrow S$ with $S'$ noetherian,
and \'etale localization on $X$ \cite[Theorem 3.6.1]{GDBC}.  In our situation,
one moreover has a canonical $\O_X$-module homomorphism:
	\begin{equation}
		\xymatrix{
			{c_{X/S}: \Omega^1_{X/S}} \ar[r] & {\omega_{X/S}}
			}\label{cmap}
	\end{equation}
constructed as in the proof of \cite[Th\'eor\`eme \Rmnum{3}.1]{elzeinapp}
(see also \cite[6.4.13]{LiuBook}).

\begin{proposition}\label{canmap}
	Let $X\rightarrow S$ be a relative curve. 
	The canonical map $(\ref{cmap})$ is of
	formation compatible with any base change $S'\rightarrow S$ with
	$S'$ is noetherian.  
 	 Moreover, the restriction of $c_{X/S}$ to the 
	 $S$-smooth locus $X^{\sm}$ of $f$ in $X$ is an isomorphism. 
\end{proposition}

\begin{proof}
	See \cite{elzeinapp}, especially Th\'eor\`eme \Rmnum{3}.1 and {\em cf.} \cite[\S6.4.2]{LiuBook}.
\end{proof}

\begin{definition}\label{complexregdiff}
	We define the two-term $\O_S$-linear complex (of $\O_S$-flat coherent $\O_X$-modules) concentrated in degrees 0 and 1
	\begin{equation}
		\xymatrix{
			{\omega_f^{\bullet}=\omega_{X/S}^{\bullet}:=\O_X} \ar[r]^-{d_S} & {\omega_{X/S}}
			}
	\end{equation}
	where $d_S$ is the composite of the map (\ref{cmap}) and the universal 
	$\O_S$-derivation $\O_X\rightarrow \Omega^1_{X/S}$.  We view $\omega_{X/S}^{\bullet}$
	as a filtered complex via ``{\em la filtration b\^ete}" \cite{DeligneHodge2},
	which provides an exact triangle
\begin{equation}
	\xymatrix{
		{\omega_{X/S}[-1]} \ar[r] & {\omega^{\bullet}_{X/S}} \ar[r] & {\O_X}
		}\label{HodgeFilComplex}
\end{equation}
	in the derived category that we call the {\em Hodge Filtration} of $\omega^{\bullet}_{X/S}$.
\end{definition}

Since $c_{X/S}$ is an isomorphism over the $S$-smooth locus $X^{\sm}$ of $f$ in $X$, the complex $\omega^{\bullet}_{X/S}$ coincides with
the usual de Rham complex over $X^{\sm}$.  Moreover, it follows immediately from 
Proposition \ref{canmap} that the formation of $\omega_{X/S}^{\bullet}$ is compatible with 
any base change $S'\rightarrow S$ to a noetherian scheme $S'$.

\begin{definition}
	Let $f:X\rightarrow S$ be a relative curve over $S$.  For each nonnegative integer $i$, we define
	\begin{equation*}
		\mathscr{H}^i(X/S):=\R^i f_*\omega_{X/S}^{\bullet}.
	\end{equation*}
	When $S=\Spec R$ is affine, we will write $H^i(X/R)$ for the global sections of the $\O_S$-module 
	$\mathscr{H}^i(X/S)$.
\end{definition}

The complex $\omega_{X/S}^{\bullet}$ and its filtration (\ref{HodgeFilComplex})
behave extremely well with respect to duality:

\begin{proposition}\label{GDuality}
	Let $f:X\rightarrow S$ be a curve over a noetherian scheme $S$ and assume that
	$S$ is Gorenstein, excellent, and of finite Krull dimension.
	  There is a canonical
	quasi-isomorphism 
	\begin{equation}
		\omega_{X/S}^{\bullet} \simeq \R\scrHom_X^{\bullet}(\omega_{X/S}^{\bullet},\omega_{X/S}[-1])
		\label{DualityIsom}
	\end{equation}
	which is compatible with the filtrations on both sides induced by $(\ref{HodgeFilComplex})$.
	In particular:
	\begin{enumerate}
		\item If $f$ is proper then there is a natural quasi-isomorphism
		\begin{equation*}
			\R f_*\omega^{\bullet}_{X/S}\simeq \R\scrHom_S^{\bullet}(\R f_*\omega^{\bullet}_{X/S},\O_S)[-2]			
		\end{equation*}
		which is compatible with the filtrations induced by $(\ref{HodgeFilComplex})$.\label{DualityOnS}
		\item If $\rho:Y\rightarrow X$ is any finite morphism of curves over $S$, 
		then there is a canonical quasi-isomorphism
	\begin{equation*}
		\R\rho_*\omega_{Y/S}^{\bullet} \simeq \R\scrHom_X^{\bullet}
		(\R\rho_*\omega_{Y/S}^{\bullet},\omega_{X/S}[-1]).
	\end{equation*}
		that is compatible with filtrations.\label{DualityRho}
		
		\end{enumerate}
\end{proposition}

\begin{proof}
	Our hypotheses on $S$  ensure that $\O_S$ is a dualizing complex for $S$ \cite[V,\S10]{RD},
	and it then follows from \cite[\Rmnum{5}, \S8]{RD} that
	the sheaf $\omega_{X/S}$ (thought of as a complex concentrated in some degree) is a dualizing complex 
	for the abstract scheme $X$.   
	To prove the first claim, we may then argue as in the proofs
	of Lemmas 4.3 and 5.4 of \cite{CaisDualizing}, noting that although
	$S$ is assumed to be the spectrum of a discrete valuation ring and the definition of curve in that paper 	
	differs somewhat from the definition here, the arguments themselves apply {\em verbatim} in our 	
	context.  The assertion (\ref{DualityOnS}) (respectively (\ref{DualityRho})) follows from this by 
	applying $\R f_*$ (respectively $\R\rho_*$) to both sides of (\ref{DualityIsom}) and appealing 
	to Grothendieck duality
	\cite[Theorem 3.4.4]{GDBC} for the proper map $f$ (respectively $\rho$); see the proofs
	of Lemma 5.4 and Proposition 5.8 in \cite{CaisDualizing} for details.
\end{proof}

In our applications, we need to understand the cohomology $H^i(X/S)$ for a proper
curve $X\rightarrow S$ when $S$ is either the spectrum of a discrete valuation
ring $R$ of mixed characteristic $(0,p)$ or the spectrum of a perfect field.
We now examine each of these situations in more detail.

First suppose that $S:=\Spec(R)$ is the spectrum of a discrete valuation ring $R$ having field of fractions $K$ of characteristic zero and  
residue field $k$ of characteristic $p>0$, and fix a normal curve $f:X\rightarrow S$ with smooth and geometrically connected
generic fiber $X_K$.   This situation is studied extensively
in \cite{CaisDualizing}, and we content ourselves with a summary 
of the results we will need.
To begin, we recall the following ``concrete" description of the relative dualizing sheaf:

\begin{lemma}\label{ConcreteDualizingDescription}
	Let $i:U\hookrightarrow X$ be any Zariski open subscheme of $X$ 
	whose complement consists of finitely many points of codimension $2$
	$($necessarily in the closed fiber of $X$$)$.  Then the canonical map
	\begin{equation*}
		\xymatrix{
			{\omega_{X/S}} \ar[r] & {i_*i^*\omega_{X/S} \simeq i_*\omega_{U/S}}
			}
	\end{equation*}
	is an isomorphism.  In particular, $\omega_{X/S}\simeq i_*\Omega^1_{U/S}$
	for any Zariski open subscheme $i:U\hookrightarrow X^{\sm}$ whose
	complement in $X$ consists of finitely many points of codimension two.	
\end{lemma}

\begin{proof}
	The first assertion is \cite[Lemma 3.2]{CaisNeron}. The second
	follows from this, since $X^{\sm}$ contains the generic fiber
	and the generic points of the closed fiber by our definition of curve.
\end{proof}

\begin{proposition}\label{ComplexFunctoriality}
	Let $\rho:Y\rightarrow X$ be a finite morphism of normal $S$-curves.
	\begin{enumerate}
		\item Attached to $\rho$ are natural
		pullback and trace morphisms of complexes 
		\begin{equation*}
			\xymatrix{
				{\rho^*: \omega^{\bullet}_{X/S}} \ar[r] & {\rho_*\omega^{\bullet}_{Y/S}}
			}
			\quad\text{and}\quad
						\xymatrix{
				{\rho_*: \rho_*\omega^{\bullet}_{Y/S}} \ar[r] & {\omega^{\bullet}_{X/S}}
			}
		\end{equation*}
		which are of formation compatible with \'etale localization on $X$ and 
		flat base change on $S$ and
		are interchanged by applying $\R\scrHom_X^{\bullet}
		(\cdot,\omega_{X/S}[-1])$ via the duality of Proposition $\ref{GDuality}$ $(\ref{DualityRho})$.
		\label{FunctorialityProps1}
			
			\item Let $U$ be any open subscheme of $X^{\sm}$
			with the property that $V:=\rho^{-1}(U)$ is contained in $Y^{\sm}$.	
			Then the induced pullback and trace
			mappings $\omega_{U/S}^{\bullet}\leftrightarrows \rho_*\omega^{\bullet}_{V/S}$
			coincide with the usual pullback and trace mappings on de Rham complexes
			attached to the finite and flat map $\rho:V\rightarrow U$.			
\label{FunctorialityProps2}
			
	\end{enumerate}
\end{proposition}

\begin{proof}
	The assertions of (\ref{FunctorialityProps1}) follow from the proofs of Propositions 4.5 and 5.5 
	of \cite{CaisDualizing}, while by $S$-flatness of $X$ and $Y$, the verification
	of (\ref{FunctorialityProps2}) reduces to checking the particular case $U=X_K$ and $V=Y_K$,
	which follows easily from the very definitions of $\rho_*$ and $\rho^*$
	in \cite[\S4]{CaisDualizing}.
\end{proof}

We henceforth assume that the normal $S$-curve $X$ is in addition proper over $S$.
Then as $X_K$ is a smooth and proper curve over $K$, 
the Hodge to de Rham spectral sequence degenerates \cite{DeligneIllusie}, and there
is a functorial short exact sequence of $K$-vector spaces
\begin{equation}
	\xymatrix{
		0\ar[r] & {H^0(X_K,\Omega^1_{X_K/K})} \ar[r] & {H^1_{\dR}(X_K/K)} \ar[r] & {H^1(X_K,\O_{X_K})} \ar[r] & 0
	}\label{HodgeFilCrv}
\end{equation}
which we call the {\em Hodge filtration} of $H^1_{\dR}(X_K/K)$.

\begin{proposition}\label{HodgeIntEx}
	Let $f:X\rightarrow S$ be a normal curve that is proper over $S=\Spec(R)$.  
	\begin{enumerate}
		\item There are natural isomorphisms of free $R$-modules of rank $1$
		\begin{equation*}
			H^0(X/R)\simeq H^0(X,\O_X)\quad\text{and}\quad H^2(X/R)\simeq H^1(X,\omega_{X/S}),
		\end{equation*}	
		which are canonically $R$-linearly dual to each other.

	\item There is a canonical short exact sequence of finite free $R$-modules, 
	which we denote $H(X/R)$,
		\begin{equation*}
			\xymatrix{
					0\ar[r] & {H^0(X,\omega_{X/S})} \ar[r] & {H^1(X/R)} \ar[r] & {H^1(X,\O_X)} \ar[r] & 0
			}
		\end{equation*}
	that recovers the Hodge filtration $(\ref{HodgeFilCrv})$ of $H^1_{\dR}(X_K/K)$ after
	extending scalars to $K$.  \label{CohomologyIntegral}
	
	\item Via the canonical cup-product auto-duality of $(\ref{HodgeFilCrv})$, 
	the exact sequence $H(X/R)$ is naturally isomorphic to
	its $R$-linear dual.\label{CohomologyDuality}
	
	\item The exact sequence $H(X/R)$ is contravariantly
	$($respectively covariantly$)$ functorial in finite morphisms $\rho:Y\rightarrow X$
	of normal and proper $S$-curves via pullback $\rho^*$ $($respectively trace $\rho_*$$)$;
	these morphisms recover the usual pullback and trace mappings on Hodge filtrations after extending scalars 
	to $K$ and are adjoint with respect to the canonical cup-product autoduality of $H(X/R)$
	in $(\ref{CohomologyDuality})$. \label{CohomologyFunctoriality} 
	
	\end{enumerate}
\end{proposition}

\begin{proof}
	By Raynaud's ``{\em crit\`ere de platitude cohomologique}" \cite[Th\'eor\`me 7.2.1]{Raynaud}
	(see also \cite[Proposition 2.7]{CaisDualizing}), our requirement that curves have geometrically 
	reduced fibers implies that $f:X\rightarrow S$ is cohomologically flat.\footnote{In other words, the $\O_S$-module $f_*\O_X$ commutes with 
	arbitrary base change.}
	The proposition now follows from Propositions 5.7--5.8 of \cite{CaisDualizing}.
\end{proof}

We now turn to the case that $S=\Spec(k)$ for a field $k$ and
$f:X\rightarrow S$ is a proper and geometrically connected curve over $k$.
Recall that $X$ is required to be geometrically reduced, so that the
$k$-smooth locus $U:=X^{\sm}$ is the complement of finitely many closed
points in $X$.  

\begin{proposition}\label{HodgeFilCrvk} 
	Let $X$ be a proper and geometrically connected curve over $k$.  
	\begin{enumerate}
		\item There are natural isomorphisms of 1-dimensional $k$-vector spaces 
		\begin{equation*}
			H^0(X/k)\simeq H^0(X,\O_X)\quad\text{and}\quad H^2(X/k)\simeq H^1(X,\omega_{X/k}).
		\end{equation*}	
		\label{H0H2overk}
	
		\item There is a natural short exact sequence, which we denote $H(X/k)$
		\begin{equation*}
			\xymatrix{
				0 \ar[r] & {H^0(X,\omega_{X/k})} \ar[r] & {H^1(X/k)} \ar[r] & 
				{H^1(X,\O_{X})} \ar[r] & 0
			}.\label{HodgeDegenerationField}
		\end{equation*}
		\label{HodgeExSeqk}
	\end{enumerate}
\end{proposition}

\begin{proof}
	Consider the long exact cohomology sequence arising from the exact
	triangle (\ref{HodgeFilComplex}).
	Since $X$ is proper over $k$, geometrically connected and reduced,  
	the canonical map $k\rightarrow H^0(X,\O_X)$ is an isomorphism, and it follows that
	the map $d:H^0(X,\O_X)\rightarrow H^0(X,\omega_{X/k})$ is zero,  
	whence the map $H^0(X/k)\rightarrow H^0(X,\O_X)$ is an isomorphism.
	Thanks to Proposition \ref{GDuality} (\ref{DualityOnS}), we have a canonical
	quasi-isomorphism 
	\begin{equation}
		\R\Gamma(X,\omega_{X/k}^{\bullet})\simeq 
		\R\Hom_k^{\bullet}(\R\Gamma(X,\omega_{X/k}^{\bullet}),k)[-2]\label{GDFieldExplicit}
	\end{equation}
	that is compatible with the filtrations induced by (\ref{HodgeFilComplex}).
	Using the double complex spectral sequence
	\begin{equation*}
		E_2^{m,n}:=\Ext_k^m(\mathbf{H}^{-n}(X,\omega_{X/k}^{\bullet}),k)
		\implies H^{m+n}(\R\Hom_k^{\bullet}(\R\Gamma(X,\omega_{X/k}^{\bullet}),k))
	\end{equation*}
	and the vanishing of $\Ext_k^m(\cdot,k)$ for $m>0$,
	we deduce that $H^2(X/k)\simeq H^0(X/k)^{\vee}$ is 1-dimensional over $k$.
	Since Grothendieck's trace map $H^1(X,\omega_{X/k})\rightarrow k$ is an isomorphism,
	we conclude that the {\em surjective} map of 1-dimensional $k$-vector spaces 
	$H^1(X,\omega_{X/k})\rightarrow H^2(X/k)$ must be an isomorphism.  It follows that
	the map $d:H^1(X,\O_X)\rightarrow H^1(X,\omega_{X/k})$ is zero as well, as desired.
\end{proof}

Using the canonical map of complexes 
$\omega^{\bullet}_{X/k}\otimes_{\O_X} \omega^{\bullet}_{X/k}\rightarrow \omega^{\bullet}_{X/k}$,
we obtain from the identification (\ref{H0H2overk}) of Proposition \ref{HodgeFilCrvk} 
and Grothendieck's trace map a natural {\em cup product} pairing
\begin{equation*}
	\xymatrix{
		{\langle\cdot,\cdot\rangle_{X}: H^1(X/k)\times H^1(X/k)}\ar[r] & 
		{\mathbf{H}^2(X,\omega^{\bullet}_{X/k}\tens_{\O_X} \omega^{\bullet}_{X/k})}\ar[r] &
		{H^2(X/k) \simeq H^1(X,\omega_{X/k})} \ar[r]^-{\Tr} & k
		} 
\end{equation*}
under which the submodule $H^0(X,\omega_{X/k})$ annihilates itself.  This pairing induces a
map 
\begin{equation}
\begin{gathered}
	\xymatrix{
		0 \ar[r] & {H^0(X,\omega_{X/k})}\ar[d] \ar[r] & {H^1(X/k)} \ar[r]\ar[d] & {H^1(X,\O_{X})} \ar[r]\ar[d] & 0 \\
		0 \ar[r] & {H^1(X,\O_{X})^{\vee}} \ar[r] & {H^1(X/k)^{\vee}} \ar[r] & {H^0(X,\omega_{X/k})^{\vee}} 
		\ar[r] & 0
	}
	\end{gathered}
	\label{DualityCrvk}
\end{equation}
of short exact sequences of $k$-vector spaces. 

\begin{proposition}
	The map $(\ref{DualityCrvk})$ is an isomorphism of short exact sequences.
\end{proposition}

\begin{proof}
	The flanking vertical maps are isomorphisms by \cite[Theorem 5.1.2]{GDBC}, 
	so the middle one must be as well.
\end{proof}

When $k$ is algebraically closed, the sheaf $\omega_{X/k}$
admits a beautiful explicit description, due to Rosenlicht
\cite{Rosenlicht}, in terms of meromorphic differentials
on the normalization of $X$.  We briefly recall this description,
and refer the reader to \cite[\S5.2]{GDBC} for details.

Let $k(X):=\prod_i k(\xi_i)$
be the ``function field" of $X$, by definition the product of the residue fields at the finitely many generic points of
$X$, and write 
$\u{\Omega}^1_{k(X)/k}$ 
for the pushforward of $\Omega^1_{k(X)/k}$ along the inclusion $\Spec(k(X))\hookrightarrow X$.
As $X$ is reduced, it is smooth
at its generic points, so this inclusion factors through 
$i:X^{\sm}\hookrightarrow X$. By \cite[Lemma 5.2.1]{GDBC}, the canonical map
of $\O_X$-modules
\begin{equation}
	\xymatrix{
		{\omega_{X/k}} \ar[r] & {i_*i^*\omega_{X/k}\simeq i_*\Omega^1_{X^{\sm}/k}}
		}\label{OmegaSubSheaf}
\end{equation}  
is injective, and it follows that $\omega_{X/k}$ is a subsheaf of
$\u{\Omega}^1_{k(X)/k}$. 
Writing $\pi:\nor{X}\rightarrow X$ for the normalization map,
we have $k(X)=k(\nor{X})$ and hence an identification
$\u{\Omega}^1_{k(X)/k}\simeq \pi_*\u{\Omega}^1_{k(\nor{X})/k}$.
For $x\in X(k)$, define 
\begin{equation}
	\xymatrix{
		{\res_x : \u{\Omega}^1_{k(X)/k,x}}\ar[r] & k
		}
	\ \text{by}\ 	
		\res_x(\eta):=\sum_{y\in \pi^{-1}(x)} \res_y(\eta),
		\label{Def:resmaps}
\end{equation}		
where $\res_{y}$ is the classical residue map on meromorphic differentials on
the smooth curve $\nor{X}$ over $k$.

\begin{proposition}[Rosenlicht]\label{Rosenlicht}
	Let $X$ be a proper and geometrically connected curve over $k$.
	As a subsheaf of $\u{\Omega}^1_{k(X)/k}\simeq \pi_*\u{\Omega}^1_{k(\nor{X})/k}$,
	the sections of $\omega_{X/k}$ over any open $V\subseteq X$ are precisely
	those meromorphic differentials on $\pi^{-1}(V)\subseteq \nor{X}$ that satisfy
	$\res_x(s\eta)=0$ for all $x\in V(k)$ and all $s\in \O_{X,x}$.
\end{proposition}

\begin{proof}
	See \cite[Theorem 5.2.3]{GDBC}.
\end{proof}

We now suppose that $S=\Spec(R)$ for a discrete valuation ring $R$ with fraction field $K$
of characteristic zero and residue field $k$ of characteristic $p>0$.

\begin{lemma}\label{ReductionCompatibilities}
	Let $X$ be a normal curve over $S=\Spec(R)$ with smooth and geometrically
	connected generic fiber, and denote by $\o{X}:=X_k$
	the special fiber of $X$; it is a geometrically connected curve over $k$ 
	by Proposition $\ref{curveproperties}$ $(\ref{fiberscrv})$.
	\begin{enumerate}
		\item If $X$ is proper, then the canonical base change maps
		\begin{equation*}
			\xymatrix{
				0 \ar[r] & {H^0(X,\omega_{X/S})\tens_R k} \ar[r]\ar[d]^-{\simeq}_-{\nu_{0}} & 
				{H^1(X/R) \tens_R k} \ar[r]\ar[d]^-{\simeq}_-{\nu_1} & 
				{H^1(X,\O_X)\tens_R k} \ar[r]\ar[d]^-{\simeq}_-{\nu_2} & 0\\			
				0\ar[r] & {H^0(\o{X},\omega_{\o{X}/k})} \ar[r] & {H^1(\o{X}/k)} \ar[r] & 
				{H^1(\o{X},\O_{\o{X}})} \ar[r] & 0
			}
		\end{equation*}
		are isomorphism, and via the auto-duality of each row 
		one has ${\nu}_{i}^{\vee} = \nu_{2-i}^{-1}$.
		
		 \label{BaseChngDiagram}  
		
		\item Let $\rho:Y\rightarrow X$ be a finite morphism of
		normal curves over $S$ with smooth and geometrically connected 
		generic fibers.		
		The canonical diagrams $($one for $\rho^*$ and one for $\rho_*$$)$
		\begin{equation*}
			\xymatrix{
				{H^0(Y,\omega_{Y/S})\tens_R k} \ar@<0.5ex>[d]^-{\rho_*\otimes 1}\ar[r] & 
				{H^0(\o{Y},\omega_{\o{Y}/k})}\ar@{^{(}->}[r]_-{(\ref{OmegaSubSheaf})} &
				{H^0(\nor{\o{Y}}, \Omega^1_{k(\nor{\o{Y}})/k})} \ar@<0.5ex>[d]^-{\nor{\o{\rho}}_*} \\
				\ar@<0.5ex>[u]^-{\rho^*\otimes 1} {H^0(X,\omega_{X/S})\tens_R k} \ar[r] &
				 {H^0(\o{X},\omega_{\o{X}/k})}\ar@{^{(}->}[r]_-{(\ref{OmegaSubSheaf})} & 
				\ar@<0.5ex>[u]^-{{\nor{\o{\rho}}}^*} {H^0(\nor{\o{X}}, \Omega^1_{k(\nor{\o{X}})/k})} 
			}			
		\end{equation*}
		commute, where ${\nor{\o{\rho}}}^*$ and $\nor{\o{\rho}}_*$ are the usual pullback
		and trace morphisms on meromorphic differential forms associated to the finite
		flat map $\nor{\o{\rho}}:\nor{\o{Y}}\rightarrow \nor{\o{X}}$
		of smooth curves over $k$.\label{PTBCCompat}	
	\end{enumerate}	
\end{lemma}

\begin{proof}
	Since $X$ is of relative dimension 1, when it is in addition proper,
	the cohomologies $H^1(X,\O_X)$ and $H^1(X,\omega_{X/S})$
	both commute with base change, and are free over $R$ by Proposition \ref{HodgeIntEx}. 
	In that case, we conclude that $H^i(X,\O_X)$ and $H^i(X,\omega_{X/S})$ commute with base change for all $i$ and
	hence that the left and right vertical maps in the base change diagram (\ref{BaseChngDiagram}) (whose
	rows are exact by Propositions \ref{HodgeIntEx} and \ref{HodgeFilCrvk}) 
	are isomorphisms.  It follows that the middle vertical map in 
	(\ref{BaseChngDiagram}) is an isomorphism as well.	
	Writing $\langle \cdot,\cdot\rangle_{\star}$ with $\star=X,\o{X}$ for the canonical duality pairing
	 on $H^1(X/R)$ and $H^1(\o{X}/k)$, respectively, our assertion about the behavior of the base change isomorphisms $\nu_i$
	under duality amounts to the equality
	\begin{equation}
	  \nu\left( \o{\langle x , y\rangle}_{X}\right) = \langle \nu_1(\o{x}), \nu_1(\o{y}) \rangle_{\o{X}} 
	  \label{BCpairing}
	\end{equation}
	for all $x,y \in H^1(X/R)$; here $\overline{(\cdot)}: M\rightarrow M\otimes_R k$ is the natural reduction map
	for any $R$-module $M$ and $\nu: R\otimes_R k\rightarrow k$ is the canonical isomorphism.
	 But the duality pairing is given by
	cup-product followed by Grothendieck's trace morphism, both of which are compatible
	with base change, whence the equality (\ref{BCpairing}).
	The compatibility of pullback and trace under base change to the special fibers, 
	as asserted by the diagram in (\ref{PTBCCompat}),
	is a straightforward consequence of Proposition \ref{ComplexFunctoriality} 
	(\ref{FunctorialityProps2}), using the facts that $X$ and $Y$ are 
	smooth at generic points of closed fibers and that $\o{\rho}:\o{Y}\rightarrow\o{X}$
	takes generic points to generic points as noted in the proof of Lemma \ref{ConcreteDualizingDescription}.
\end{proof}

We end this appendix with a brief discussion
of correspondences on curves and their induced action on cohomology and Jacobians.
Fix a ring $R$ and a proper normal curve $X$ over $S=\Spec R$.  Throughtout this discussion, we assume either that 
$R$ is a discrete valuation ring of mixed characteristic $(0,p)$, 
or that $R$ is a perfect field (and hence the normal $X$ is smooth). 

\begin{definition}\label{CorrDef}
	A {\em correspondence $T:=(\pi_1,\pi_2)$ on $X$} is an ordered pair $\pi_1,\pi_2:Y\rightrightarrows X$
	of finite morphisms of normal and proper curves over $S$.
	  The {\em transpose} 
	of a correspondence $T:=(\pi_1,\pi_2)$ on $X$ is the correspondence
	on $X$ given by the ordered pair $T^*:=(\pi_2,\pi_1)$. 
\end{definition}

Thanks to Proposition \ref{HodgeIntEx} (\ref{CohomologyFunctoriality}), any correspondence
$T=(\pi_1,\pi_2)$ on $X$ induces an $R$-linear endomorphism of the short exact sequence
$H(X/R)$ via ${\pi_1}_*\pi_2^*$.  By a slight abuse of notation, we denote this endomorphism
by $T$;  as endomorphisms of $H(X/R)$ we then have
\begin{equation}
	T= {\pi_1}_*\pi_2^*\qquad\text{and}\qquad T^* = {\pi_2}_* \pi_1^*.
	\label{HeckeDef}
\end{equation}
Given a finite map $\pi:X\rightarrow X$, we will consistently view $\pi$ as a correspondence on $X$
via the association $\pi\rightsquigarrow (\id,\pi)$.  
In this way, we may think of correspondences
on $X$ as ``generalized endomorphisms."  This perspective can be made more compelling as follows.

Suppose that $R$ is a field, and fix a correspondence $T$ given by an ordered pair
$\pi_1,\pi_2:Y\rightrightarrows X$ of finite morphisms of smooth and proper curves.
Then $T$ and its transpose $T^*$ induce endomorphisms of the Jacobian $J_X:=\Pic^0_{X/R}$ of $X$, which we again
denote by the same symbols, via
\begin{equation} 
	T:=\Alb(\pi_2)\circ \Pic^0(\pi_1)\qquad\text{and}\qquad T^*:= \Alb(\pi_1) \circ \Pic^0(\pi_2).
	\label{JacHecke}
\end{equation}
Here, for any finite map $\pi:Y\rightarrow X$ of smooth and proper curves over a field, 
$\Pic^0(\pi):J_X\rightarrow J_Y$ and $\Alb(\pi):J_Y\rightarrow J_X$ are the
maps on Jacobians functorially induced by pullback and pushforward of Weil divisors along $\pi$, respectively.
Note that when $T=(\id,\pi)$ for a morphism $\pi:X\rightarrow X$, the induced endomorphisms
(\ref{JacHecke}) of $J_X$ are given by $T=\Alb(\pi)$ and $T^*:=\Pic^0(\pi)$.\footnote{Because of this fact, 
for a general correspondence $T=(\pi_1,\pi_2)$,
the literature often refers to the induced endomorphism $T$ (respectively $T^*$) of $J_X$
as the {\em Albanese} (respectively {\em Picard}) or 
{\em covariant} (respectively {\em contravariant}) action of the correspondence
$(\pi_1,\pi_2)$.  Since the definitions (\ref{JacHecke}) of $T$ and $T^*$ {\em both}
literally involve Albanese and Picard functoriality, we find this old terminology
confusing, and eschew it in favor of the consistent notation we have introduced. 
}
Abusing notation, we will simply write $\pi$ for the endomorphism $\Alb(\pi)$ of $J_X$
induced by the correspondence $(1,\pi)$, and $\pi^*$ for the endomorphism $\Pic^0(\pi)$
induced by $(\pi,1) = (1,\pi)^*$.  When $\pi:X\rightarrow X$ is an automorphism, an easy argument 
shows that $\pi^* = \pi^{-1}$ as automorphisms of $J_X$.

With these definitions, the canonical filtration-compatible isomorphism 
$H^1_{\dR}(X/R) \simeq H^1_{\dR}(J_X/R)$ is $T$ (respectively $T^*$)-equivariant with respect to 
the action (\ref{HeckeDef}) on $H^1_{\dR}(X/R)$ and the action on $H^1_{\dR}(J_X/R)$
induced by pullback along the endomorphisms (\ref{JacHecke}); see \cite[Proposition 5.4]{CaisNeron}.

\renewcommand{\thesubsection}{B} 

\subsection{Integral models of modular curves}\label{tower}

We record some basic facts about integral models of modular curves that will be needed in what follows.
We use \cite{KM} as our basic reference, and freely utilize the notation and terminology 
therein. Throughout, we keep the notation of \S\ref{Notation}.

\begin{definition}\label{Def:MainModuliProblem}
	Let $r$ be a nonnegative integer and $R$ a ring containing a fixed choice $\zeta$ of 
	primitive $p^r$-th root of unity (by definition \cite[9.1.1]{KM}, a root of the $p^r$-th
	cyclotomic polynomial) in which $N$ is invertible.
	The moduli problem 
	$\scrP_{r}^{\zeta}:=([\bal\ \Upgamma_1(p^r)]^{\zeta\can}; [\mu_N]])$ 
	on $(\Ell/R)$
	assigns to $E/S$ the set of quadruples $(\phi:E\rightarrow E',P,Q ; \alpha)	$
	where:
	\begin{enumerate}
		\item $\phi:E\rightarrow E'$ is a cyclic $p^r$-isogeny.
		\item $P\in (\ker\phi)(S)$ and $Q\in (\ker\phi^t)(S)$ are generators of
		$\ker\phi$ and $\ker\phi^t$, respectively, 
		which pair to $\zeta$ under the canonical pairing
		$\langle\cdot,\cdot\rangle_{\phi}: \ker\phi\times\ker\phi^t\rightarrow \mu_{\deg\phi}$
		\cite[\S2.8]{KM}.
		\item $\alpha$ is a $\mu_N$-structure on $E/S$, {\em i.e.}
		a closed immersion $\alpha:\mu_N\hookrightarrow E[N]$ of $S$-group schemes.\label{muNstr}
	\end{enumerate}	
\end{definition}

\begin{proposition}\label{XrRepresentability}
	If $N \ge 4$, then the moduli problem $\scrP_{r}^{\zeta}$ is represented
	by a regular scheme $\M(\scrP_r^{\zeta})$ that is flat of pure relative dimension
	$1$ over $\Spec(R)$.   
\end{proposition}

\begin{proof}
	Using that $N$ is a unit in $R$, one first shows that for $N\ge 4$, the moduli problem
	$[\mu_N]$ of $\mu_N$-structures in the sense of (\ref{muNstr}) on $(\Ell/R)$ is 
	representable over $\Spec(R)$ and finite \'etale;
	this follows from  2.7.4, 3.6.0, 4.7.1 and 5.1.1 of \cite{KM}, as $[\mu_N]$
	is isomorphic to $[\Upgamma_1(N)]$ over any $R$-scheme containing a fixed choice 
	of primitive $N$-th root of unity (see also \cite[8.4.11]{KM}).  By \cite[4.3.4]{KM}, 
	it is then enough to show that
	$[\bal \Upgamma_1(p^r)]^{\zeta\can}$ on $(\Ell/R)$ is relatively representable and
	regular, which (via \cite[9.1.7]{KM}) is a consequence of \cite[7.6.1 (2)]{KM}.
\end{proof}

As in \cite[8.2]{KM}, for any relatively representable moduli problem $\scrP$
that is finite over $(\Ell/R)$, 
there is a canonical ``$j$-line" morphism of moduli problems $\scrP\rightarrow [\Gamma(1)]$
on $(\Ell/R)$ that makes the coarse moduli scheme $\M(\scrP)$ \cite[8.1]{KM} 
into a scheme over $\Aff^1_R:=\Spec(R[j])$.    
Using this structure map, we would like to extend $\M(\scrP)$
to a scheme $\c{\M}(\scrP)$ over $\P^1_R$ by ``normalizing near infinity"
via the procedure described in \cite[8.6.3]{KM}.  
Quite generally, assume that $Y$ is a curve over $R$
equipped with a finite morphism $Y\rightarrow \Aff_R^1:=\Spec(R[j])$
making $Y$ ``normal near infinity" \cite[8.6.2 C2]{KM}
in the sense that there exists a monic polynomial $f\in R[j]$ with the property that 
$Y$ is normal over the open subscheme $U$ of $\Aff^1_R$
where $f(j)\neq 0$.  Writing $D:=V(f)\subseteq \Aff^1_R$, 
we then denote by $\c{Y}$ the glueing of
$Y$ with the normalization of $\c{U}:=\P^1_R-D$
in $Y\big|_U$ over $Y\big|_U$.
Intrinsically, $Y^c$ is the unique scheme
that is finite over $\P^1_R$, agrees with $Y$
over $\Aff^1_{R}$, and is normal over an open neighborhood
of the $\infty$-section in $\P^1_R$.
  
\begin{proposition}\label{XrCptRepresentability}
	If $N \ge 4$, then the compactified moduli scheme $\c{\M}(\scrP_r^{\zeta})$ is regular and proper flat of pure
	relative dimension $1$ over $\Spec(R)$.
\end{proposition}

\begin{proof}
	This follows from \cite[10.9.1 (2)]{KM}; see also \cite[10.9.3]{KM}
	and {\em cf.} \cite[8.6.8 (2)]{KM}.
\end{proof}

In what follows, we will frequently need to extend
the map of moduli schemes $\M(\scrP)\rightarrow \M(\scrQ)$
corresponding to a morphism of representable moduli problems
$\scrP\rightarrow \scrQ$
that are finite over $(\Ell/R)$,
and normal near infinity, to a morphism $\c{\M}(\scrP)\rightarrow \c{\M}(\scrQ)$
on compactified moduli schemes.  That such an extension
exists is clear from the very construction of $\c{\M}(\cdot)$
when $\scrP\rightarrow \scrQ$ is a morphism {\em over} $(\Ell/R)$,
as in this case the induced mapping $\M(\scrP)\rightarrow \M(\scrQ)$
is compatible with the canonical projections to the affine $j$-line
by which the compactifications are built via normalization near infinity.
As we will need to consider ``abstract" (or exotic) morphisms $\scrP\rightarrow \scrQ$ \cite[11.2]{KM}
which are not necessarily morphisms of moduli problems
on $(\Ell/R)$, an additional argument for why the desired 
morphism on compactified moduli schemes exists is required.
The key is the following Lemma:

\begin{lemma}\label{UlmerLemma}
	Let $Y$ be a curve over $R$ equipped with two finite morphisms
	$Y\rightrightarrows \Aff^1_R$ such that $Y$ is normal near infinity
	with respect to each map.
	Then the two compactifications 	of $Y$ constructed as above using either projection to the affine line
	are canonically isomorphic.
\end{lemma}

\begin{proof}	
	Let $\pr_i:Y\rightarrow \Aff_R^1:=\Spec(R[x_i])$	
	for $i=1,2$ be the given projections and again denote by $\pr_i:\c{Y}_i\rightarrow \P^1_R$
	the compactification of $Y$ using $\pr_i$.
	From the very construction
	of $\c{Y}_i$, it suffices to prove that there are
	open neighborhoods $V_i$ of $\pr_i^{-1}(\infty)$
	with $V_1\simeq V_2$ canonically.
	
	We construct $V_i$ and such a canonical identification
	as follows.  By the hypothesis that $\pr_i$ is normal near infinity,
	there is a monic polynomial $f_i\in R[x_i]$ of degree $d_i$ 
	with $Y$ normal over the open $D(f_i)$ in $\Spec(R[x_i])$.
	Let $g_i=x_i^{-d_i} f_i$, viewed as a polynomial in $R[x_i^{-1}]$,
	and denote by $W$ the scheme-theorteic image of the map 
	$\pr_1\times \pr_2: Y\rightarrow \Aff_R^1\times_R \Aff_R^1 = \Spec(R[x_1,x_2])$.
	Then $W=V(F)$ for a polynomial $F(x_2,x_2)$ over $R$
	which, as each $\pr_i$ is finite, may be taken to be monic 
	of degree $e_1$ as a polynomial in $x_1$ over $R[x_2]$
	and with leading term $r x_2^{e_2}$ for a unit $r\in R[x_1]^{\times}=R^{\times}$
	when viewed as a polynomial over $R[x_1]$.
	Define $h_1:=x_1^{-e_1}F(x_1,0)$ and $h_2:=x_2^{-e_2}F(0,x_2)$
	in $R[x_i^{-1}][g_i^{-1}]$ for $i=1$, $2$, respectively,
	and set $B_i:=R[x_i^{-1}][g_i^{-1}][h_i^{-1}]$.
	It is clear from the definitions that
	$g_i$ and $h_i$ are units near $V(x_i^{-1})$, so
	$\Spec(B_i)$ is an open neighborhood
	of the infinity section in $\P^1_R$.
	Let $A:=R[x_1,x_2]/(F)$ be the coordinate ring of $W$
	and set $A':=A[\{x_i^{-1},g_i^{-1},h_i^{-1}\}_{i=1,2}]$,
	which is the coordinate ring of an open $U$ in $W$ over the infinity section of each copy of $\P^1_R$.
	Denote by $\wt{B}_i$ the normalization of $B_i$ in $A'$.
	We claim that there are suitable localizations $C_i$
	of $\wt{B}_i$ for $i=1,2$ with $C_1=C_2$ as subrings of
	the total ring of fractions of $A$.  Granting this,
	we then take $V_i$ to be the 
	normalization of $\Spec(C_i)$ in $Y\big|_U$.
	By transitivity of integral closure, it is clear
	that $V_i$ is an open neighborhood $\pr_i^{-1}(\infty)$
	in $\c{Y}_i$, and by construction we have $V_1\simeq V_2$ canonically.
	
	To see our claim, we observe that $x_2^{-1}$
	lies in $\wt{B}_1$.  Indeed, in $A'$ we have
	$$0=x_1^{-e_1}x_2^{-e_2}F(x_1,x_2) = x_2^{-e_2}h_1 + \text{terms of lower degree in }x_2^{-1},$$	
	from which it follows that $x_2^{-1}\in A'$ is integral over
	$B_1=R[x_1^{-1}][g_1^{-1}][h_1^{-1}]$ so lies in $\wt{B}_1$.
	As $g_2$ and $h_2$ are polynomials in $x_2^{-1}$ with $R$-coefficients,
	we can then make sense of the localization
	$\wt{B}_1[g_2^{-1}][h_2^{-1}]$.
	Now if $\alpha \in \wt{B}_2$, then $\alpha$
	is integral over $B_2=R[x_2^{-1}][g_2^{-1}][h_2^{-1}]$,
	so $(g_2h_2)^d\alpha$ is integral over $R[x_2^{-1}]$
	for some positive integer $d$, whence $(g_2h_2)^d\alpha$
	lies in $\wt{B}_1$ and we conclude that
	$\wt{B}_2 \subseteq \wt{B}_1[g_2^{-1}][h_2^{-1}]$.
	Similarly, we have $\wt{B}_1 \subseteq \wt{B}_2[g_1^{-1}][h_1^{-1}]$,
	and it follows that as 
	$$C_1:=\wt{B}_1[g_2^{-1}][h_2^{-1}] = \wt{B}_2[g_1^{-1}][h_1^{-1}]=:C_2,$$
	as subrings of the total ring of fractions of $A$.
	This establishes our claim, and completes the proof.
\end{proof}

\begin{corollary}\label{MorExtCor}
	Let $\scrP\rightarrow \scrQ$ be an abstract morphism of representable moduli problems
	that are finite over $(\Ell/R)$ and normal near infinity, 
	and denote by $\rho:\M(\scrP)\rightarrow \M(\scrQ)$ the induced morphism
	of moduli schemes.  Assume that $\rho$ is finite and that $\M(\scrP)$
	is normal near infinity with respect to the composition $j\circ\rho$.
	Then there is a unique extension of $\rho$ to a finite morphism
	$\c{\rho}:\c{\M}(\scrP)\rightarrow \c{\M}(\scrQ)$ of the compactified moduli schemes.  
\end{corollary}

\begin{proof}
	Let us write $j':=j\circ \rho : \M(\scrP)\rightarrow \Aff^1_R$
	and denote by $\c{\M}(\scrP)'$ the compactification of $M(\scrP)$
	using $j'$.  It follows easily from the given constructions
	that $\rho$ uniquely extends to a morphism 
	$\c{\M}(\scrP)'\rightarrow \c{\M}(\scrQ)$
	that is readily checked to be finite.  Applying Lemma \ref{UlmerLemma}
	to $Y=\M(\scrP)$ with the two projections $j,j'$ to $\Aff^1_R$
	then gives $\c{\M}(\scrP)\simeq \c{\M}(\scrP)'$ canonically,
	and the conclusion follows.
\end{proof}

Recall that we have fixed a compatible sequence $\{\varepsilon^{(r)}\}_{r\ge1}$
of primitive $p^r$-th roots of unity in $\o{\Q}_p$, and that we set $R_r:=\Z_p[\mu_{p^r}]$
and $K_r:=\Frac(R_r)$.

\begin{definition}\label{XrDef}
	We define $\X_r:=\c{\M}(\scrP_r^{\varepsilon^{(r)}})$, viewed as a scheme over $T_r:=\Spec(R_r)$.
\end{definition}

\begin{remark}\label{genfiberrem}
We may identify the generic fiber of $\X_r$ with the (base change of the) modular curve $(X_r)_{K_r}$.
Indeed, it is a standard exercise that $X_r$
is the compactified moduli scheme of triples $(E,\iota, \alpha)$ consisting 
of an elliptic curve $E$ with embeddings of group schemes
$\iota:\mu_{p^r}\hookrightarrow E$ and
$\alpha:\mu_{N}\hookrightarrow E$.
Given such a triple $(E,\iota, \alpha)$ over a $K_r$-scheme $S$, 
we set $E':=E/\im(\iota)$, $P_{\iota}:=\iota(\varepsilon^{(r)})$ and denote by 
$Q_{\iota}\in E'(S)$ the unique point satisfying $\langle \iota(z),Q_{\iota}\rangle = z$
for all $z\in \mu_{p^r}(S)$.  The association 
\begin{equation*}
	(E, \iota, \alpha)\mapsto 
	(\phi: E\rightarrow E', P_{\iota},Q_{\iota}; \alpha)
\end{equation*}
then induces the claimed identification 
$(\X_r)_{K_r}\simeq (X_r)_{K_r}$.
\end{remark}

There is a canonical action of $\Z_p^{\times}\times (\Z/N\Z)^{\times}$
by $R_r$-automorphisms of $\X_r$,
defined at the level of the underlying moduli problem by
\begin{equation}
		{(u,v)\cdot (\phi:E\rightarrow  E',P,Q; \alpha)} :={(\phi:E\rightarrow E',uP, u^{-1}Q; \alpha v)}
	\label{balcanaction}
\end{equation}
as one checks by means of the computation
	$\langle uP,u^{-1}Q\rangle_{\phi} = \langle P,Q\rangle^{uu^{-1}}_{\phi} = \langle P,Q \rangle_{\phi}$.
Here, we again write $v:\mu_N\rightarrow\mu_N$ for the automorphism of $\mu_N$ 
functorially defined by $\zeta \mapsto \zeta^v$ for any $N$-th root of unity $\zeta$. 
We refer to this action of $\Z_p^{\times}\times (\Z/N\Z)^{\times}$ as the {\em diamond operator}
action, and will denote by $\langle u \rangle$ (respectively $\langle v \rangle_N$) the 
automorphism induced by $u\in \Z_p^{\times}$ (respectively $v\in (\Z/N\Z)^{\times}$).

There is also an $R_r$-semilinear ``geometric inertia" action of $\Gamma:=\Gal(K_{\infty}/K_0)$
on $\X_r$, which reflects the fact that the generic fiber of $\X_r$ descends to $\Q_p$.
To explain this action,
for $\gamma \in \Gamma$ and any $T_r$-scheme $T'$, let us write 
$T'_{\gamma}$ for the base change of $T'$ along the induced morphism $\gamma: T_r\rightarrow T_r$.
For a moduli problem $\scrP$ on $(\Ell/T_r)$, we denote by $\gamma^*\scrP$
the moduli problem on $(\Ell/(T_r)_{\gamma})$ determined by the requirement
$(\gamma^*\scrP)(E_{\gamma}/S_{\gamma})=\scrP(E/S)$ for all objects $E/S$ of $(\Ell/T_r)$;
we note that if $\scrP$ is representable, then so is $\gamma^*\scrP$
and we have $\c{\M}(\gamma^*\scrP)\simeq \c{\M}(\scrP)_{\gamma}$ 
\cite[4.13]{KM} and {\em cf.} \cite[12.10.1]{KM} and \cite[\S4]{Ulmer}.
Each $\gamma\in \Gamma$ gives rise to 
a morphism of moduli problems $\gamma: \scrP_r^{\varepsilon^{(r)}}\rightarrow \gamma^*\scrP_r^{\varepsilon^{(r)}}$
via the assignment $\gamma:\scrP(E/S)\rightsquigarrow (\gamma^*\scrP)(E/S) = \scrP(E_{\gamma^{-1}}/S_{\gamma^{-1}})$
determined by
\begin{equation}
		{\gamma(\phi:E\rightarrow E',P,Q; \alpha)} := 
		{(\phi_{\gamma^{-1}}:E_{\gamma^{-1}}\rightarrow E'_{\gamma^{-1}},\chi(\gamma)P_{\gamma^{-1}},Q_{\gamma^{-1}}; \alpha_{\gamma^{-1}})}
	\label{gammamapsModuli}
\end{equation}
where the subscript of $\gamma^{-1}$ means ``base change along $\gamma^{-1}$" (see \S\ref{Notation}).
Since 
\begin{equation*}
	\langle \chi(\gamma)P_{\gamma^{-1}}, Q_{\gamma^{-1}}\rangle_{\phi_{\gamma^{-1}}} = 
	\gamma^{-1}\langle P,Q\rangle_{\phi}^{\chi(\gamma)} = \langle P,Q \rangle_{\phi},
\end{equation*}
this really is a morphism of moduli problems on $(\Ell/T_r)$,
so by the discussion preceding Lemma \ref{UlmerLemma} and
\cite[Proposition 9.3.1]{KM} induces
a morphism of $T_r$-schemes
\begin{equation}
	\xymatrix{
		{\gamma:\X_r} \ar[r] & {(\X_r)_{\gamma}}
		}\label{gammamaps}
\end{equation}
for each $\gamma\in \Gamma$, compatibly with change in $\gamma$.

Recall (\cite[\S6.7]{KM}) that over any base scheme $S$, a cyclic $p^{r+1}$-isogeny of elliptic curves 
$\phi:E\rightarrow E'$  admits a ``standard factorization" (in the sense of \cite[6.7.7]{KM})
\begin{equation}
	\xymatrix{
		E=:E_0 \ar[r]^-{\phi_{0,1}} & E_1 \ar[r] \cdots & E_{r} \ar[r]^-{\phi_{r,r+1}} & E_{r+1}:=E'
		}.\label{standardFac}
\end{equation}
For each pair of nonnegative integers $a<b\le r+1$ we will write
$\phi_{a,b}$ for the composite $\phi_{a,a+1}\circ\cdots\circ\phi_{b-1,b}$
and $\phi_{b,a}:=\phi_{a,b}^t$ for the dual isogeny.

Using this notion
and appealing to Corollary \ref{MorExtCor}, we may define
``degeneracy maps" $\pr,\ps:\X_{r+1} \rightrightarrows \X_r$ (over the map
$T_{r+1}\rightarrow T_r$) at the level of underlying moduli problems
as follows ({\em cf.}: \cite[11.3.3]{KM}):
\begin{equation}
\begin{aligned}
		&{\pr(\phi:E_0\rightarrow E_{r+1},P,Q; \alpha)}:=
		{(\phi_{0,r}:E_0\rightarrow E_{r},pP,\phi_{r+1,r}(Q); \alpha)}\\
		&{\ps(\phi: E_0\rightarrow E_{r+1},P,Q; \alpha)}:=
		{(\phi_{1,r+1}:E_1\rightarrow E_{r+1},\phi_{0,1}(P),pQ; \phi_{0,1}\circ \alpha)}
\end{aligned}\label{XrDegen}
\end{equation} 
By the universal property of fiber products, we obtain morphisms of
$T_{r+1}$-schemes
\begin{equation}
	\xymatrix{
		{\X_{r+1}} \ar@<0.5ex>[r]^-{\pr}\ar@<-0.5ex>[r]_-{\ps} &  {\X_r\times_{T_r} T_{r+1}}
	}.\label{rdegen}
\end{equation}
that are compatible with the diamond operators and
the geometric inertia action of $\Gamma$.

\begin{remark}
Via the identification of Remark \ref{genfiberrem}, the maps 
on generic fibers induced by (\ref{rdegen}) coincide
with the (base change to $K_{r+1}$ of the) degeneracy maps $\pr,\ps:X_{r+1}\rightrightarrows X_r$
of \S\ref{Notation}, which (on complex analytifications) arise from
the mappings $\rho:\tau\mapsto\tau$ and $\sigma:\tau\mapsto p\tau$ on the complex upper half-plane.
\end{remark}

Recall that we have fixed a choice of primitive $N$-th root of unity $\zeta_N$ in $\o{\Q}_p$.
The Atkin Lehner ``involution" $w_{\zeta_N}$ 
is defined as follows.
Let $E$ be an elliptic curve over an $R_r':=R_r[\mu_N]$-scheme $S$ and 
$\alpha: \mu_N\hookrightarrow E$ an embedding of group schemes.
Writing $\psi: E\rightarrow E/\im(\alpha)$ and
$\psi': E\rightarrow E/\im(\phi\circ \alpha)$
for the canonical maps, as $N$ is a unit in $R_r'$
there is a unique embedding $\beta:\mu_N\rightarrow E/\im(\alpha)$ 
of group schemes over $S$ satisfying $\langle \alpha(z), \beta(\zeta_N)\rangle_{\psi}=z$
for all $z\in \mu_N(S)$.  Define
\begin{equation*}
	{w_{\zeta_N} (\phi:E\rightarrow E',P,Q; \alpha)} :=
		{(\o{\phi}:E/\im(\alpha) \rightarrow E'/\im(\phi\circ\alpha),\psi(P),N^{-1}\psi'(Q); \beta)},
		\label{NAtkinLehnerInv}
\end{equation*}
where $\o{\phi}$ is the isogeny induced by $\phi$, and by a slight abuse we write ``$N^{-1}$''
for the multiplicative inverse of $N$ in $\Z_p$.  Using the basic compatibility
properties of the Weil pairing we calculate
\begin{align*}
	\langle \psi(P),N^{-1}\psi'(Q) \rangle_{\o{\phi}} &= 
	\langle P,N^{-1}\psi'(Q) \rangle_{\psi\o{\phi}}=
	\langle P,N^{-1}\psi'(Q) \rangle_{\psi'\phi}=
	\langle N^{-1}\psi'(Q), P \rangle^{-1}_{\Dual{\phi}\Dual{{\psi'}}}\\
		&=\langle N^{-1}\Dual{{\psi'}}\psi'(Q), P \rangle^{-1}_{\Dual{\phi}}=
				\langle Q, P \rangle^{-1}_{\Dual{\phi}}=
				\langle P,Q\rangle_{\phi}=\varepsilon^{(r)}
\end{align*}
so that this really does define an endomorphism of $\scrP_r^{\varepsilon^{(r)}}$.
Invoking Corollary \ref{MorExtCor} then gives a self-map of
$\X_r\times_{R_r} R_r'$ over $R_r'$ that we again denote by $w_{\zeta_N}$.
Exactly as in (\ref{gammamapsModuli}), each $g\in \Gal(K_0'/K_0)$
induces a morphism $\X_r\times_{R_r} R_r'\rightarrow g^*(\X_r\times_{R_r} R_r')$.

Following \cite[11.3.2]{KM} and once again employing Corollary \ref{MorExtCor},
we define the Atkin Lehner automorphism $w_{\varepsilon^{(r)}}$ of $\X_r$ over $R_r$ at the level of
underlying moduli problem by
\begin{equation*}
	{w_{\varepsilon^{(r)}} (\phi:E\rightarrow E',P,Q; \alpha)} :=
		{(\phi^t:E'\rightarrow E,-Q,P;\ \phi\circ\alpha )}
		\label{AtkinLehnerInv}
\end{equation*}
We then define $w_r:=w_{\varepsilon^{(r)}} \circ w_{\zeta_N}$; 
it is an automorphism of $\X_r \times_{R_r} R_r'$
over $R_r'$.

\begin{proposition}\label{ALinv}
	For all $(u,v)\in \Z_p^{\times}\times(\Z/N\Z)^{\times}$ and all 
	$(\gamma,g)\in \Gamma\times \Gal(K_0'/K_r)\simeq\Gal(K_{\infty}'/K_0)$, the identities
	\begin{align*}
		w_{\zeta_N}^2 &= \langle -1\rangle_N\langle N\rangle \\
		w_{\varepsilon^{(r)}}^2 &= \langle p^r\rangle_N\langle -1\rangle \\
		w_{\zeta_N}\circ w_{\varepsilon^{(r)}} &= \langle p^r\rangle^{-1}_N \langle N\rangle w_{\varepsilon^{(r)}}\circ
		w_{\zeta_N}\\
		w_r^2 &= \langle -1\rangle_N\langle -1\rangle \\
		w_{\varepsilon^{(r)}} \langle u\rangle \langle v\rangle_N &= 
		\langle u^{-1}\rangle\langle v\rangle_N w_{\varepsilon^{(r)}} \\
		w_{\zeta_N} \langle u\rangle \langle v\rangle_N &= 
		\langle u\rangle\langle v^{-1}\rangle_N w_{\zeta_N} \\
				w_{r} \langle u\rangle \langle v\rangle_N &= 
		\langle u^{-1}\rangle\langle v^{-1}\rangle_N w_{r} \\
		\pr w_{\varepsilon^{(r+1)}} &= w_{\varepsilon^{(r)}}\ps\\
		\ps w_{\varepsilon^{(r+1)}} &= \langle p\rangle_N w_{\varepsilon^{(r)}}\pr\\
		\pr w_{\zeta_N} &= w_{\zeta_N}\pr\\
		\ps w_{\zeta_N} &= \langle p\rangle_N w_{\zeta_N}\ps\\
		\pr w_r &= \langle p\rangle_N w_r \ps\\
		\ps w_r &= \langle p\rangle_N w_r \pr\\
		(\gamma^*w_{\varepsilon^{(r)}})\gamma & = \langle \chi(\gamma)^{-1}\rangle\gamma w_{\varepsilon^{(r)}}\\
		(g^*w_{\zeta_N})g & = \langle a(g)^{-1}\rangle_Ng w_{\zeta_N}\\
		((\gamma,g)^* w_r)(\gamma,g) &= \langle \chi(\gamma)^{-1} \rangle \langle a(g)^{-1}\rangle_N (\gamma,g) w_r
	\end{align*}
	hold, with $a:\Gal(K_{\infty}'/K_0)\rightarrow (\Z/N\Z)^{\times}$ the character determined
	by $\gamma\zeta=\zeta^{a(\gamma)}$ for all $\zeta\in \mu_N(\Qbar_p)$.
\end{proposition}

\begin{proof}
	This is a straightforward consequence of definitions.
\end{proof}

In order to describe the special fiber of $\X_r$, we must first introduce Igusa curves:

\begin{definition}\label{IgusaStrDef}
	Let $r$ be a nonnegative integer.  The moduli problem $\I_r:=([\Ig(p^r)]; [\mu_N])$ on 
	$(\Ell/\F_p)$ assigns to $(E/S)$ the set of triples $(E,P;\alpha)$ where $E/S$ is an elliptic curve
	and
	\begin{enumerate}
		\item $P\in E^{(p^r)}(S)$ is a point that generates the $r$-fold iterate of 
		Verscheibung $V^{(r)}:E^{(p^r)}\rightarrow E$.   
		\item $\alpha:\mu_N\hookrightarrow E[N]$ is a closed immersion of $S$-group schemes.
	\end{enumerate}
\end{definition}

\begin{proposition}\label{Pr:IgusaSmoothCrv}
	If $N\ge 4$, then the moduli problem $\I_r$ on $(\Ell/\F_p)$ is represented by a smooth 
	affine curve $\M(\I_r)$ over $\F_p$.
\end{proposition}

\begin{proof}
	One argues as in the proof of Proposition \ref{XrRepresentability}, using  \cite[12.6.1]{KM}
	to know that $[\Ig(p^r)]$ is relatively representable on $(\Ell/\F_p)$,
	regular 1-dimensional and finite flat over $(\Ell/\F_p)$.
\end{proof}

\begin{definition}\label{IgusaDef}
	We write $\Ig_r:=\c{\M}(\I_r)$ for the compactified moduli scheme.
\end{definition}

As $\Ig_r$ is normal and 1-dimensional by the very construction of $\c{\M}(\cdot)$,
it follows that $\Ig_r$ is a smooth and proper $\F_p$-curve.
There is a canonical action of the diamond operators $\Z_p^{\times}\times (\Z/N\Z)^{\times}$ on
$\I_r$ via $(u,v)\cdot (E,P; \alpha):= (E,u^{-1}P; \alpha  v)$; this induces a corresponding action on $\Ig_r$
by $\F_p$-automorphisms.  We again write $\langle u\rangle$ (respectively $\langle v\rangle_N$) for the
action of $u\in \Z_p^{\times}$ (respectively $v\in (\Z/N\Z)^{\times}$).

\begin{remark}\label{GrossCompat}
	The astute reader will wonder why we have defined the action of $\langle u\rangle$ 
	on the moduli problem $\I_r$ to be $(E,P,\alpha)\rightsquigarrow (E,u^{-1}P,\alpha)$
	rather than $(E,P,\alpha)\rightsquigarrow (E,uP,\alpha)$.  The reason is this: 
	in the literature ({\em e.g} \cite[\S5]{tameness}, \cite{MW-Hida}, \cite[\S2.1]{Saby}), 
	one works instead with the
	moduli problem $\I_{r}^{\mu}$ classifying $(E,\iota,\alpha)$, where $\alpha$
	is as above and $\iota:\mu_{p^r}\hookrightarrow E[p^r]$ is a closed immersion of $S$-group
	schemes.  On the category of ordinary $E/S$, there is an isomorphism of moduli problems
	$\I_r\rightarrow \I_r^{\mu}$ given as follows: a triple $(E,P,\alpha)$
	determines an isomorphism $\psi_P: \Z/p^r\Z\rightarrow \ker(V^r:E^{(p^r)}\rightarrow E)$,
	the Cartier dual of which is an isomorphism $\psi_P^t:\ker(F^r:E\rightarrow E^{(p^r)})\rightarrow \mu_{p^r}$.
	The inverse of this isomorohism gives an embedding $\iota:=(\psi_P^t)^{-1}\hookrightarrow E[p^r]$,
	which gives an element $(E,\iota,\alpha)$ of $\I_r^{\mu}(E/S)$, and this association induces
	the claimed isomorphism of moduli problems.  Under this identification, the action of $\langle u\rangle$
	on $\I_r$ that we have defined is carried to the action of $\langle u\rangle$ on $\I_r^{\mu}$
	given by $\langle u\rangle: (E,\iota,\alpha)\mapsto (E,u\circ\iota,\alpha)$, which agrees
	with the diamond operator action in {\em loc.~cit.}
\end{remark}

Thanks to the ``backing up theorem" \cite[6.7.11]{KM}, one also has natural degeneracy maps
\begin{equation}
	\xymatrix{
		{\pr:\Ig_{r+1}} \ar[r] & {\Ig_r}
	}\qquad\text{induced by}\qquad
	\pr(E,P;\alpha):= (E,VP,\alpha)
	\label{Vmapsch}
\end{equation}
on underlying moduli problems.  These maps are visibly
equivariant for the diamond operator action on source and target.
Let $\SS_r$ be the (reduced) closed subscheme of $\Ig_r$
that is the support of the coherent ideal sheaf of relative differentials $\Omega^1_{\Ig_r/\Ig_0}$;
over the unique degree 2 extension of $\F_p$, this scheme breaks up as a disjoint union of rational 
points---the supersingular points. The map (\ref{Vmapsch})
is finite of degree $p$, generically \'etale and totally (wildly) ramified over
each supersingular point. 

We can now describe the special fiber of $\X_r$:

\begin{proposition}	\label{redXr}
	The scheme $\o{\X}_r:=\X_r\times_{T_r} \Spec(\F_p)$ is the disjoint union, with crossings at the supersingular points, of the 		following proper, smooth $\F_p$-curves:
	for each pair $a,b$ of nonnegative integers with $a+b=r$, and for each $u\in (\Z/p^{\min(a,b)}\Z)^{\times}$, 
	one copy of 
	$\Ig_{\max(a,b)}$.
\end{proposition} 

We refer to \cite[13.1.5]{KM} for the definition of ``disjoint union with crossings at the supersingular points".
Note that the special fiber of $\X_r$ is (geometrically) reduced; this will be crucial in our later work.  We often write $I_{(a,b,u)}$ for the irreducible component of $\o{\X}_r$ indexed by the triple $(a,b,u)$ and will refer to it as the {\em $(a,b,u)$-component} (for fixed $(a,b)$ we have $I_{(a,b,u)}=\Ig_{\max(a,b)}$ for all $u$).

For the proof of Proposition \ref{redXr}, we refer the reader to \cite[13.11.2--13.11.4]{KM},
and content ourselves with recalling
the correspondence between (non-cuspidal) points of the $(a,b,u)$-component and 
$[\bal \Upgamma_1(p^r)]^{1\can}$-structures on elliptic curves.

Let $S$ be any $\F_p$ scheme, fix an ordinary elliptic curve $E_0$ over $S$, and
let $(\phi:E_0\rightarrow E_r,P,Q; \alpha)$ be an element of $\scrP_{r}^1(E_0/S)$.
By \cite[13.11.2]{KM}, there exist unique nonnegative integers $a,b$
with the property that the cyclic $p^r$-isogeny $\phi$ factors as a purely inseparable cyclic $p^a$-isogeny
followed by an \'etale $p^b$-isogeny (this {\em is} the standard factorization of $\phi$).  
Furthermore, there exists a unique elliptic curve $E$ over $S$ and $S$-isomorphisms
$E_0\simeq E^{(p^b)}$ and $E_r\simeq E^{(p^a)}$ such that the cyclic $p^r$ isogeny $\phi$ is: 
\begin{equation*}
	\xymatrix{
		{E_0\simeq E^{(p^b)}} \ar[r]^-{F^a} & {E^{(p^r)}} \ar[r]^-{V^b} & {E^{(p^a)} \simeq E_r}
		}
\end{equation*}
and $P\in E^{(p^b)}(S)$ (respectively $Q\in E^{(p^a)}(S)$) is an Igusa structure of level $p^b$ (respectively $p^a$) on $E$ over $S$.
When $a\ge b$ there is a unique unit $u\in (\Z/p^{b}\Z)^{\times}$ such that $V^{a-b}(Q)=uP$ in $E^{(p^b)}(S)$ and when
$b\ge a$ there is a unique unit $u\in (\Z/p^a\Z)^{\times}$ such that $uV^{b-a}(P)=Q$ in $E^{(p^a)}(S)$.  
Thus, for $a\ge b$ (respectively $b\ge a$) and fixed $u$, the data $(E,Q;V^b \alpha p^{-b})$ 
(respectively $(E,P;V^b \alpha  p^{-b})$) 
gives an $S$-point of the $(a,b,u)$-component
$\Ig_{\max(a,b)}$.  
Conversely, suppose given $(a,b,u)$ and an $S$-valued point of $\Ig_{\max(a,b)}$ which is neither a cusp nor a supersingular point (in the sense that it corresponds to an ordinary elliptic curve with extra structure).
If $a\ge b$ and $(E,Q;\alpha)$ is the given $S$-point of $\Ig_{a}$ then we set
$P:=u^{-1}V^{a-b}(Q)$, while if $b\ge a$ and $(E,P;\alpha)$ is the given $S$-point of $\Ig_b$ then we set $Q:=uV^{b-a}P$.  Due to \cite[13.11.3]{KM}, the data
\begin{equation*}
	(\xymatrix{
		{E^{(p^b)}} \ar[r]^-{F^a} & {E^{(p^r)}} \ar[r]^-{V^b} & {E^{(p^a)}, P,Q; F^b \alpha} 
	})
\end{equation*}
gives an $S$-point of $\M(\scrP_r^{1})$.  These constructions are visibly inverse to each other.

\begin{remark}\label{rEvenChoice}
	When $r$ is even and $a=b=r/2$, there is a choice to be made as to how one identifies the 
	$(r/2,r/2,u)$-component of $\o{\X}_r$ with $\Ig_{r/2}$: if
	$(\phi:E_0\rightarrow E_r,P,Q;\alpha)$ is an element of  $\scrP_r^1(E_0/S)$
	which corresponds to a point on the $(r/2,r/2,u)$-component, then for $E$ with 
	$E_0\simeq E^{(p^{r/2})}\simeq E_r$,
	{\em both} $(E,P;V^{r/2}\alpha  p^{-r/2})$ and $(E,Q;V^{r/2}\alpha  p^{-r/2})$ are 
	$S$-points of $\Ig_{r/2}$. 
	Since $uP=Q$,
	the corresponding closed immersions
	$\Ig_{r/2}\hookrightarrow \o{\X}_r$ 	
	are twists of each other by the automorphism $\langle u\rangle$ of the source.
	We will {\em consistently choose} $(E,Q;V^{r/2}\alpha p^{-r/2})$ to identify the $(r/2,r/2,u)$-component
	of $\o{\X}_r$ with $\Ig_{r/2}$.  
\end{remark} 
  
\begin{remark}\label{MWGood}
	 As in \cite[pg.~236]{MW-Hida}, we will refer to $I_r^{\infty}:=I_{(r,0,1)}$ and $I_r^0:=I_{(0,r,1)}$
  	as the two ``good" components of $\o{\X}_r$.  
	The $\Q_p$-rational cusp $\infty$ of $X_r$ induces a section
	of $\X_r\rightarrow T_r$ which meets $I_r^{\infty}$, while
	the section
	induced by the $K_r'$-rational cusp $0$ meets $I_r^0$.
	It is precisely these irreducible components of 
	$\o{\X}_r$ which contribute to the ``ordinary" part of cohomology. 
	We note that $I_r^{\infty}$ corresponds to the image of $\Ig_r$ under the map $i_1$
	of \cite[pg. 236]{MW-Hida}, and corresponds to the component of $\o{\X}_r$ denoted by $C_{\infty}$
	in \cite[pg. 343]{Tilouine}, by $C_r^{\infty}$ in \cite[pg. 231]{Saby}
	and, for $r=1$, by $I$ in \cite[\S 7]{tameness}.  
\end{remark}

By base change, the degeneracy mappings (\ref{rdegen}) induce morphisms 
$\o{\pr},\o{\ps}:\o{\X}_{r+1}\rightrightarrows \o{\X}_r$
of curves over $\F_p$ which admit the following descriptions on irreducible components:

\begin{proposition}\label{pr1desc}
	Let $a,b$ be nonnegative integers with $a+b=r+1$ and $u\in (\Z/p^{\min(a,b)}\Z)^{\times}$.
	The restriction of the map $\o{\ps}: \o{\X}_{r+1}\rightarrow \o{\X}_r$ to the $(a,b,u)$-component
	of $\o{\X}_{r+1}$ is:
	\begin{equation*}
		\begin{cases}
			\xymatrix{
				{\Ig_{a}=I_{(a,b,u)}} \ar[r]^-{ F\circ \pr} & {I_{(a-1,b,u)}=\Ig_{a-1}}
				}
				&:\quad b < a \le r+1 \\
				\xymatrix{
				{\Ig_{b}=I_{(a,b,u)}} \ar[r]^-{\langle u\rangle^{-1}F} & 
				{I_{(a-1,b,u\bmod p^{a-1})}=\Ig_{b}}
				}
				&:\quad a = b = r/2 \\
				\xymatrix{
				{\Ig_{b}=I_{(a,b,u)}} \ar[r]^-{F} & {I_{(a-1,b,u\bmod p^{a-1})}=\Ig_{b}}
				}
				&:\quad a < b < r+1 \\
				\xymatrix{
				{\Ig_{r+1}=I_{(0,r+1,1)}} \ar[r]^-{\langle p\rangle_N\pr} & {I_{(0,r,1)}=\Ig_{r}}
				}
				&:\quad (a,b,u)=(0,r+1,1)
		\end{cases}
	\end{equation*}
	and the restriction of the map $\overline{\pr}: \o{\X}_{r+1}\rightarrow \o{\X}_r$ to the $(a,b,u)$-component
	of $\o{\X}_{r+1}$ is: 
		\begin{equation*}
		\begin{cases}
			\xymatrix{
				{\Ig_{r+1}=I_{(r+1,0,1)}} \ar[r]^-{\pr} & {I_{(r,0,1)}=\Ig_{r}}
				}
				&:\quad (a,b,u)=(r+1,0,1)\\
				\xymatrix{
				{\Ig_{a}=I_{(a,b,u)}} \ar[r]^-{F} & {I_{(a,b-1,u\bmod p^{b-1})}=\Ig_{a}}
				}
				&:\quad b < a+1 \le r+1 \\
					\xymatrix{
				{\Ig_{b}=I_{(a,b,u)}} \ar[r]^-{\langle u\rangle F\circ \pr} & {I_{(a,b-1,u)}=\Ig_{b-1}}
				}
				&:\quad a+1 = b = r/2+1\\
				\xymatrix{
				{\Ig_{b}=I_{(a,b,u)}} \ar[r]^-{F\circ \pr} & {I_{(a,b-1,u)}=\Ig_{b-1}}
				}
				&:\quad a+1 < b \le r+1  \\
		\end{cases}
	\end{equation*}
	Here, for any $\F_p$-scheme $I$, the map $F:I\rightarrow I$ is the absolute Frobenius morphism.
\end{proposition}

\begin{proof}
	Using the moduli-theoretic definitions (\ref{XrDegen}) of the degeneracy maps,
	the proof is a pleasant exercise in tracing through 
	the functorial correspondence between the points of $\o{\X}_r$
	and points of $\Ig_{(a,b,u)}$.  We leave the details to the reader.
\end{proof}

We likewise have a description of the automorphism 
of $\o{\X}_r$ induced via base change by
the geometric inertia action\footnote{
Since $\Gamma$ acts trivially on $\F_p$,
for each $\gamma\in \Gamma$ the base change of the $R_r$-morphism $\gamma: \X_r\rightarrow (\X_r)_{\gamma}$
along the map induced by the canonical surjection $R_r\twoheadrightarrow \F_p$
is an $\F_p$-morphism $\o{\gamma}:\o{\X}_r\rightarrow (\o{\X}_r)_{\gamma}\simeq \o{\X}_r$.
} 
(\ref{gammamapsModuli}) of $\Gamma$:

\begin{proposition}\label{AtkinInertiaCharp}
	Let $a,b$ be nonnegative integers with $a+b=r$ and $u\in (\Z/p^{\min(a,b)}\Z)^{\times}$.
		For $\gamma\in \Gamma$, the restriction of $\o{\gamma}:\o{\X}_{r}\rightarrow \o{\X}_r$ to 
		the $(a,b,u)$-component of $\o{\X}_{r}$ is:
		\begin{equation*}
			\begin{cases}
				\xymatrix{
				{\Ig_{a}=I_{(a,b,u)}} \ar[r]^-{\id} & 
				{I_{(a,b,\chi(\gamma)u)}=\Ig_{a}}
				}
				&:\quad b\le a \le r\\
				\xymatrix{
				{\Ig_{b}=I_{(a,b,u)}} \ar[r]^-{\langle \chi(\gamma)\rangle^{-1}} & {I_{(a,b,\chi(\gamma)u)}=\Ig_{b}}
				}
				&:\quad a < b \le r \\
		\end{cases}
		\end{equation*}
		\label{InertiaCharp}
\end{proposition}

Following \cite[\S7--8]{Ulmer}, we now define a correspondence $\pi_1,\pi_2:\Y_r\rightrightarrows\X_r$ on $\X_r$ over $R_r$ which naturally 
extends the correspondence on $X_r$ giving the Hecke operator $U_p$ (see Appendix \ref{GD}
for a brief discussion of correspondences).

\begin{definition}
	Let $r$ be a nonnegative integer and $R$ a ring containing a fixed choice $\zeta$ of 
	primitive $p^r$-th root of unity
	in which $N$ is invertible.
	The moduli problem 
	$\scrQ_{r}^{\zeta}:=([\Upgamma_0(p^{r+1}); r,r]^{\zeta\can}; [\mu_N])$ on $(\Ell/R)$ 
	assigns to $E/S$ the set of quadruples $(\phi:E\rightarrow E',P,Q;\alpha)$
	where:
	\begin{enumerate}
		\item $\phi$ is a cyclic $p^{r+1}$-isogeny
		with standard factorization 
		\begin{equation*}
		\xymatrix{
			E=:E_0 \ar[r]^-{\phi_{0,1}} & E_1 \ar[r] \cdots & E_{r} \ar[r]^-{\phi_{r,r+1}} & E_{r+1}:=E'
		}
		\end{equation*}
		
		\item $P\in E_1(S)$ and $Q\in E_r(S)$ are generators of 
		$\ker\phi_{1,r+1}$ and $\ker \phi_{r,0}$, respectively, and satisfy
		\begin{equation*}
			\langle P, \phi_{r,r+1}(Q) \rangle_{\phi_{1,r+1}}=\langle \phi_{1,0}(P),Q \rangle_{\phi_{0,r}}=\zeta.
		\end{equation*}
		
		\item  $\alpha:\mu_N\hookrightarrow E[N]$ is a closed immersion of $S$-group schemes.
	\end{enumerate}	
\end{definition}

\begin{proposition}\label{YrRepresentability}
	If $N \ge 4$, then the moduli problem $\scrQ_{r}^{\zeta}$ is represented
	by a regular scheme $\M(\scrQ_r^{\zeta})$ that is flat of pure relative dimension
	$1$ over $\Spec(R)$.  The compactified moduli scheme
	$\c{\M}(\scrQ_r^{\zeta})$ is regular and proper flat of pure
	relative dimension $1$ over $\Spec(R)$.
\end{proposition}

\begin{proof}
	Keeping in mind \cite[\S7.9]{KM}, the argument
	is identical to that of Propositions 
	\ref{XrRepresentability}--\ref{XrCptRepresentability}.
\end{proof}

\begin{definition}\label{YrDef}
	We set $\Y_r:=\c{\M}(\scrQ_r^{\varepsilon^{(r)}})$, viewed as a scheme over $T_r=\Spec(R_r)$.
\end{definition}

\begin{remark}\label{Yrgeniden}
	Write $Y_r:=X_1(Np^r; Np^{r-1})$ for the canonical model over $\Q$
	with rational cusp $i\infty$ of the modular curve 	
arising as the upper half-plane quotient
	$\Upgamma_{r+1}^r\backslash \h^*$, for
	$\Upgamma_{r+1}^r:=\Upgamma_1(Np^r)\cap \Upgamma_0(p^{r+1})$. 
	As in Remark \ref{genfiberrem}, we may identify the generic fiber of $\Y_r$ with the base change $(Y_r)_{K_r}$
	as follows.
	First, it is easy to check that $Y_r$ is the compactified moduli scheme
	of quadruples $(E_1,\iota, \alpha,C)$ where $E_1$ is an elliptic curve,
	$\iota:\mu_{p^r}\hookrightarrow E_1$ and $\alpha:\mu_{N}\hookrightarrow E_1$ are embeddings of group schemes,
	and $C$ is a locally free subgroup scheme of rank $p$ in $E_1[p]$ with the property that
	$C\cap\im\iota = 0$.  Now given such a quadruple $(E_1,\iota, \alpha,C)$ over a $K_r$-scheme $S$,
	we set $E:=E_1/C$, $E':=E_1/\im\iota$, and we denote by $\phi:E\rightarrow E'$ the evident
	isogeny.  We then put
	$P_{\iota}:=\iota(\varepsilon^{(r)})$
	and denote by $Q_{\iota}\in E_r(S)$ the unique point satisfying 
	$\langle \iota(z),\phi_{r,r+1}(Q_{\iota})\rangle_{\phi_{1,r+1}}=z$ for all $z\in \mu_{p^r}(S)$.
	The association
	\begin{equation*}
		(E_1,\iota, \alpha,C)\mapsto (\phi:E\rightarrow E', P_{\iota},Q_{\iota}; \phi_{1,0}\circ\alpha)
	\end{equation*}
	then induces the claimed identification.
\end{remark}

The scheme $\Y_r$ is equipped with an action of the diamond operators $\Z_p^{\times}\times (\Z/N\Z)^{\times}$,
as well as a ``geometric inertia" action of $\Gamma$ given moduli-theoretically
exactly as in (\ref{balcanaction}) and (\ref{gammamapsModuli}).


There is a canonical morphism of curves $\pi:\X_{r+1}\rightarrow \Y_r$ over $T_{r+1}\rightarrow T_r$
induced by the morphism 
\begin{equation}
\begin{gathered}
			\scrP_{r+1}^{\varepsilon^{(r)}}\rightarrow \scrQ_r^{\varepsilon^{(r)}}\quad\text{given by}
		\quad{\pi(\phi:E\rightarrow E',P,Q;\alpha)} := {(\phi:E\rightarrow E',\phi_{0,1}(P),\phi_{r+1,r}(Q); \alpha)}.
\end{gathered}\label{XtoY}	
\end{equation}

It is straightforward to check that the two projection maps $\ps,\pr: \X_{r+1}\rightrightarrows \X_r$ 
of (\ref{XrDegen}) factor through $\pi$ via unique maps of $T_r$-schemes 
$\pi_1,\pi_2: \Y_{r}\rightrightarrows \X_r$, given (via Corollary \ref{MorExtCor}) as morphisms of underlying
moduli problems on $(\Ell/R_r)$ by
\begin{equation}
\begin{aligned}
		&{\pi_1(\phi:E_0\rightarrow E_{r+1},P,Q;\alpha)}:=
		{(E_1\xrightarrow{\phi_{1,r+1}} E_{r+1},P,\phi_{r,r+1}(Q);\ \phi_{0,1}\circ\alpha)}\\
		&{\pi_2(\phi:E_0\rightarrow E_{r+1},P,Q;\alpha)}:=
		{(E_0\xrightarrow{\phi_{0,r}} E_{r},\phi_{1,0}(P),Q; \alpha)}
\end{aligned}\label{Upcorr}
\end{equation}
That these morphisms are well-defined and
that one has $\pr=\pi\circ \pi_2$ and $\ps=\pi\circ \pi_1$
is easily verified (see \cite[\S7]{Ulmer} and compare to
\cite[\S11.3.3]{KM}).  They are moreover 
finite of generic degree $p$, equivariant for the diamond operators,
and $\Gamma$-compatible. 

\begin{remark}
	Via the identifications of Remarks \ref{genfiberrem} and \ref{Yrgeniden},  
	on generic fibers the maps $\pi_1$ and $\pi_2$ 
	are induced by the usual degeneracy maps $\pi_1,\pi_2:Y_r\rightrightarrows X_r$
	of modular curves over $\Q$ giving the Hecke correspondence $U_p:=(\pi_1,\pi_2)$.
	Interpreting $Y_r$ and $X_r$ moduli-theoretically as in Remarks \ref{genfiberrem} and \ref{Yrgeniden},  
	the map $\pi_1$ correspons to ``forget $C$," while $\pi_2$ corresponds to ``quotient by $C$." 
We stress that the ``standard" degeneracy map $\rho:X_{r+1}\rightarrow X_r$ factors 
through $\pi_2$ (and not $\pi_1$).
\end{remark}

\begin{proposition}\label{redYr}
	The scheme $\o{\Y}_r:=\Y_r\times_{T_r} \Spec(\F_p)$ is the disjoint union, with crossings 
	at the supersingular points, of the following proper, smooth $\F_p$-curves:
	for each pair of nonnegative integers $a,b$ with $a+b=r+1$ and for each $u\in (\Z/p^{\min(a,b)}\Z)^{\times}$,
	one copy of
	\begin{equation*}
		\begin{cases}
			\Ig_{\max(a,b)} &\text{if}\ ab\neq 0\\
			\Ig_{r} &\text{if}\ (a,b)=(r+1,0)\ \text{or}\ (0,r+1)
		\end{cases}
	\end{equation*}
\end{proposition}

We will write $J_{(a,b,u)}$ for the irreducible component of $\o{\Y}_r$
indexed by $(a,b,u)$, and refer to it as the $(a,b,u)$-component; again, 
$J_{(a,b,u)}$ is independent of $u$.
The proof of Proposition \ref{redYr} is a straightforward adaptation of the arguments of 
\cite[13.11.2--13.11.4]{KM} (see also \cite[Proposition 8.2]{Ulmer}).  
We recall the correspondence between non-cuspidal points of the $(a,b,u)$-component and 
$[\Upgamma_0(p^{r+1}); r,r]^{1\can}$-structures on elliptic curves.

Fix an ordinary elliptic curve  $E_0$ over an $\F_p$-scheme $S$, 
and let $(\phi:E_0\rightarrow E_{r+1},P,Q; \alpha)$ be an element of 
$\scrQ_r^{1}(E_0/S)$.  As before, there exist unique 
nonnegative integers $a,b$ with $a+b=r+1$ 
and a unique elliptic curve $E/S$
with the property that the cyclic $p^{r+1}$-isogeny $\phi$ factors as
\begin{equation*}
	\xymatrix{
		{E_0\simeq E^{(p^b)}} \ar[r]^-{F^a} & {E^{(p^{r+1})}} \ar[r]^-{V^b} & {E^{(p^a)} \simeq E_{r+1}}
		}.
\end{equation*}
First suppose that $ab\neq 0$. Then the point $P\in E^{(p^{b+1})}(S)$ (respectively $Q\in E^{(p^{a+1})}(S)$)
is an $[\Ig(p^b)]$ (respectively $[\Ig(p^a)]$) structure on $E^{(p)}$ over $S$.  
If $a\ge b$, there is a unit $u\in (\Z/p^b\Z)^{\times}$ such that $V^{a-b}(Q)=uP$ in $E^{(p^{b+1})}(S)$,
while if $a\le b$ then there is a unique $u\in (\Z/p^a\Z)^{\times}$ with $uV^{b-a}(P)=Q$ in $E^{(p^{a+1})}(S)$.
For $a\ge b$ (respectively $a < b$), and fixed $u$, the data
 $(E^{(p)}, Q ; V^{b-1}\alpha p^{1-b})$
(respectively $(E^{(p)}, P; V^{b-1}\alpha p^{1-b})$) gives an $S$-point of the $(a,b,u)$-component
$\Ig_{\max(a,b)}$.  If $b=0$ (respectively $a=0$), then $Q\in E^{(p^r)}(S)$ (respectively $P\in E^{(p^r)}(S)$)
is an $[\Ig(p^r)]$-structure on $E=E_0$ (respectively $E = E_{r+1}$).
In these extremal cases, the data $(E,Q;\alpha)$
(respectively $(E,P;V^{r+1} \alpha  p^{-r-1})$) 
gives an $S$-point of the $(r+1,0,1)$-component (respectively $(0,r+1,1)$-component)
$Ig_r$.

Conversely, suppose given $(a,b,u)$ and an $S$-point of $\Ig_{\max(a,b)}$ which is neither
cuspidal nor supersingular.  If $r+1 > a\ge b$ and $(E,Q;\alpha)$ is the given point of $\Ig_a$,
then we set $P:=u^{-1}V^{a-b}(Q)\in E^{(p^b)}(S)$, while if $r+1 > b\ge a$ and $(E,P;\alpha)$ is the given point of
$\Ig_b$, we set $Q:=uV^{b-a}P\in E^{(p^{a})}(S)$.  Then 
\begin{equation*}
	(\xymatrix{
		{E^{(p^{b-1})}} \ar[r]^-{F} & {E^{(p^b)}} \ar[r]^-{F^{a-1}} & {E^{(p^r)}} \ar[r]^-{V^{b-1}} & 
		{E^{(p^a)}} \ar[r]^-{V} & {E^{(p^{a-1})}}, P,Q; F^{b-1} \alpha 
	})
\end{equation*}
is an $S$-point of $\M(\scrQ_r^{1})$.
If $b=0$ and $(E,Q,\alpha)$ is an $S$-point of $\Ig_r$, then we let $P\in E^{(p)}(S)$ be the identity 
section and we obtain an $S$-point $(F^{r+1}:E\rightarrow E^{(p^{r+1})},P,Q;\alpha)$
of $\M(\scrQ_r^1)$.
If $a=0$ and $(E,P,\alpha)$ is an $S$-point of $\Ig_r$, then we let $Q\in E^{(p)}(S)$
be the identity section and we obtain an $S$-point $(V^{r+1}:E^{(p^{r+1})}\rightarrow E,P,Q;F^{r+1} \alpha)$
of $\M(\scrQ_r^1)$.

Using the descriptions of $\o{\X}_r$ and $\o{\Y}_r$ furnished by
Propositions \ref{redXr} and \ref{redYr}, we can calculate the restrictions of
the degenercy maps $\o{\pi}_1,\o{\pi}_2:\o{\Y}_r\rightrightarrows \o{\X}_r$ 
to each irreducible component
of the special fiber of $\Y_r$.  The following is due to Ulmer\footnote{We warn the reader, however,
that Ulmer omits the effect of the degeneracy maps on $[\mu_N]$-structures, so his formulae
are slightly different from ours.}
\cite[Proposition 8.3]{Ulmer}: 

\begin{proposition}\label{UlmerProp}
	Let $a,b$ be nonnegative integers with $a+b=r+1$ and $u\in (\Z/p^{\min(a,b)}\Z)^{\times}$.
	The restriction of the map $\o{\pi}_1: \o{\Y}_r\rightarrow \o{\X}_r$ to the $(a,b,u)$-component
	of $\o{\Y}_r$ is:
	\begin{equation*}
		\begin{cases}
			\xymatrix{
				{\Ig_{r}=J_{(r+1,0,1)}} \ar[r]^-{F} & {I_{(r,0,1)}=\Ig_{r}}
				}
				&:\quad (a,b,u)=(r+1,0,1)\\
			\xymatrix{
				{\Ig_{a}=J_{(a,b,u)}} \ar[r]^-{\pr} & {I_{(a-1,b,u)}=\Ig_{a-1}}
				}
				&:\quad b < a < r+1 \\
			\xymatrix{
				{\Ig_{b}=J_{(a,b,u)}} \ar[r]^-{\langle u^{-1}\rangle} & {I_{(a-1,b,u\bmod p^{a-1})}=\Ig_{b}}
				}
				&:\quad a=b=(r+1)/2\\	
			\xymatrix{
				{\Ig_{b}=J_{(a,b,u)}} \ar[r]^-{\id} & {I_{(a-1,b,u\bmod p^{a-1})}=\Ig_{b}}
				}
				&:\quad a < b < r+1 \\
			\xymatrix{
				{\Ig_{r}=J_{(0,r+1,1)}} \ar[r]^-{\langle p\rangle_N} & {I_{(0,r,1)}=\Ig_{r}}
				}
				&:\quad (a,b,u)=(0,r+1,1)
		\end{cases}
	\end{equation*}
	and the restriction of the map $\overline{\pi}_2: \o{\Y}_r\rightarrow \o{\X}_r$ to the $(a,b,u)$-component
	of $\o{\Y}_r$ is: 
		\begin{equation*}
		\begin{cases}
			\xymatrix{
				{\Ig_{r}=J_{(r+1,0,1)}} \ar[r]^-{\id} & {I_{(r,0,1)}=\Ig_{r}}
				}
				&:\quad (a,b,u)=(r+1,0,1)\\
				\xymatrix{
				{\Ig_{a}=J_{(a,b,u)}} \ar[r]^-{\id} & {I_{(a,b-1,u\bmod p^{b-1})}=\Ig_{a}}
				}
				&:\quad b < a+1 \le r+1 \\
				\xymatrix{
				{\Ig_{b}=J_{(a,b,u)}} \ar[r]^-{\langle u \rangle \pr} & {I_{(a,b-1,u)}=\Ig_{b-1}}
				}
				&:\quad a +1 = b =r/2 + 1  \\
				\xymatrix{
				{\Ig_{b}=J_{(a,b,u)}} \ar[r]^-{\pr} & {I_{(a,b-1,u)}=\Ig_{b-1}}
				}
				&:\quad a+1 < b < r+1  \\
				\xymatrix{
				{\Ig_{r}=J_{(0,r+1,1)}} \ar[r]^-{F} & {I_{(0,r,1)}=\Ig_{r}}
				}
				&:\quad (a,b,u)=(0,r+1,1)
		\end{cases}
	\end{equation*}
\end{proposition}

\begin{proof}
	The proof is similar to the proof of Proposition \ref{pr1desc}, using the correspondence between
	irreducible components of $\Y_r$, $\X_r$ and Igusa curves that we have explained, together with the moduli-theoretic 
	definitions (\ref{Upcorr}) of the degeneracy mappings.  We leave the details to the reader.
\end{proof}

As we make frequent use of the action of the Hecke operators on the cohomology 
of $\X_r$ and $I_r$, we would be remiss not to record the definition of the Hecke
correspondences on these curves.

\begin{definition}\label{Def:HeckeModuliProblem}
	For a prime $\ell\neq p$, a nonnegative integer $r$, and a $\Z[1/N\ell]$-algebra
	$R$ containing a fixed choice $\zeta$ of 
	primitive $p^r$-th root of unity, 
	the moduli problem 
	$\prescript{}{\ell}{\scrP_{r}^\zeta}:=([\bal\ \Upgamma_1(p^r)]^{\zeta\can}; [\Gamma_{\mu}(N;\ell)])$ 
	on $(\Ell/R)$
	assigns to $E/S$ the set of quintuples $(\phi:E\rightarrow E',P,Q ; \alpha; C)	$
	where:
	\begin{enumerate}
		\item The data $(\phi:E\rightarrow E',P,Q ; \alpha)$
		is an element of $\scrP_r^{\zeta}(E/S)$.
		\item $C\subseteq E[\ell]$ is a locally-free, cyclic, subgroup-scheme of rank $\ell$.
		\item $\im(\alpha) \cap C = 0$.
	\end{enumerate}	
	When $R=\F_p$, we similarly define the moduli problem 
	$\prescript{}{\ell}{\I_r}:=([\Ig(p^r)];[\Gamma_{\mu}(N,\ell)])$ 
	on $(\Ell/\F_p)$ as the functor assigning to $E/S$ the set of quadruples $(E,P;\alpha;C)$
	where 
	\begin{enumerate}
		\item The triple $(E,R;\alpha)$ is an element of $\I_r(E/S)$.
		\item $C\subseteq E[\ell]$ is a locally-free, cyclic, subgroup-scheme of rank $\ell$.
		\item $\im(\alpha) \cap C = 0$
	\end{enumerate}
\end{definition}

Using the fact that the moduli problem $[\mu_N]$ is representable (recall we assume $N\ge 4)$, it is not difficult to 
prove that $[\Gamma_{\mu}(N;\ell)]$ is representable as well.  Indeed, if $\ell\nmid N$,
then $[\Gamma_{\mu}(N;\ell)]$ is just the simultaneous moduli problem $([\mu_N],[\Gamma_0(\ell)])$,
so representability follows immediately from 3.7.1 and 4.3.4 of \cite{KM}.  When $\ell|N$
we instead proceed as follows.  Writing $Y_{\mu}(N)$ for the representing object of $[\mu_N]$
and $(E^{\univ},\alpha^{\univ})/Y_{\mu}(N)$ for the universal elliptic curve with universal $\mu_N$-structure
$\alpha^{\univ}$, the scheme $E^{\univ}[\ell]$ is finite \'etale over $Y_{\mu}(N)$ since $\ell$
is a unit on the base, and we can form the open subscheme $Z\subseteq E^{\univ}[\ell]$
that is the complement of the closed subscheme $\im(\alpha^{\univ})$.  It is easy to see
that, for any $R$-scheme $S$, the $S$-points of $Z$ parameterize isomorphism classes of 
triples $(E,\alpha,Q)$ where $\alpha:\mu_N\hookrightarrow E[N]$ is a closed immersion
of $S$-group schemes and $Q$ is a point of $E[\ell]$ that generates a cyclic subgroup of order $\ell$
which is disjoint from $\im(\alpha).$
There is a natural free action of $(\Z/\ell\Z)^{\times}$ on $Z$ given by $a:(E,\alpha,Q)\mapsto (E,\alpha,aQ)$,
and we write $Z'$ for the quotient of $Z$ by this action.  Then the $S$-points of $Z'$
are identified 
with isomorphism classes of triples $(E,\alpha,C)$, with $E$ and $\alpha$ as before, 
and $C$ a loclly free, rank $\ell$, cyclic subgroup scheme of $E[\ell]$
with $C\cap \im(\alpha)=0$. 
Now arguing as in the proof of Proposition \ref{XrRepresentability}, we conclude
that $\prescript{}{\ell}{\scrP_{r}^\zeta}$ (respectively $\prescript{}{\ell}{\I_r}$) is represented
by a regular (respectively smooth) scheme that is flat of pure relative dimension 1 over $\Spec(R)$.
Writing $\X_r(\ell)$ (respectively $\Ig_r(\ell)$)
for the corresponding compactified moduli schemes, there are canonical degeneracy mappings
(keeping in mind Corollary \ref{MorExtCor}) 
$\pi_1^{(\ell)}, \pi_2^{(\ell)}:\X_r(\ell)\rightrightarrows \X_r$ given moduli-theoretically
by
\begin{align*}
	\pi_1^{(\ell)}(\phi:E\rightarrow E',P,Q ; \alpha; C)&:=
	(\phi:E\rightarrow E',P,Q ; \alpha)\\
		\pi_2^{(\ell)}(\phi:E\rightarrow E',P,Q ; \alpha; C)&:=
	(\o{\phi}:E/C\rightarrow E'/\phi(C),\psi(P),\psi'(Q) ; \psi\circ\alpha),
\end{align*}
where $\psi:E\rightarrow E/C$ and $\psi':E'\rightarrow E'/\phi(C)$
are the canonical isogenies and $\o{\phi}$ is the isogeny induced by $\phi$.
One similarly defines degeneracy mappings $\pi_1^{(\ell)}, \pi_2^{(\ell)}:\Ig_r(\ell)\rightrightarrows \Ig_r$.
Writing $T_{\ell}$ (respectively $U_{\ell}$) when $\ell\nmid N$ (respectively $\ell\mid N$)
for the resulting correspondence
$(\pi_1^{(\ell)}, \pi_2^{(\ell)})$ on $\X_r$
or $\Ig_r$, it is straightforward to check that $T_{\ell}$ and $U_{\ell}$ on $\X_r$ induce
$T_{\ell}$ and $U_{\ell}$, respectively, on the irreducible components 
$I_r^{\infty}=I_{(r,0,1)}$ and $I_r^{0}=I_{(0,r,1)}$
via Proposition \ref{redXr}.
The correspondence $U_p:=(\pi_1,\pi_2)$ (respectively $U_p^*:=(\pi_2,\pi_1)$)
on $\X_r$ is defined using the maps (\ref{Upcorr}), and induces the correspondence
$(F,\id)$ on $I_r^{\infty}$ (respectively $(F,\langle p\rangle_N)$ on $I_r^0$).

\bibliographystyle{amsalpha_noMR}
\bibliography{mybib}

\def\cprime{$'$}
\providecommand{\bysame}{\leavevmode\hbox to3em{\hrulefill}\thinspace}
\providecommand{\MR}{\relax\ifhmode\unskip\space\fi MR }
\providecommand{\MRhref}[2]{%
  \href{http://www.ams.org/mathscinet-getitem?mr=#1}{#2}
}
\providecommand{\href}[2]{#2}
\begin{thebibliography}{Hid86b}

\bibitem[AEZ78]{elzeinapp}
B.~Angeniol and F.~El~Zein, \emph{Appendice: ``{L}a classe fondamentale
  relative d'un cycle''}, Bull. Soc. Math. France M\'em. \textbf{58} (1978),
  67--93.

\bibitem[Cai09]{CaisDualizing}
Bryden Cais, \emph{Canonical integral structures on the de {R}ham cohomology of
  curves}, Ann. Inst. Fourier (Grenoble) \textbf{59} (2009), no.~6, 2255--2300.

\bibitem[Cai10]{CaisNeron}
\bysame, \emph{Canonical extensions of {N}\'eron models of {J}acobians},
  Algebra Number Theory \textbf{4} (2010), no.~2, 111--150.

\bibitem[Cai14]{CaisHida2}
\bysame, \emph{The geometry of {H}ida families \rmnum{2}: $\lambda$-adic
  $(\varphi,\gamma)$-modules and $\lambda$-adic {H}odge theory}, Submitted
  (2014).

\bibitem[Car57]{CartierNouvelle}
Pierre Cartier, \emph{Une nouvelle op{\'e}ration sur les formes
  diff{\'e}rentielles}, C. R. Acad. Sci. Paris \textbf{244} (1957), 426--428.

\bibitem[Car86]{CarayolReps}
Henri Carayol, \emph{Sur les repr\'esentations {$l$}-adiques associ\'ees aux
  formes modulaires de {H}ilbert}, Ann. Sci. \'Ecole Norm. Sup. (4) \textbf{19}
  (1986), no.~3, 409--468.

\bibitem[Con00]{GDBC}
Brian Conrad, \emph{Grothendieck duality and base change}, Lecture Notes in
  Mathematics, vol. 1750, Springer-Verlag, Berlin, 2000.

\bibitem[Dee01]{Dee}
Jonathan Dee, \emph{{$\Phi$}-{$\Gamma$} modules for families of {G}alois
  representations}, J. Algebra \textbf{235} (2001), no.~2, 636--664.

\bibitem[Del71a]{DeligneFormes}
P.~Deligne, \emph{Formes modulaires et representations e-adiques},
  S{\'e}minaire Bourbaki vol. 1968/69 Expos{\'e}s 347-363 (1971), 139--172.

\bibitem[Del71b]{DeligneHodge2}
Pierre Deligne, \emph{Th\'eorie de {H}odge. {II}}, Inst. Hautes \'Etudes Sci.
  Publ. Math. (1971), no.~40, 5--57.


\bibitem[DI87]{DeligneIllusie}
Pierre Deligne and Luc Illusie, \emph{Rel\`evements modulo {$p\sp 2$} et
  d\'ecomposition du complexe de de {R}ham}, Invent. Math. \textbf{89} (1987),
  no.~2, 247--270.

\bibitem[DS74]{DeligneSerre}
Pierre Deligne and Jean-Pierre Serre, \emph{Formes modulaires de poids {$1$}},
  Ann. Sci. \'Ecole Norm. Sup. (4) \textbf{7} (1974), 507--530 (1975).

\bibitem[Edi06]{EdixhovenComparison}
Bas Edixhoven, \emph{Comparison of integral structures on spaces of modular
  forms of weight two, and computation of spaces of forms mod 2 of weight one},
  J. Inst. Math. Jussieu \textbf{5} (2006), no.~1, 1--34, With appendix A (in
  French) by J.~F.~Mestre and appendix B by Gabor Wiese.

\bibitem[EGA]{EGA}
J.~Dieudonn{\'e} and A.~Grothendieck, \emph{\'{E}l\'ements de g\'eom\'etrie
  alg\'ebrique}, Inst. Hautes \'Etudes Sci. Publ. Math. (1960--7),
  no.~4,8,11,17,20,24,28,37.



\bibitem[FK12]{FukayaKato}
T.~Fukaya and K.~Kato, \emph{On conjectures of {S}harifi}, Preprint (2012).

\bibitem[FM87]{FontaineMessing}
Jean-Marc Fontaine and William Messing, \emph{{$p$}-adic periods and {$p$}-adic
  {\'e}tale cohomology}, Current trends in arithmetical algebraic geometry
  ({A}rcata, {C}alif., 1985), Contemp. Math., vol.~67, Amer. Math. Soc.,
  Providence, RI, 1987, pp.~179--207.

\bibitem[Gro90]{tameness}
Benedict~H. Gross, \emph{A tameness criterion for {G}alois representations
  associated to modular forms (mod {$p$})}, Duke Math. J. \textbf{61} (1990),
  no.~2, 445--517.

\bibitem[Har66]{RD}
Robin Hartshorne, \emph{Residues and duality}, Lecture notes of a seminar on
  the work of A. Grothendieck, given at Harvard 1963/64. With an appendix by P.
  Deligne. Lecture Notes in Mathematics, No. 20, Springer-Verlag, Berlin, 1966.

\bibitem[Hid86a]{HidaGalois}
Haruzo Hida, \emph{Galois representations into {${\rm GL}_2({\bf Z}_p[[X]])$}
  attached to ordinary cusp forms}, Invent. Math. \textbf{85} (1986), no.~3,
  545--613.

\bibitem[Hid86b]{HidaIwasawa}
\bysame, \emph{Iwasawa modules attached to congruences of cusp forms}, Ann.
  Sci. {\'E}cole Norm. Sup. (4) \textbf{19} (1986), no.~2, 231--273.

\bibitem[Kit94]{Kitagawa}
Koji Kitagawa, \emph{On standard {$p$}-adic {$L$}-functions of families of
  elliptic cusp forms}, {$p$}-adic monodromy and the {B}irch and
  {S}winnerton-{D}yer conjecture ({B}oston, {MA}, 1991), Contemp. Math., vol.
  165, Amer. Math. Soc., Providence, RI, 1994, pp.~81--110.

\bibitem[KM85]{KM}
Nicholas~M. Katz and Barry Mazur, \emph{Arithmetic moduli of elliptic curves},
  Annals of Mathematics Studies, vol. 108, Princeton University Press,
  Princeton, NJ, 1985.

\bibitem[Laz75]{LazardGroups}
Michel Lazard, \emph{Commutative formal groups}, Lecture Notes in Mathematics,
  Vol. 443, Springer-Verlag, Berlin, 1975.

\bibitem[Liu02]{LiuBook}
Qing Liu, \emph{Algebraic geometry and arithmetic curves}, Oxford Graduate
  Texts in Mathematics, vol.~6, Oxford University Press, Oxford, 2002,
  Translated from the French by Reinie Ern\'e, Oxford Science Publications.

\bibitem[Mat89]{matsumura}
Hideyuki Matsumura, \emph{Commutative ring theory}, second ed., Cambridge
  Studies in Advanced Mathematics, vol.~8, Cambridge University Press,
  Cambridge, 1989, Translated from the Japanese by M. Reid.

\bibitem[MW83]{MW-Analogies}
B.~Mazur and A.~Wiles, \emph{Analogies between function fields and number
  fields}, Amer. J. Math. \textbf{105} (1983), no.~2, 507--521.

\bibitem[MW84]{MW-Iwasawa}
\bysame, \emph{Class fields of abelian extensions of {${\bf Q}$}}, Invent.
  Math. \textbf{76} (1984), no.~2, 179--330.

\bibitem[MW86]{MW-Hida}
\bysame, \emph{On {$p$}-adic analytic families of {G}alois representations},
  Compositio Math. \textbf{59} (1986), no.~2, 231--264.

\bibitem[Nak85]{Nakajima}
Sh{\=o}ichi Nakajima, \emph{Equivariant form of the {D}euring-\v {S}afarevi\v c
  formula for {H}asse-{W}itt invariants}, Math. Z. \textbf{190} (1985), no.~4,
  559--566.

\bibitem[Oda69]{Oda}
Tadao Oda, \emph{The first de {R}ham cohomology group and {D}ieudonn\'e
  modules}, Ann. Sci. \'Ecole Norm. Sup. (4) \textbf{2} (1969), 63--135.

\bibitem[Oht95]{OhtaEichler}
Masami Ohta, \emph{On the {$p$}-adic {E}ichler-{S}himura isomorphism for
  {$\Lambda$}-adic cusp forms}, J. Reine Angew. Math. \textbf{463} (1995),
  49--98.

\bibitem[Oht00]{Ohta2}
\bysame, \emph{Ordinary {$p$}-adic \'etale cohomology groups attached to towers
  of elliptic modular curves. {II}}, Math. Ann. \textbf{318} (2000), no.~3,
  557--583.

\bibitem[Ray74]{Raynaud}
Michel Raynaud, \emph{G\'eom\'etrie analytique rigide d'apr\`es {T}ate,
  {K}iehl,{$\cdots $}}, Table Ronde d'Analyse non archim\'edienne (Paris,
  1972), Soc. Math. France, Paris, 1974, pp.~319--327. Bull. Soc. Math. France,
  M\'em. No. 39--40.

\bibitem[Ros58]{Rosenlicht}
Maxwell Rosenlicht, \emph{Extensions of vector groups by abelian varieties},
  Amer. J. Math. \textbf{80} (1958), 685--714.

\bibitem[Sab96]{Saby}
Nicolas Saby, \emph{Th\'eorie d'{I}wasawa g\'eom\'etrique: un th\'eor\`eme de
  comparaison}, J. Number Theory \textbf{59} (1996), no.~2, 225--247.

\bibitem[Sen81]{DSen}
Shankar Sen, \emph{Continuous cohomology and {$p$}-adic {G}alois
  representations}, Invent. Math. \textbf{62} (1980/81), no.~1, 89--116.

\bibitem[Ser58]{SerreTopology}
Jean-Pierre Serre, \emph{Sur la topologie des vari{\'e}t{\'e}s alg{\'e}briques
  en caract{\'e}ristique {$p$}}, Symposium internacional de topolog\'\i a
  algebraica {I}nternational symposium on algebraic topology, Universidad
  Nacional Aut{\'o}noma de M{\'e}xico and UNESCO, Mexico City, 1958,
  pp.~24--53.

\bibitem[SGA6]{SGA6}
\emph{Th{\'e}orie des intersections et th{\'e}or{\`e}me de {R}iemann-{R}och},
  Lecture Notes in Mathematics, Vol. 225, Springer-Verlag, Berlin, 1971,
  S{{\'e}}minaire de G{{\'e}}om{{\'e}}trie Alg{{\'e}}brique du Bois-Marie
  1966--1967 (SGA 6), Dirig{{\'e}} par P. Berthelot, A. Grothendieck et L.
  Illusie. Avec la collaboration de D. Ferrand, J. P. Jouanolou, O. Jussila, S.
  Kleiman, M. Raynaud et J. P. Serre.

\bibitem[Sha07]{SharifiEisenstein}
R.T. Sharifi, \emph{Iwasawa theory and the eisenstein ideal}, Duke Mathematical
  Journal \textbf{137} (2007), no.~1, 63--101.

\bibitem[Sha11]{SharifiConj}
Romyar Sharifi, \emph{A reciprocity map and the two-variable {$p$}-adic
  {$L$}-function}, Ann. of Math. (2) \textbf{173} (2011), no.~1, 251--300.

\bibitem[Tat67]{Tate}
J.~T. Tate, \emph{{$p$-divisible} {groups.}}, Proc. {C}onf. {L}ocal {F}ields
  ({D}riebergen, 1966), Springer, Berlin, 1967, pp.~158--183.

\bibitem[Tat68]{TateResidues}
John Tate, \emph{Residues of differentials on curves}, Ann. Sci. \'Ecole Norm.
  Sup. (4) \textbf{1} (1968), 149--159.

\bibitem[Til87]{Tilouine}
Jacques Tilouine, \emph{Un sous-groupe {$p$}-divisible de la jacobienne de
  {$X_1(Np^r)$} comme module sur l'alg\`ebre de {H}ecke}, Bull. Soc. Math.
  France \textbf{115} (1987), no.~3, 329--360.

\bibitem[Ulm90]{Ulmer}
D.~L. Ulmer, \emph{On universal elliptic curves over {I}gusa curves}, Invent.
  Math. \textbf{99} (1990), no.~2, 377--391.

\bibitem[Wil88]{WilesLambda}
A.~Wiles, \emph{On ordinary {$\lambda$}-adic representations associated to
  modular forms}, Invent. Math. \textbf{94} (1988), no.~3, 529--573.

\bibitem[Wil90]{WilesTotallyReal}
\bysame, \emph{The {I}wasawa conjecture for totally real fields}, Ann. of Math.
  (2) \textbf{131} (1990), no.~3, 493--540.

\end{thebibliography}
\end{document}